\documentclass[11pt, reqno]{amsart}

\usepackage[version=4]{mhchem} 

\usepackage{upgreek}
\usepackage{graphicx, amsmath, amsthm,amsfonts, amssymb}
\usepackage{hyperref}
\usepackage{subfigure}
\usepackage{algorithm,algorithmic}

\usepackage{mathtools,bm}
\usepackage{multirow,booktabs,makecell}
\usepackage{tikz,url}
\usetikzlibrary{shapes,arrows}
\tikzstyle{process} = [rectangle,minimum width=2cm,minimum height=1cm,text centered,text width =4cm,draw=black]

\newcommand{\bmb}{\bm{\beta}}
\newcommand{\bmr}{\bm{\rho}}
\newcommand{\bms}{\bm{s}}

\newcommand{\bmm}{\bm{m}}

\newcommand{\bma}{\bm{\alpha}}

\newcommand{\m}{\bmm}
\newcommand{\s}{s}
\newcommand{\bu}{\mathbf{u}}

\newtheorem{theorem}{Theorem}[section]
\newtheorem{proposition}{Proposition}[section]

 \theoremstyle{definition}
 \newtheorem{definition}[theorem]{Definition}
 \newtheorem{example}[theorem]{Example}

\newtheorem{remark}[theorem]{Remark}

\numberwithin{equation}{section}

\newcommand{\nn}{\bm{n}}
\newcommand{\uu}{\bm{u}}
\newcommand{\uui}{\bm{\mathrm{u}}}
\newcommand{\rti}{\uprho_{T,i}}
\newcommand{\mm}{\bm{m}}
\newcommand{\so}{s}
\newcommand{\rrho}{\rho}
\newcommand{\nnu}{\bm{\nu}}
\newcommand{\pphi}{\phi}
\newcommand{\Pphi}{\Phi}
\newcommand{\bnn}{\underline{\bm{n}}}
\newcommand{\buu}{\underline{\bm{u}}}

\newcommand{\cmm}{\underline{\bm{m}}}
\newcommand{\bso}{\bm{s}}
\newcommand{\brrho}{\bm{\rho}}
\newcommand{\sop}{{s}_{p}}
\newcommand{\bpphi}{\bm{\phi}}
\newcommand{\bppsi}{\bm{\psi}}
\newcommand{\bPphi}{\bm{\Phi}}

\newcommand{\jj}{{j}}
\newcommand{\ii}{{i}}
\newcommand{\FF}{F}
\newcommand{\GG}{G}
\newcommand{\HH}{H}
\newcommand{\bFF}{\bm F}
\newcommand{\bGG}{\bm G}
\newcommand{\bHH}{\bm H}
\newcommand{\ssigma}{\bm{\sigma}}
\newcommand{\bssigma}{\underline{\bm{\sigma}}}
\newcommand{\bttau}{\underline{\bm{\tau}}}
\newcommand{\ttau}{\bm{\tau}}
\newcommand{\stint}[1]{\langle\!\langle #1 \rangle\!\rangle_h}
\newcommand{\sint}[1]{\langle #1 \rangle_h}

\begin{document}
\title[MFC with FEM]{Generalized optimal transport and mean field control problems for reaction-diffusion systems with high-order finite element computation}
\author[Fu]{Guosheng Fu}
\address{Department of Applied and Computational Mathematics and
Statistics, University of Notre Dame, Notre Dame, IN, USA.}
\email{gfu@nd.edu}
 \thanks{G. Fu's work is supported by NSF grant DMS-2012031.}

\author[Osher]{Stanley Osher}
\address{Department of Mathematics, University of California, Los Angeles,
Los Angeles, CA, USA.}
\email{sjo@math.ucla.edu}
\thanks{S. Osher's work is supported in part by AFOSR MURI FP 9550-18-1-502 and ONR grants:  N00014-20-1-2093 and N00014-20-1-2787.}

\author[Pazner]{Will Pazner}
\address{Fariborz Maseeh Department of Mathematics and Statistics, Portland State University, Portland, OR, USA.}
\email{pazner@pdx.edu}

\author[Li]{Wuchen Li}
\address{Department of Mathematics, University of South Carolina, Columbia, SC, USA.}
\email{wuchen@mailbox.sc.edu}
\thanks{W. Li's work is supported by AFOSR MURI FP 9550-18-1-502, AFOSR YIP award No. FA9550-23-1-0087, NSF DMS-2245097, and NSF RTG: 2038080.}




 \keywords{Optimal transport; Multi-population mean-field control problems; Generalized Fisher information functional; Reaction-diffusion systems; Finite element methods; ALG2 algorithms.}
\subjclass{35K57, 49N80, 49M41, 65N30.}
\begin{abstract}
We design and compute a class of optimal control problems for reaction-diffusion systems. They form mean field control problems related to multi-density reaction-diffusion systems. To solve proposed optimal control problems numerically, we first apply high-order finite element methods to discretize the space-time domain and then solve the optimal control problem using augmented Lagrangian methods (ALG2). Numerical examples, including generalized optimal transport and mean field control problems between Gaussian distributions and image densities, demonstrate the effectiveness of the proposed modeling and computational methods for mean field control problems involving reaction-diffusion equations/systems. 
\end{abstract}\maketitle

\section{Introduction}
\label{sec:intro}
 Reaction-diffusion systems are essential classes of modeling dynamics \cite{OnsagerMachlup1953_fluctuations}, which have applications in tumor growth modeling \cite{ChenBenvenisteTannenbaum2023_unbalanced}, propagation of pandemic spread \cite{LeeLiuLiOsher2022_mean,LeeLiuTembineLiOsher2021_controllinga}, evolutionary games \cite{GaoLiLiu}, etc. The reaction term reflects general nonlinear interacting behaviors of agents/particles in complex systems. Patterns of population behaviors often arise in the solution of reaction-diffusion equations, which model the collective behaviors of particles/agents.

In recent years, mean field control problems \cite{huang2006large,LasryLions2007} have been studied, which are optimal control problems for mean-field limits of infinitely many identical particles/agents. The problem models the complex behaviors of identical particles/agents interacting with each other. It is worth mentioning that a special example of a mean-field control problem forms the dynamical optimal transport problem \cite{benamou1999numerical}. It studies a particular optimal control problem in density space, which transfers an initial density toward the terminal density function. The optimal transport problem also introduces a functional distance, namely the Wasserstein distance. It helps study and compute a class of initial value evolutionary dynamics, namely, Wasserstein gradient flows \cite{JKO}. They are also valuable for modeling interaction behaviors of particle dynamics. 

Nowadays, optimal transport and mean field control problems have vast applications in modeling and computations. Classical studies of optimal transport and mean field control problems are often limited to a single-density function. The complex interaction between densities is a vital modeling factor that has not been systemically studied. 

In this paper, we generalize a class of mean field control of reaction-diffusion equations/systems proposed in \cite{LiLeeOsher22}. It forms an optimal control problem among multiple density functions interacting with each other with a nonlinear reaction vector function. After change of variables, we reformulate the optimal control problem into an optimization problem with a Fisher information type potential functional. We apply the high-order finite element method to discretize the spatial-time domain and use the augmented Lagrangian method, ALG2 from \cite{FortinBook}, to compute the discretized optimization problems. Numerical examples, including mean field control problems between Gaussian densities and a system with $12$ images, demonstrate the solution of the proposed generalized mean field control problems. 

Generalized optimal transport and mean field control problems have been widely investigated in  \cite{CardaliaguetCarlierNazaret2013_geodesics,CarrilloLisiniSavareSlepcev2010_nonlineara, DolbeaultNazaretSavare2009_new,MielkeRengerPeletier2016_generalizationa}. 
For example, multi-population mean field games were discussed in \cite{BensoussanHuangLauriere2018_mean}, and generalized optimal transport distances between vector densities were formulated in \cite{ChenGeorgiouTannenbaum2018_vectorvalued}. Meanwhile,  unnormalized and unbalanced optimal transport were proposed in \cite{Chizat18,lee2020generalized,LieroMielkeSavare2018_optimala}, which allows to control densities with different total masses. In applications, one also applies mean field control problems to model the propagation of pandemics \cite{LeeLiuLiOsher2022_mean,LeeLiuTembineLiOsher2021_controllinga}. 
They are all important examples of optimal control problems for reaction-diffusion systems. Modeling and computational mean-field control problems for general convection-reaction-diffusion systems are new research directions \cite{ li2022controlling,li2021controlling}. They have potential applications in classical modeling dynamics based on reaction-diffusion equations and systems. This paper proposes a class of mean field control of reaction-diffusion systems. Meanwhile, the optimal control problem of gradient flows in Wasserstein space is widely studied, namely Schr{\"o}dinger bridge problems \cite{ChenGeorgiouPavon2016_relationd, Conforti2019_seconda,LegerLi2021_hopfcole,LiYinOsher2018_computationsc,MonsaingeonTamaniniVorotnikov2020_dynamical}. In addition, the generalized gradient flows in multiple-density spaces have been studied \cite{FuOsherLi,GallouetMonsaingeon2017_jko,Mielke11,MielkeRengerPeletier2016_generalizationa}. Our formulation extends the optimal control problem of gradient flows based on reaction-diffusion equations and systems. Thus, we study the mean field control using formalisms of controlling gradient flows. In simulations, we apply high-order finite element schemes to simulate the proposed generalized optimal transport and mean field control problems.

The paper is organized as follows. In section \ref{sec2}, we first review gradient flows and their induced metric distances in generalized optimal transport spaces. We then formulate and derive optimality conditions for generalized optimal transport and mean field control problems of reaction-diffusion systems. In section \ref{sec3}, we approximate the proposed mean field control problems using high-order finite element schemes and then use the ALG2 method to compute the discretized optimization problems. Numerical results, including scalar reaction-diffusion equations and two-species, twelve-species reaction-diffusion systems, are presented in section \ref{sec4}. Finally, we conclude with a discussion in section \ref{sec5}. 

\section{Generalized optimal transport and mean field control of reaction-diffusion systems}\label{sec2}
This section presents the main formulation of mean field control (MFC) problems for scalar reaction-diffusion equations and systems. 
 \subsection{Background: Reaction-diffusion induced  metric distances}
 Before delving into the mean field control problems that will be discussed in this manuscript, we first review the definitions of two metric distances: one for scalars and one for systems, which are obtained from generalized optimal transport type gradient flow problems. The material in this subsection follows closely our previous work on variational time implicit schemes for reaction-diffusion systems in \cite{FuOsherLi}.
 We only give the definition of these metric distances without further elaboration, but refer to the references \cite{2005_gradienta,LegerLi2021_hopfcole, LiLeeOsher22,Mielke11} for more details on 
 optimal transport type gradient flows, distances, mean-field control and related problems.
 
 \subsubsection{The metric distance: scalar case}
 The scalar metric distance is derived from the following reaction-diffusion equation \cite{LiLeeOsher22,FuOsherLi}:
\begin{equation}\label{dissipative}
\begin{split}
&\partial_t \rho=\nabla\cdot(V_1(\rho)\nabla\frac{\delta}{\delta \rho}\mathcal{E}(\rho))-V_2(\rho)\frac{\delta}{\delta\rho}\mathcal{E}(\rho), \text{ on }[0,T]\times\Omega,
\end{split}
\end{equation}
with homogeneous Neumann boundary conditions 
$V_1(\rho)\nabla\frac{\delta}{\delta \rho}\mathcal{E}(\rho)\cdot\nnu|_{\partial \Omega} = 0$, where $\nnu$ is the unit outward normal direction on the boundary $\partial\Omega$. Here $\Omega\subset\mathbb{R}^d$ is the spatial domain, $\rho\colon [0,T]\times\Omega \rightarrow \mathbb{R}$ is a scalar non-negative density function satisfying
\begin{equation}
\label{d-space0}
   \rho(t,\cdot) \in  \mathcal{M}=\{\rho\in H^1(\Omega): \rho\geq 0\},\quad \forall t\geq 0,
   \end{equation} 
 $\mathcal{E}\colon \mathcal{M}\rightarrow\mathbb{R}$ is an energy functional, $V_1, V_2\colon \mathbb{R}_+\rightarrow\mathbb{R}_+$ are two positive mobility functions, and $\frac{\delta}{\delta\rho}$ is the  first variation operator in $L^2$ space.

A crucial property of the equation \eqref{dissipative} is that it
satisfies an energy-dissipation law:
\begin{equation}
\label{ener-law}
\frac{d}{dt}\mathcal{E}(\rho)=-\int_\Omega\Big[\|\nabla\tfrac{\delta}{\delta \rho}\mathcal{E}(\rho)\|^2V_1(\rho) + |\tfrac{\delta}{\delta \rho}\mathcal{E}(\rho)|^2V_2(\rho)\Big]dx\leq  0,
\end{equation}
where we use the fact that functions $V_1$, $V_2$ are non-negative. 
In the literature, the above right hand side (dissipation rate)
\begin{align}
\label{fisher}
\mathcal{I}(\rho)=
\int_\Omega\Big[\|\nabla\tfrac{\delta}{\delta \rho}\mathcal{E}(\rho)\|^2V_1(\rho) + |\tfrac{\delta}{\delta \rho}\mathcal{E}(\rho)|^2V_2(\rho)\Big]dx,
\end{align}
is called the generalized {\it Fisher information functional}. 
This Fisher information functional induces a metric in the space $\mathcal{M}$, which further defines a distance between two densities $\rho^0, \rho^1\in \mathcal{M}$.
\begin{definition}[Scalar distance functional]
\label{Dis}
Define a distance functional \[\mathrm{Dist}_{V_1,V_2}\colon \mathcal{M}\times \mathcal{M}\rightarrow\mathbb{R}_+\] as below. Consider the following optimal control problem: 
\begin{subequations}\label{DisZ}    \begin{equation}\label{Dis1}
\mathrm{Dist}_{V_1, V_2}(\rho^0, \rho^1)^2:=\inf_{\rho, \bm v_1, v_2}\quad\int_0^T\int_\Omega \Big[\|\bm v_1\|^2 V_1(\rho)+ |v_2|^2V_2(\rho)\Big]dxdt, 
\end{equation}
where the infimum is taken among $\rho(t,x)\colon [0,T]\times\Omega\rightarrow\mathbb{R}_+$,  $\bm v_1(t,x)\colon  [0,T]\times \Omega \rightarrow\mathbb{R}^d$, 
$v_2(t,x)\colon  [0,T]\times \Omega \rightarrow\mathbb{R}$, 
such that $\rho$ satisfies a reaction-diffusion type equation with drift vector field $\bm v_1$, drift mobility $V_1$, reaction rate $v_2$, reaction mobility $V_2$, connecting initial and terminal densities $\rho^0$, $\rho^1\in\mathcal{M}$: 
\begin{equation}\label{Dis2}
\left\{\begin{aligned}
&\partial_t \rho + \nabla\cdot( V_1(\rho) \bm v_1)=V_2(\rho)v_2,\quad \text{on } [0,T]\times \Omega,\\
&\rho(0, x)=\rho^0(x),\quad \rho(T,x)=\rho^1(x),
\end{aligned}\right.
\end{equation}
with no-flux boundary condition $V_1(\rho)\bm v_1\cdot\nnu|_{\partial\Omega}=0$
\end{subequations}
\end{definition}

Introducing the flux function $\bmm(t,x)\colon [0,T]\times\Omega\rightarrow\mathbb{R}^d$ and source function $s(t, x)\colon [0,T]\times\Omega\rightarrow\mathbb{R}$, such that 
\begin{equation*}    \bmm(t,x)=V_1(\rho(t,x))\bm v_1(t,x), \quad s(t,x)=V_2(\rho(t,x))v_2(t,x), 
\end{equation*}
the distance in Definition \ref{Dis} is rewritten as
the following  optimization problem with a linear constraint:
\begin{subequations}
\label{DisM}
\begin{equation}
\label{DisM1}
    \mathrm{Dist}_{V_1, V_2}(\rho^0, \rho^1)^2
    :=\inf_{\rho, \bmm, s}\quad\int_0^T\int_\Omega \Big[\frac{\|\bmm\|^2} {V_1(\rho)}+ \frac{|s|^2}{V_2(\rho)}\Big]dxdt, 
\end{equation}
such that
\begin{equation}
\label{DisM2}
\begin{split}
&\partial_t \rho(t,x)+ \nabla\cdot \bmm(t,x) =s(t,x), \quad \text{in }[0,T]\times\Omega,\\
&
\bmm\cdot\nnu = 0, \quad\text{ on } [0,T]\times\partial\Omega,\\
&
\rho(0,x)=\rho^0(x),\; \rho(T,x)=\rho^1(x), \quad
\text{in }\Omega.
\end{split}
\end{equation}
\end{subequations}
The optimization functional in \eqref{DisM1}
is convex under the condition that both 
$V_1$ and $V_2$ are positive concave functions.
 \subsubsection{The metric distance: system case}
 The system  metric distance is derived from the following reaction-diffusion system
 with 
$M$ species and $R$ reactions \cite{FuOsherLi,Mielke11}:
\begin{equation}\label{dissipativeS}
\begin{split}
&
\partial_t\rho_i =\;\nabla\cdot\left(
V_{1,i}(\rho_i)\nabla\frac{\delta}{\delta \rho}\mathcal{E}_i(\rho_i)\right) \; -\sum_{p=1}^RV_{2,p}(\bmr)\gamma_{i,p}\sum_{j=1}^M\gamma_{j,p}
\frac{\delta}{\delta \rho}\mathcal{E}_j(\rho_j),
\end{split}
\end{equation}
with homogeneous Neumann boundary conditions 
$V_{1,i}(\rho_i)\nabla\frac{\delta}{\delta \rho}\mathcal{E}_i(\rho_i)\cdot\nnu|_{\partial\Omega}=0$
for $1\le i\le M$.
Here $\rho_i$ is the density such that 
$\rho_i(t,\cdot)\in\mathcal{M}$, 
$\mathcal{E}_i\colon \mathcal{M}\rightarrow\mathbb{R}$ is the energy functional 
and $V_{1,i}\colon \mathbb{R}_+\rightarrow\mathbb{R}_+$ is the positive mobility function
for $i$-th species, and $V_{2,p}\colon \mathbb{R}^M_+\rightarrow\mathbb{R}_+$ is the positive mobility function for $p$-th reaction.
The bolded density $\bm\rho=(\rho_1,\cdots, \rho_M)$ is simply the collection of $M$ densities.
Moreover, the coefficient matrix  $\Gamma=(\gamma_{i,p})\in\mathbb{R}^{M\times R}$ satisfies 
$\sum_{i=1}^M\gamma_{i,p}=0$ for all $1\le p\le R$, which is related to the total mass conservation:
$\frac{d}{dt}\int_{\Omega}\sum_{i=1}^M\rho_i\,dx =0.
$

The equation \eqref{dissipativeS}
satisfies the energy-dissipation law:
\begin{equation}
\label{ener-lawS}
\frac{d}{dt}\sum_{i=1}^M\mathcal{E}_i(\rho_i)=-
\bm{\mathcal{I}}(\bm\rho)
\le 0,
\end{equation}
where $\bm{\mathcal{I}}(\bm\rho)$ denotes the generalized Fisher information functional
\begin{equation}
\label{fisherS}
\bm{\mathcal{I}}(\bm\rho):=
\int_\Omega\Big[\sum_{i=1}^M\|\nabla\frac{\delta\mathcal{E}_i}{\delta \rho}\|^2V_{1,i}(\rho_i) + 
\sum_{p=1}^R\left|\sum_{j=1}^M\gamma_{j,p}\frac{\delta\mathcal{E}_j}{\delta \rho}\right|^2V_{2,p}(\bm\rho)\Big]dx.
\end{equation}
In the above, we use the fact that functions $V_{1,i}$, $V_{2,p}$, $1\leq i\leq M$, $1\leq p\leq R$, are non-negative. The functional $\bm{\mathcal{I}}$, again, induces a metric function in space $\mathcal{M}^M$, which defines distances between two system densities $\bm\rho^0, \bm\rho^1\in \mathcal{M}^M$.\\ 
\begin{definition}[System distance functional]
Define a distance functional 
\begin{align*}
\bm{\mathrm{Dist}}_{\bm V_1,\bm V_2}\colon \mathcal{M}^M\times \mathcal{M}^M\rightarrow\mathbb{R}_+
\end{align*}
as below. 
\begin{equation}
\label{DisS}
\begin{aligned}
\mathrm{\mathbf{Dist}}_{V_1, V_2}(\bmr^0, \bmr^1)^2:=& \inf_{\bmr, \cmm, \bms} \Big\{\int_0^T\int_\Omega 
\left(\sum_{i=1}^M\frac{|\bmm_i|^2}{V_{1,i}(\rho_i)}+
\sum_{p=1}^R\frac{|s_p|^2}{V_{2,p}(\bmr)}\right)
 dx{dt}:\;\;\\
&\hspace{1cm}\begin{tabular}{l}
$\partial_t\rho_i+\nabla\cdot \bmm_i=
\sum_{p=1}^R\gamma_{i,p}s_p,\forall 1\le i\le M$,\\
$\bmm_i\cdot\nnu|_{\partial\Omega}=0$,
$\bmr(0,\cdot)=\bmr^0, \;\;
\bmr(T,\cdot)=\bmr^1.$
\end{tabular}\Big\},
\end{aligned}
\end{equation}
where $\cmm=(\bmm_1,\cdots, \bmm_M)$ is the collection of fluxes, and 
$\bms=(s_1,\cdots, s_R)$ is the collection of sources.
\end{definition}
Again, the optimization functional in \eqref{DisS}
is convex under the condition that 
$V_{1,i}$ and $V_{2,p}$ are positive concave functions for all $1\le i\le M$ and $1\le p\le R$.

 \subsection{MFC for scalar reaction-diffusion}
In this subsection, we introduce the following MFC problem. This can be viewed as a generalization of the scalar distance \eqref{DisM}.
\begin{definition}[Scalar MFC problem]
\label{smfc}
Suppose that there exist two positive mobility functions $V_1,V_2:\mathbb{R}_+\rightarrow\mathbb{R}_+$, 
a terminal time $T>0$, a non-negative regularization parameter $\beta\ge0$, 
an energy functional $\mathcal{E}\colon \mathcal{M}\rightarrow\mathbb{R}$, 
a potential functional $\mathcal{F}\colon \mathcal{M}\rightarrow\mathbb{R}$, 
and a terminal functional 
$\mathcal{G}\colon \mathcal{M}\rightarrow\mathbb{R}$. 
\begin{subequations}\label{smfcA}
Consider 
\begin{equation}\label{smfcA1}
\begin{split}
&\inf_{\rho, \tilde{\bmm}, \tilde s}\quad\int_0^T\left[\int_\Omega \left(\frac{\|\tilde{\bmm}\|^2} {2V_1(\rho)}+ \frac{|\tilde s|^2}{2V_2(\rho)}\right)dx
-\mathcal{F}(\rho)\right]dt+\mathcal{G}(\rho(T,\cdot)), 
\end{split}
\end{equation}
where the infimum is among all densities
$\rho$ with $\rho(t,\cdot)\in\mathcal{M}$ for $t\in[0, T]$, flux $\tilde{\bmm}\colon [0, T]\times \Omega \rightarrow\mathbb{R}^d$, and source $\tilde s\colon [0, T]\times \Omega \rightarrow\mathbb{R}$, such that
\begin{equation}\label{smfcA2}
\partial_t \rho + \nabla\cdot \tilde{\bmm}-\tilde s=\beta
\left(\nabla\cdot(V_1(\rho)\nabla\frac{\delta}{\delta \rho}\mathcal{E}(\rho))-V_2(\rho)\frac{\delta}{\delta\rho}\mathcal{E}(\rho)\right),  
\end{equation}
with  boundary condition 
\begin{align}
\label{smfcA3}
\left.\left(\tilde{\bmm}-\beta V_1(\rho)\nabla\frac{\delta}{\delta \rho}\mathcal{E}(\rho)\right)\cdot\nnu\right|_{\partial\Omega} =0,
\end{align}
and fixed initial density $\rho(0,\cdot) = \rho^0$ in $\Omega$. 
\end{subequations}
\end{definition}
From the modeling perspective,  the variational problem in Definition \ref{smfc} models the movement of the density from the initial density under the dynamical constraint of reaction-diffusion equations. It aims to find the optimal choices of vector fields and reaction rate functions under general kinetic, potential, and terminal energies. 
 

We will construct numerical scheme for problem \eqref{smfc} based on the following change of variables.
\begin{proposition}[Scalar MFC reformulation]\label{smfc2}
Denote $\bmm\colon [0,T]\times\Omega\rightarrow\mathbb{R}^d$, 
 and $s\colon [0,T]\times\Omega\rightarrow\mathbb{R}$, such that 
\begin{equation*}
 \bmm=\tilde\bmm-\beta V_1(\rho)\nabla\frac{\delta}{\delta \rho}\mathcal{E}(\rho),\quad s=\tilde s-\beta V_2(\rho)\frac{\delta}{\delta \rho}\mathcal{E}(\rho), 
\end{equation*}
and denote
\begin{equation}
\label{v3-form}
   V_3(\rho)=\frac{1}{V_1(\rho)|\frac{\delta^2}{\delta \rho^2}\mathcal{E}(\rho)|^2}, 
\end{equation}
where $\frac{\delta^2}{\delta \rho^2}\mathcal{E}(\rho)$ is the second variational derivative of the energy functional $\mathcal{E}(\rho)$.
Then, the scalar MFC problem \eqref{smfc} is equivalent to the following optimization problem:
Consider 
\begin{subequations}\label{smfc2A}
\begin{equation}\label{smfc2A1}
\begin{split}
&\inf_{\rho,\bmm, s} ~~ \int_0^T\int_\Omega \Big[\frac{\|\bmm\|^2}{2 V_1(\rho)}+ \frac{|s|^2}{2V_2(\rho)}+\beta^2\frac{|\nabla \rho|^2}{2V_3(\rho)}+\frac{\beta^2}{2}{|\frac{\delta\mathcal{E}}{\delta \rho}|^2}{V_2(\rho)}\Big]dxdt\\
&\hspace{2.0cm}-\int_{0}^T\mathcal{F}(\rho)\,dt+\mathcal{G}(\rho(T,\cdot))+\beta[\mathcal{E}(\rho(T,\cdot))-
\mathcal{E}(\rho^0)], 
\end{split}
\end{equation}
where the infimum is among all densities $\rho$ with $\rho(t,\cdot)\in\mathcal{M}$, for $t\in [0, T]$, flux $\bmm\colon [0, T]\times \Omega \rightarrow\mathbb{R}^d$, and source $s\colon [0, T]\times \Omega \rightarrow\mathbb{R}$, such that
\begin{equation}\label{smfc2A2}
\partial_t \rho(t, x) + \nabla\cdot \bmm(t,x)-s(t,x)=0, 
\end{equation}
with no-flux boundary condition
$
 \bmm\cdot\nnu|_{\partial\Omega}=0, 
 $
\end{subequations}
and fixed initial density $\rho(0,\cdot)=\rho^0$. 
\end{proposition}
The proof of the above proposition follows from a similar argument in \cite[Proposition 2]{LiWang22}.  We present the detailed derivation in the Appendix. The main idea is to use integration by parts to show that 
\begin{equation*}
\begin{split}
&\int_0^T\int_\Omega \Big[\frac{\|\tilde{\bmm}\|^2} {2V_1(\rho)}+\frac{\|\tilde{s}\|^2} {2V_2(\rho)}\Big]dxdt\\
=&\;
\int_0^T\int_\Omega \Big[\frac{\|{\bmm}\|^2} {2V_1(\rho)}
+\frac{\|s\|^2} {2V_2(\rho)}
+\beta^2\frac{|\nabla \rho|^2}{2V_3(\rho)}
+\frac{\beta^2}{2}{|\frac{\delta\mathcal{E}}{\delta \rho}|^2}{V_2(\rho)}
\Big]dxdt\\
&\;+\beta[\mathcal{E}(\rho(T,\cdot))-
\mathcal{E}(\rho^0)].
\end{split}
\end{equation*}

For simplicity of discussions, we assume that there are two functions $F:[0,T]\times\Omega\times \mathbb{R}_+\rightarrow \mathbb{R}$
and $G:\Omega\times\mathbb{R}_+\rightarrow \mathbb{R}$, such that 
\begin{align*}
\int_{\Omega}F(t, x, \rho)dx = &\;
\mathcal{F}(\rho(t,\cdot))-\int_{\Omega}\frac{\beta^2}{2}{|\frac{\delta\mathcal{E}}{\delta \rho}(\rho(t,x))|^2}{V_2(\rho(t,x))}dx,\quad\forall \,t\geq 0,\\
\int_{\Omega}G(x,\rho)dx = &\;
\mathcal{G}(\rho)+\beta\mathcal{E}(\rho).
\end{align*}
Then the functional in \eqref{smfc2A1}
simplifies to the following formula:
\begin{align*}
\int_0^T\int_\Omega \Big[\frac{\|\bmm\|^2}{2 V_1(\rho)}+ \frac{|s|^2}{2V_2(\rho)}+\frac{\beta^2|\nabla \rho|^2}{2V_3(\rho)}-F(t,x,\rho)\Big]dxdt+
\int_{\Omega}
G(x, \rho(T,x))dx.
\end{align*}
This is the form we use in our numerical scheme.

For completeness, we formulate the Karush–Kuhn–Tucker (KKT) condition, i.e., the critical point system, for the optimization problem \eqref{smfc2A}. The proof is presented in the Appendix; see also \cite{ LiLeeOsher22,LiLuWang20}. 
\begin{proposition}[KKT system for \eqref{smfc2A}]\label{smfcKKT}
Let $(\rho, \bmm, s)$ be the critical point of the optimization problem \eqref{smfc2A}. Then there exists a function $\phi\colon [0, T]\times\Omega\rightarrow \mathbb{R}$, such that 
\begin{equation*}
\frac{\bmm(t,x)}{V_1(\rho(t,x))}=\nabla \phi(t,x),\quad \frac{s(t,x)}{V_2(\rho(t,x))}=\phi(t,x),
\end{equation*}
and 
\begin{equation}\label{SBPI}
\left\{\begin{aligned}
&\partial_t\rho(t,x) +\nabla\cdot(V_1(\rho(t,x))\nabla \phi(t,x))-V_2(\rho(t,x))\phi(t,x)=0,\\
&\partial_t\phi(t,x)+\frac{1}{2}\|\nabla \phi(t,x)\|^2 V_1'(\rho(t,x))+\frac{1}{2}|\phi(t,x)|^2V'_2(\rho(t,x))\\
&\hspace{6.3cm}+\frac{\delta}{\delta \rho}\Big[\mathcal{F}(\rho)-\frac{\beta^2}{2}\mathcal{I}(\rho)\Big]=0,
\end{aligned}\right.
\end{equation}
where 
$\mathcal{I}(\rho)$
is the generalized Fisher information functional given in \eqref{fisher}, with initial and terminal time boundary conditions
\begin{equation}
   \rho(0,x)=\rho^0(x),\quad \phi(T,x)=-\frac{\delta}{\delta \rho}\Big(\mathcal{G}(\rho(T,\cdot))+\beta \mathcal{E}(\rho(T,\cdot))\Big).  
\end{equation}
\end{proposition}

\subsubsection{Examples}
Different choices of the functions 
$V_1$, $V_2$, $V_3$, $F$, $G$, 
along with the regularization parameter $\beta\ge0$ and initial density $\rho^0$, 
lead to various scalar MFC problems \eqref{smfc2A}.
We present some examples of these functions in Table \ref{table:1} below. 
\begin{table}[ht!]
\centering
\begin{tabular}{l l}
\hline
$V_1(\rho)$: & $\alpha_1\rho^{\gamma_1}$ with $ 0\le \gamma_1\le 1$. $\alpha_1>0$. \\
$V_2(\rho)$: & $\alpha_2\rho^{\gamma_2}$ with $0\le \gamma_2\le 1$, or 
$\alpha_2\frac{\rho-1}{\log(\rho)}$. $\alpha_2\ge 0$\\
$V_3(\rho)$: & $\alpha_3\rho^{\gamma_3}$ with $ 0\le \gamma_3\le 1$. $\alpha_3>0$. \\
$F(t,x,\rho)$: & $-\tau\rho(\log(\rho)+V(t,x))$, or 
$-\tau(\rho^m+\rho V(t,x))$.
$\tau\ge 0$, $m> 1$. \\
$G(x,\rho)$: & $\gamma(\rho-\rho^1)^2$, or 
$\gamma\rho\log(\rho/\rho^1)$.
$\gamma> 0$. \\
\hline
\end{tabular}
\caption{Example of various functions for scalar MFC \eqref{smfc2A}.}
\label{table:1}
\end{table}
Here 
$V(t,x)$ for $F(t,x,\rho)$ is a given drift coefficient, and 
$\rho^1$ for $G(x, \rho)$ is a given terminal density. 
Taking the energy 
\begin{align*}
\mathcal{E}(\rho)=\begin{cases}
\int_{\Omega}\frac{1}{\sqrt{\alpha_1\alpha_3}}\rho\log(\rho)\,dx, \quad \hspace{1.4cm}\text{ if }r_1=r_3=1,\\[1ex]
\int_{\Omega}\frac{4\rho^{2-(r_1+r_3)/2}}{(2-r_1-r_3)(4-r_1-r_3)\sqrt{\alpha_1\alpha_3}}\,dx, \quad \text{ else},\\
\end{cases}
\end{align*}
ensures that the relation \eqref{v3-form} holds.
Since $V_1, V_2, V_3$ in Table \ref{table:1}
are positive and concave, $F$ is concave, and $G$ is convex,  the objection functional in the MFC problem \eqref{smfc2A} is convex under a linear constraint. Hence there exists a minimizer. 

\begin{remark}[Comparison with existing models]
When the terminal density $\rho(T,\cdot)=\rho^1$ is prescribed (or when $\gamma\rightarrow +\infty$ 
for the terminal function in Table \ref{table:1}), the problem \eqref{smfc2A} is called the planning problem for MFC. This problem with $\beta=0$ (no $V_3$) was consider in the work \cite{LiLeeOsher22}. 

Taking the regularization parameter $\beta>0$ in MFC \eqref{smfc2A} forms a Schr\"odinger bridge type problem  \cite{ChenGeorgiouPavon2016_relationd,LegerLi2021_hopfcole}. A related planning problem with $\beta>0$ and no potential  $F=0$ was considered recently in \cite{ChenBenvenisteTannenbaum2023_unbalanced}, known as unbalanced regularized optimal mass transport (urOMT). It has applications in cancer imaging. 
Therein, $V_1(\rho)=V_3(\rho)=\rho$, $V_2(\rho)=\alpha(t,x)\rho$ for $\alpha\ge 0$,
and the objective function has an additional term 
$\int_0^T\int_{\Omega}\beta s\log(\rho)dxdt$ compared with our formulation \eqref{smfc2A1} under the same linear constraint \eqref{smfc2A2}.  
\end{remark}

 \subsection{MFC for reaction-diffusion systems}
In this subsection, we introduce a system of MFC problem. 
It can be viewed as a generalization of the system distance \eqref{DisS}.
\begin{definition}[System MFC problem]
\label{vmfc}
Suppose there exist $M\in\mathbb{N}_+$ species and $R\in\mathbb{N}_+$ reactions with reaction coefficient matrix $\Gamma=(\gamma_{i,p})\in\mathbb{R}^{M\times R}$. Let $V_{1,i}:\mathbb{R}_+\rightarrow\mathbb{R}_+$
for $1\le i\le M$,
and
$V_{2,p}:\mathbb{R}_+^M\rightarrow\mathbb{R}_+$
for $1\le p\le R$
be positive mobility functions.   
Take terminal time $T>0$, and non-negative regularization parameter $\beta\ge0$. Choose energy functionals $\mathcal{E}_i\colon \mathcal{M}\rightarrow\mathbb{R}$ for $1\le i\le M$, 
an potential functional $\bm{\mathcal{F}}\colon \mathcal{M}^M\rightarrow\mathbb{R}$, 
and a terminal functional 
$\bm{\mathcal{G}}\colon \mathcal{M}^M\rightarrow\mathbb{R}$. 
\begin{subequations}\label{vmfcA}
Consider 
\begin{equation}\label{vmfcA1}
\begin{split}
&\inf_{\bmr, \tilde{\cmm}, \tilde\bso}\quad\int_0^T\left[\int_\Omega
\left(\sum_{i=1}^M\frac{|\tilde\bmm_i|^2}{2V_{1,i}(\rho_i)}+
\sum_{p=1}^R\frac{|\tilde s_p|^2}{2V_{2,p}(\bmr)}\right)dx
-\bm{\mathcal{F}}(\bmr)\right]dt+\bm{\mathcal{G}}(\bmr(T,\cdot)), 
\end{split}
\end{equation}
where the infimum is among all densities
$\bmr=(\rho_1,\cdots,\rho_M)$ with $\bmr(t,\cdot)\in\mathcal{M}^M$, fluxes
$\tilde\cmm=(\tilde\bmm_1,\cdots,\tilde\bmm_M)$ with $\tilde{\bmm}_i\colon [0, T]\times \Omega \rightarrow\mathbb{R}^d$, and sources $\tilde \bso=(\tilde s_1, \cdots, \tilde s_M)\colon [0, T]\times \Omega \rightarrow\mathbb{R}^M$, such that
\begin{equation}\label{smfcA2}
\begin{split}
&\partial_t \rho_i + \nabla\cdot \tilde{\bmm}_i-
\sum_{p=1}^R\gamma_{i,p}\tilde s_p\\
=&\;\beta
\left(\nabla\cdot(V_{1,i}(\rho_i)\nabla\frac{\delta}{\delta \rho}\mathcal{E}_i(\rho_i))-
\sum_{p=1}^R\gamma_{i,p}V_{2,p}(\bmr)\sum_{j=1}^M\gamma_{j,p}\frac{\delta}{\delta\rho}\mathcal{E}_j(\rho_j)\right),  
\end{split}
\end{equation}
with boundary condition 
\begin{align}
\label{smfcA3}
\left.\left(\tilde{\bmm}_i-\beta V_{1,i}(\rho_i)\nabla\frac{\delta}{\delta \rho}\mathcal{E}_i(\rho_i)\right)\cdot\nnu\right|_{\partial\Omega} =0,
\end{align}
for all $1\le i\le M$, 
and fixed initial density $\bmr(0,\cdot) = \bmr^0\in\mathcal{M}^M$. 
\end{subequations}
\end{definition}

Again, our numerical scheme is based on the following equivalent formulation after the change of variables, whose proof is given in the Appendix.
\begin{proposition}[System MFC reformulation]\label{vmfc2}
For $1\le i\le M$ and $1\le p\le R$, denote $\bmm_i\colon [0,T]\times\Omega\rightarrow\mathbb{R}^d$ 
 and $s_p\colon [0,T]\times\Omega\rightarrow\mathbb{R}$, such that 
\begin{equation*}
 \bmm_i=\tilde\bmm_i-\beta V_{1_i}(\rho_i)\nabla\frac{\delta}{\delta \rho}\mathcal{E}_i(\rho_i),\quad 
 s_p=\tilde s_p-
 \beta V_{2,p}(\bmr)
 \sum_{j=1}^M\gamma_{j,p}\frac{\delta}{\delta\rho}\mathcal{E}_j(\rho_j), 
\end{equation*}
and denote
\begin{equation}
\label{v3-forms}
   V_{3,i}(\rho_i)=\frac{1}{V_{1,i}(\rho_i)|\frac{\delta^2}{\delta \rho^2}\mathcal{E}_i(\rho_i)|^2}. 
\end{equation}
Then, the system MFC problem \eqref{vmfc} is equivalent to the following optimization problem:
Consider 
\begin{subequations}\label{vmfc2A}
\begin{equation}\label{vmfc2A1}
\begin{split}
&\inf_{\rho,\bmm, s} ~~ \int_0^T\int_\Omega 
\sum_{i=1}^M\left(\frac{|\bmm_i|^2}{2V_{1,i}(\rho_i)}+\beta^2\frac{|\nabla \rho_i|^2}{2V_{3,i}(\rho_i)}\right)dxdt\\
&\hspace{1.0cm}+\int_0^T\int_\Omega
\sum_{p=1}^R\left(\frac{|s_p|^2}{2V_{2,p}(\bmr)}
+\frac{\beta^2}{2}{|\sum_{j=1}^M\gamma_{j,p}\frac{\delta\mathcal{E}_j}{\delta \rho}(\rho_j)|^2}{V_{2,p}(\bmr)}
\right)dxdt\\
&\hspace{1.0cm}-\int_{0}^T\bm{\mathcal{F}}(\bmr)\,dt+\bm{\mathcal{G}}(\bmr(T,\cdot))+\beta\sum_{i=1}^M[\mathcal{E}_i(\rho_i(T,\cdot))-
\mathcal{E}_i(\rho_{i}^0)], 
\end{split}
\end{equation}
subject to the constraints
\begin{equation}\label{vmfc2A2}
\partial_t \rho_i(t, x) + \nabla\cdot \bmm_i(t,x)-\sum_{p=1}^R\gamma_{i,p}s_p(t,x)=0, 
\end{equation}
with no-flux boundary condition
$
 \bmm_i\cdot\nnu|_{\partial\Omega}=0, 
 $
 for all $1\le i\le M$,
\end{subequations}
and fixed initial density $\bmr(0,\cdot)=\bmr_0=(\rho_{1}^0,\cdots, \rho_{M}^0)\in\mathcal{M}^M$. 
\end{proposition}
Note that when $M=1$ and $R=1$, the MFC problem \eqref{vmfc2A} corresponds to the scalar case considered in Definition \eqref{smfc}.

Similar to the scalar case, we assume that there are two functions $\bm F:[0,T]\times\Omega\times \mathbb{R}_+^M\rightarrow \mathbb{R}$
and $\bm G:\Omega\times\mathbb{R}_+^M\rightarrow \mathbb{R}$, such that 
\begin{align*}
\int_{\Omega}\bm F(t, x, \bmr)dx = &\;
\bm{\mathcal{F}}(\bmr(t,\cdot))-\int_{\Omega}\sum_{p=1}^{R}\frac{\beta^2}{2}{|\sum_{j=1}^{M}\frac{\delta\mathcal{E}_j}{\delta \rho}(\rho_j)|^2}{V_{2,p}(\bmr)}dx,\\
\int_{\Omega}\bm G(x,\bmr)dx = &\;
\bm{\mathcal{G}}(\bmr)+\beta\sum_{i=1}^M\mathcal{E}_i(\rho_i).
\end{align*}
This simplifies the functional in \eqref{vmfc2A1} as follows:
\begin{equation*}
\begin{split}
&\int_0^T\int_\Omega \left[
\sum_{i=1}^M\left(\frac{|\bmm_i|^2}{2V_{1,i}(\rho_i)}+\beta^2\frac{|\nabla \rho_i|^2}{2V_{3,i}(\rho_i)}\right)+ \sum_{p=1}^R\frac{|s_p|^2}{2V_{2,p}(\bmr)}-\bm{F}(t,x,\bmr)\right]dxdt\\
&+
\int_{\Omega}
\bm{G}(x, \bmr(T,x))dx, 
\end{split}
\end{equation*}

The KKT system for the optimization problem \eqref{vmfc2A} is given below. Its proof is again presented in the Appendix. 
\begin{proposition}[KKT system for \eqref{vmfc2A}]\label{vmfcKKT}
Let $(\bmr, \cmm, \bso)$ be the critical point of the optimization problem \eqref{vmfc2A}. Then there exists a function $\bm{\phi}=(\phi_1,\cdots,\phi_M)\colon [0, T]\times\Omega\rightarrow \mathbb{R}^M$, such that 
\begin{equation*}
\frac{\bmm_i(t,x)}{V_{1,i}(\rho_i(t,x))}=\nabla \phi_i(t,x),\quad \frac{s_p(t,x)}{V_{2,p}(\bmr(t,x))}=
\sum_{j=1}^M\gamma_{j,p}\phi_j(t,x),
\end{equation*}
and 
\begin{equation}\label{SBPI}
\left\{\begin{aligned}
&\partial_t\rho_i +\nabla\cdot(V_{1,i}(\rho_i)\nabla \phi_i)-\sum_{p=1}^RV_{2,p}(\bmr)\gamma_{i,p}\sum_{j=1}^M\gamma_{j,p}\phi_j=0,\\
&\partial_t\phi_i+\frac{1}{2}\|\nabla \phi_i\|^2 V_{1,i}'(\rho_i)+\frac{1}{2}\sum_{p=1}^R|\sum_{j=1}^M\gamma_{j,p}\phi_j|^2
\frac{\partial}{\partial \rho_i}V_{2,p}(\bmr)\\
&\hspace{6.3cm}+\frac{\delta}{\delta \rho_i}\Big[\bm{\mathcal{F}}(\bmr)-\frac{\beta^2}{2}\bm{\mathcal{I}}(\bmr)\Big]=0,
\end{aligned}\right.
\end{equation}
where $\bm{\mathcal{I}}(\bmr)$
is the generalized Fisher information functional given in \eqref{fisherS},
with initial and terminal time boundary conditions
\begin{equation}
   \rho_i(0,x)=\rho_i^0(x),\quad \phi_i(T,x)=-\frac{\delta}{\delta \rho_i}\Big(\bm{\mathcal{G}}(\bmr(T,\cdot))+\beta \mathcal{E}_i(\rho_i(T,\cdot))\Big).  
\end{equation}
\end{proposition}

\subsubsection{Examples}
Let us consider examples with only pairwise interactions. This setting is similar to the reversible Markov chains on discrete states (symmetric weighted graph) \cite{GaoLiLiu,Mielke11, Maas2011_gradientb}. We assume the mobility function $V_{2,p}$ only depends on {\it two} densities $\rho_{p_0}$ and $\rho_{p_1}$
for $p_0\not= p_1\in\{1,\cdots, M\}$ for $1\le p\le R$. 
We take the reaction coefficient $\Gamma=(\gamma_{i,p})\in \mathbb{R}^{M\times R}$, such that
\begin{align}
\label{rr}
\gamma_{i,p}= \begin{cases}
1 & \text{if } i=p_0,\\
-1 & \text{if } i=p_1,\\
0 & \text{else}.
\end{cases}
\end{align}
These reactions naturally lead to a graph with $M$
vertices $\{X_1,\cdots, X_M\}$ and 
$R$ edges $\{E_1,\cdots, E_R\}$ with 
$E_p=(X_{p_0}, X_{p_1})$ connecting vertices $X_{p_0}$ and $X_{p_1}$, where each vertex represents a density species, and each edge represents a reaction between two densities. 
    
    

Different choices of the functions 
$V_{1,i}$, $V_{2,p}$, $V_{3,i}$, $\bm F$, $\bm G$, 
along with the reaction coefficient matrix $\Lambda\in\mathbb{R}^{M\times R}$ in \eqref{rr},
regularization parameter $\beta\ge0$ and initial density $\bmr^0$, 
lead to various system MFC problems \eqref{vmfc}. We present some examples of these functions in Table \ref{table:2} below. 
\begin{table}[ht!]
\centering
\begin{tabular}{l l}
\hline
$V_{1,i}(\rho_i)$: & $\alpha_{1}\rho_i^{\gamma_1}$ with $ 0\le \gamma_1\le 1$. $\alpha_1>0$. \\
$V_{2,p}(\bmr)$: & $\alpha_2(\rho_{p_0}+\rho_{p_1})/2$, or 
$\alpha_2\sqrt{\rho_{p_0}\rho_{p_1}}$,\\
&or 
$\alpha_2\frac{\rho_{p_0}\rho_{p_1}}{\rho_{p_0}+\rho_{p_1}}$, or 
$\alpha_2\frac{\rho_{p_0}-\rho_{p_1}}{\log(\rho_{p_0})-\log(\rho_{p_1})}$. $\alpha_2\ge 0$\\
$V_{3,i}(\rho_i)$: & $\alpha_3\rho_i^{\gamma_3}$ with $ 0\le \gamma_3\le 1$. $\alpha_3>0$. \\
$\bm F(t,x,\bmr)$: & $-
\sum_{i=1}^M\tau_i\rho_i(\log(\rho_i)+V_i(t,x))$, \\
& or 
$-\sum_{i=1}^M\tau_i(\rho_i^m+\rho_i V_i(t,x))$.
$\tau_i\ge 0$, $m> 1$. \\
$G(x,\rho)$: & $\gamma\sum_{i=1}^M(\rho_i-\rho_i^1)^2$, or 
$\gamma\sum_{i=1}^M\rho_i\log(\rho_i/\rho_i^1)$.
$\gamma> 0$. \\
\hline
\end{tabular}
\caption{Example of functions for system MFC \eqref{vmfc}.}
\label{table:2}
\end{table}
Here we can again take the energies
\begin{align*}
\mathcal{E}_i(\rho_i)=\begin{cases}
\int_{\Omega}\frac{1}{\sqrt{\alpha_1\alpha_3}}\rho_i\log(\rho_i)\,dx, \quad \hspace{1.4cm}\text{ if }r_1=r_3=1,\\[1ex]
\int_{\Omega}\frac{4\rho_i^{2-(r_1+r_3)/2}}{(2-r_1-r_3)(4-r_1-r_3)\sqrt{\alpha_1\alpha_3}}\,dx, \quad \text{ else },\\
\end{cases}
\end{align*}
so that the relations \eqref{v3-forms} hold.
Here, the four choices for $V_{2,p}$ correspond to arithmetic, geometric, harmonic, and logarithmic averages of two densities; see \cite{GaoLiLiu,Mielke11, Maas2011_gradientb}. 
Similar to the scalar case, $V_{1,i}, V_{2,p}, V_{3,i}$ in Table \ref{table:2}
are positive and concave, $\bm F$ is concave, and $\bm G$ is convex. Hence the objective functional in MFC problem \eqref{vmfc2A} is convex under a linear constraint. A minimizer is guaranteed. 

This formulation naturally allows general nonlinear reactions among more than two densities. We leave the detailed modeling study of MFC for more general reaction-diffusion systems in future work.

\section{High order discretizations and optimization algorithms}\label{sec3}
In this section, we present the high-order spatial-time finite element formulations for the proposed MFC problems \eqref{smfc2A} and \eqref{vmfc2A}.  
\subsection{Scalar case}\label{sec3:scalar}
In this subsection, we discretize the scalar model \eqref{smfc2A}.
By introducing a new approximation variable $\nn=\beta\nabla\rrho$, the optimization problem \eqref{smfc2A} is rewritten as follows.
Consider 
\begin{subequations}\label{model-scalar}
\begin{equation}\label{MS1}
\inf_{\rrho,\mm, \so, \nn} ~ \int_0^T\int_\Omega \Big[\frac{\|\mm\|^2}{2 V_1(\rrho)}+ \frac{|\so|^2}{2V_2(\rrho)}
+\frac{|\nn|^2}{2V_3(\rrho)}
-\FF(\rrho)\Big]dxdt +
\int_{\Omega}\GG(\rrho(T,x))dx, 
\end{equation}
where the infimum is taken over density $\rrho\colon [0, T]\times \Omega\rightarrow\mathbb{R}_+$, flux $\mm\colon [0, T]\times \Omega \rightarrow\mathbb{R}^d$, source $\so\colon [0, T]\times \Omega \rightarrow\mathbb{R}$, and 
vector field 
$\nn\colon [0, T]\times \Omega \rightarrow\mathbb{R}^d$,
such that
\begin{align}
\label{MS2}
\partial_t \rrho + \nabla\cdot \mm-\so&\;=0, \\
\label{MS3}
\nn - \beta \nabla \rrho &\;= 0,
\end{align}
\end{subequations}
on $[0,T]\times \Omega$,
with fixed initial density $\rrho(0,x)=\rrho^0(x)$ in $\Omega$
and no-flux boundary condition $\mm\cdot\nnu=0$ on $[0,T]\times \partial\Omega$. 
Here we suppress the coordinate dependence of functions $F$ and $G$ for simplicity of presentation.

Following the high-order finite element scheme for MFC problems in \cite{FuLiu23}, we discretize the constraint optimization problem \eqref{model-scalar} using high-order finite element methods 
and apply the ALG2 optimization algorithm to solve the discrete saddle-point problem. 
\subsubsection{The continuous saddle-point formulation}
We introduce a scalar Lagrange multiplier $\pphi\in H^1(\Omega_T)$, where $\Omega_T:=[0,T]\times\Omega$ is the space-time domain, 
for the constraint \eqref{MS2}, 
and a vectorial Lagrange multiplier $\ssigma\in L^2([0,T])\otimes H(\mathrm{div}_0;\Omega)$ for the constraint \eqref{MS3},  where 
\[
H(\mathrm{div}_0;\Omega):=\{
\ttau\in H(\mathrm{div};\Omega): 
\ttau\cdot\nnu|_{\partial\Omega} = 0\}, 
\]
in which the H(div)-conforming space
\[
H(\mathrm{div};\Omega):=\{
\ttau\in [L^2(\Omega)]^d: 
\nabla\cdot\ttau\in L^2(\Omega)
\}.
\]
We reformulate the constrained optimization problem \eqref{model-scalar}
into the following saddle-point system:
Find the critical point of the system
\begin{equation}\label{SD-scalarZ}
\begin{split}
\inf_{\rrho,\mm, \so, \nn, \rrho_T} \sup_{\pphi, \ssigma}~& \int_0^T\int_\Omega \Big[\frac{\|\mm\|^2}{2 V_1(\rrho)}+ \frac{|\so|^2}{2V_2(\rrho)}
+\frac{|\nn|^2}{2V_3(\rrho)}
-\FF(\rrho)\Big]dxdt \\
& \hspace{-0.8cm}
-\int_0^T\int_\Omega \Big[(\rrho\, \partial_t \pphi
+\mm\cdot\nabla \pphi+\so\pphi)
+(\nn\cdot\ssigma+\beta\rho\nabla\cdot\ssigma)
\Big] 
 dxdt\\
 & \hspace{-0.8cm}+
 \int_\Omega \left(\rrho_T\, \pphi(T,\cdot)
-\rrho^0\,\pphi(0,\cdot)\right)
 dx
  +\int_{\Omega}\GG(\rrho_T)dx,
\end{split}
\end{equation} 
where
the variables $\mm,\nn\in [L^2(\Omega_T)]^d$, $\rrho, \so\in L^2(\Omega_T)$, 
$\rho_T\in L^2(\Omega)$
with $\rrho,\rrho_T\ge 0$ a.e., $\pphi\in H^1(\Omega_T)$, 
and $\ssigma\in L^2([0,T])\otimes H(\mathrm{div}_0;\Omega)$.
Notice that integration by parts is used to incorporate the linear constraints \eqref{MS2}--\eqref{MS3} to the saddle point problem \eqref{SD-scalarZ}. Moreover, since no derivative information is needed for the physical variables $\mm, \nn, \rrho$ and $\so$, it is natural to approximate them using $L^2$-conforming spaces.

To simplify the notation, we collect the variables
\begin{subequations}
\label{scalar-collect}
\begin{align}
\label{sc1}
\uu:=\;(\rrho, \mm, \so, \nn), \text{ and }
\Pphi:=\;(\pphi, \ssigma),
\end{align}
and define the differential operator 
$\mathcal{D}$
by 
\begin{equation}
\label{sc2}
\mathcal{D}(\Pphi):=
(\partial_t\pphi+\beta\nabla\cdot\ssigma, \nabla\pphi, \pphi, \ssigma)\in [L^2( \Omega_T)]^{2d+2}.
\end{equation}
\end{subequations}
With this notation, the above saddle-point problem simplifies to the following:
Find the critical point of  
\begin{equation}\label{SD-scalar}
\begin{split}
\inf_{\uu, \rrho_T} \sup_{\Pphi}~& \int_0^T\int_\Omega\Big[\HH(\uu)
-\uu\cdot\mathcal{D}(\Pphi)\Big]
\,dxdt \\
 & \hspace{-0.2cm}+
 \int_\Omega \Big[\GG(\rrho_T)+\rrho_T\, \pphi(T,\cdot)
-\rrho^0\,\pphi(0,\cdot)\Big]
 dx,
\end{split}
\end{equation} 
where 
the variables $\uu\in [L^2(\Omega_T)]^{2d+2}$, $\rho_T\in L^2(\Omega)$
with $\rrho, \rrho_T\ge 0 $ a.e., and 
$\Pphi\in H^1(\Omega_T) \times  \Big(L^2([0,T])\otimes H(\mathrm{div}_0;\Omega)\Big)$, 
the nonlinear functional $\HH(\uu)$ is given by
\[
\HH(\uu):=
\frac{\|\mm\|^2}{2 V_1(\rrho)}+ \frac{|\so|^2}{2V_2(\rrho)}
+\frac{|\nn|^2}{2V_3(\rrho)}
-\FF(\rrho),
\]
and the inner-product is defined by
\[
\uu\cdot\mathcal{D}(\Pphi):=
\rrho (\partial_t \pphi+\beta\nabla\cdot\ssigma)
+\mm\cdot\nabla \pphi+\so\,\pphi+\nn\cdot\ssigma.
\]

\subsubsection{The high-order finite element scheme and numerical integration}
Following the work \cite{FuLiu23}, we approximate the physical variables $\mm, \nn, \so$, $\rrho$ using (discontinuous) high-order integration rule spaces, 
the variable $\pphi$ using a high-order $H^1$-conforming finite element space, and 
the variable $\ssigma$ using a vectorial $H^1$-conforming finite element space.
Let $\Omega_h=\{T_\ell\}_{\ell=1}^{N_S}$ be a conforming mesh of the spatial domain $\Omega$ with $N_S$ elements, and 
$I_h=\{I_j\}_{j=1}^{N_T}$ be a uniform discretization of the temporal domain $[0, T]$
with $N_T$ line segments. We denote the space-time mesh $\Omega_{T,h}:=I_h\otimes\Omega_h$. 
For simplicity, we restrict ourselves to the case where the domain $\Omega=[0,1]^d$ is a unit hypercube and its spatial mesh $\Omega_h$ is a Cartesian mesh with uniform hypercubic elements. 
Given a polynomial degree $k\ge 1$,
we denote the following $H^1$ 
and $L^2$ finite element spaces:
\begin{align}
\label{fes-V}
V_h^k:=&\;\{v\in H^1(\Omega_T): \;\; v|_{I_j\times T_\ell}
\in Q^k(I_j)\otimes Q^k(T_\ell)\;\;
\forall j, \ell\},\\
\label{fes-W}
W_h^{k-1}:=&\;\{w\in L^2(\Omega_T): \;\; w|_{I_j\times T_\ell}
\in Q^{k-1}(I_j)\otimes Q^{k-1}(T_\ell)\;\;
\forall j, \ell\},\\
\label{fes-M}
M_h^{k-1}:=&\;\{\mu\in L^2(\Omega):\;\;\; \;\; \mu|_{T_\ell}
\in Q^{k-1}(T_\ell)\;\;\forall \ell\},
\end{align}
where $Q^k(S)$ is the tensor-product polynomial space of degree no greater than $k$ in each coordinate direction.

Using the above finite element spaces, we obtain the following discrete saddle point problem:
Find the critical point of the discrete system
\begin{equation}\label{SD-scalarh}
\begin{split}
\inf_{\uu_h, \,\rrho_{T,h}} \sup_{\Pphi_h}~& \int_0^T\int_\Omega\Big[\HH(\uu_h)
-\uu_h\cdot\mathcal{D}(\Pphi_h)\Big]
\,dxdt \\
 & \hspace{-0.2cm}+
 \int_\Omega \Big[\GG(\rrho_{T,h})+\rrho_{T,h}\, \pphi_h(T,\cdot)
-\rrho^0\,\pphi_h(0,\cdot)\Big]
 dx,
\end{split}
\end{equation} 
where the variables
$\uu_h:=(\rrho_h, \mm_h, \so_h, \nn_h)$
with $\mm_h,\nn_h\in [W_h^{k-1}]^d$, $\rrho, \so\in W_h^{k-1}$, 
$\rrho\ge0$ a.e., $\rho_{T,h}\in M_h^{k-1}$
with $\rrho_{T,h}\ge 0$ a.e., 
and $\Pphi_h:=(\pphi_h,\ssigma_h)$
with $\pphi_h\in V_h^k$
and $\ssigma_h\in [V_h^{k}]^d$
with $\ssigma_h\cdot\nnu|_{\partial\Omega}=0$.

\begin{remark}[On the choice of the finite element spaces]
In the scheme \eqref{SD-scalarh}, we use $L^2$-conforming spaces to approximate the 
physical variables $\uu_h$ and $\rrho_{T,h}$, which is  natural as no derivative information is needed. Moreover, we use $H^1(\Omega_T)$-conforming finite element spaces to approximate the Lagrange multipliers $\Pphi_h$. The use of $H^1$-conforming space to approximate the scalar variable $\pphi_h$ is natural as the formulation \eqref{SD-scalarh} requires space-time derivative information for $\pphi_h$.
On the other hand, the use of $H^1(\Omega_T)$-conforming space for the vectorial Lagrange multiplier $\ssigma_h$ is only for the simplicity of implementation. 
A more natural finite element space for this variable is the following:
\[
[W_h^k]^d\cap \left(L^2([0,T])\otimes H(\mathrm{div}_0;\Omega)\right),
\]
which is discontinuous in time and $H(\mathrm{div})$-conforming in space. 
We will implement this choice in future work. 

Here we use polynomial degree of one order lower to approximate the physical variables $\uu_h$ than that for the Lagrange multipliers $\Pphi_h$.
Our lowest order scheme is a staggered grid scheme where the physical variables stay in cell centers and the Lagrange multipliers stay on mesh vertices.
\end{remark}

In practice, the integrals in \eqref{SD-scalarh} are discretized using 
Gauss-Legendre quadrature rules with $k$ integration points per direction.
The total number of integration points for the space-time integral on $[0,T]\times \Omega$ is $(kN_T)\times (k^{d}N_S)$, and that for the spatial integral on $\Omega$
is $k^{d}N_S$, where  $N_T$ is the number of temporal cells and $N_T$ is the number of spatial cells.
Denoting these spatial/temporal Gauss-Legendre quadrature points as 
\begin{subequations}
\label{disc-Q}    
\begin{equation}
\label{quad-pt}
\{\bm{\xi}_i\}_{i=1}^{k^dN_S}, \text{ and } \{\eta_j\}_{j=1}^{kN_T}
\end{equation}
with their associated quadrature weights as 
\begin{equation}
\label{quad-wt}
\{\omega_i\}_{i=1}^{k^dN_S}, \text{ and } \{\zeta_j\}_{j=1}^{kN_T},
\end{equation}
we write the discrete integrals as follows:
\begin{align}
\label{d-int1}    
\stint{f(t,x)}:=&\;
\sum_{j=1}^{kN_T}
\sum_{i=1}^{k^dN_S}f(\eta_j, \bm{\xi}_i)\omega_i\zeta_j,\\
\sint{f(x)}:=&\;
\sum_{i=1}^{k^dN_S}f(\bm{\xi}_i)\omega_i.
\label{d-int2}    
\end{align}
\end{subequations}
Furthermore, we use the Gauss-Legendre basis for the $L^2$-conforming spaces $W_h^{k-1}$
and $M_h^{k-1}$ in the implementation, i.e., a function $w_h\in W_h^{k-1}$
is expressed as 
\[
w_h = 
\sum_{j=1}^{kN_T}
\sum_{i=1}^{k^dN_S} \mathsf{w}_{ij}\,\psi_j(t)\varphi_i(x),
\]
with coefficient $\mathsf{w}_{ij}\in \mathbb{R}$.
Here the tensor-product Gauss-Legendre basis function 
$\psi_j(t)\varphi_i(x)\in W_h^{k-1}$ satisfies the nodal interpolation property
\[
\psi_j(\zeta_{j'})=\delta_{jj'}, \text{ and }
\varphi_i(\bm\xi_{i'})=\delta_{ii'},
\]
where $\delta_{ii'}=1$ if $i=i'$ and 
$\delta_{ii'}=0$ if $i\not=i'$. 
A function $v_h\in M_h^{k-1}$ is expressed as 
$
v_h = 
\sum_{i=1}^{k^dN_S} \mathsf{v}_{i}\varphi_i(x)
$
with coefficient $\mathsf{v}_i\in \mathbb{R}$.
Here $\mathrm{w}_{ij}=w_h(\zeta_j,\bm\xi_i)$ is the value of $w_h$
on the space-time integration point $(\zeta_j, \bm\xi_i)$, and 
$\mathrm{v}_{i}=v_h(\bm\xi_i)$ is the value of $v_h$
on the spatial integration point $\bm\xi_i$.
Hence there holds 
\begin{alignat*}{2}
\stint{f(w_h)}:=&\;
\sum_{j=1}^{kN_T}
\sum_{i=1}^{k^dN_S}f(\mathsf{w}_{ij})\omega_i\zeta_j,&&\quad
\forall w_h=\sum_{j=1}^{kN_T}
\sum_{i=1}^{k^dN_S} \mathsf{w}_{ij}\,\psi_j(t)\varphi_i(x)\in W_h^{k-1},
\\
\sint{f(v_h)}:=&\;
\sum_{i=1}^{k^dN_S}f(\mathsf{v}_i)\omega_i,&&\quad\forall v_h
=\sum_{i=1}^{k^dN_S} \mathsf{v}_{i}\varphi_i(x)
\in M_h^{k-1}.
\end{alignat*}

With these notation, the fully discrete scheme \eqref{SD-scalarh}
with numerical integration is given as follows:
Find the critical point of the fully discrete system
\begin{equation}\label{SD-scalarh1}
\begin{split}
\inf_{\uu_h, \,\rrho_{T,h}} \sup_{\Pphi_h}~& \stint{\HH(\uu_h)
-\uu_h\cdot\mathcal{D}(\Pphi_h)}
\\
 & \hspace{-0.2cm}+
 \sint{\GG(\rrho_{T,h})+\rrho_{T,h}\, \pphi_h(T,\cdot)
-\rrho^0\,\pphi_h(0,\cdot)}.
\end{split}
\end{equation} 
Here the unknowns $\rho_{T,h}=\sum_{i=1}^{k^dN_S}\uprho_{T,i}\varphi_i(x)$
with non-negative coefficient $\uprho_{T,i}\in\mathbb{R}_+$,
$\uu_h:=(\rrho_h, \mm_h, \so_h, \nn_h)$ with
\begin{alignat*}{2}
\rrho_h=&\;\sum_{j=1}^{kN_T}
\sum_{i=1}^{k^dN_S}{\uprho}_{ij}\,\psi_j(t)\varphi_i(x),\quad
\mm_h=&&\;\sum_{j=1}^{kN_T}
\sum_{i=1}^{k^dN_S}\bm{\mathrm{m}}_{ij}\psi_j(t)\varphi_i(x),\\
\so_h=&\;\sum_{j=1}^{kN_T}
\sum_{i=1}^{k^dN_S}{\mathrm{s}}_{ij}\psi_j(t)\varphi_i(x),\quad 
\nn_h=&&\;\sum_{j=1}^{kN_T}
\sum_{i=1}^{k^dN_S}\bm{\mathrm{n}}_{ij}\psi_j(t)\varphi_i(x),
\end{alignat*}
where coefficients $\uprho_{ij}\in \mathbb{R}_+$, $\bm{\mathrm{m}}_{ij},\bm{\mathrm{n}}_{ij}\in\mathbb{R}^d$, and $\mathrm{s}_{ij}\in\mathbb{R}$,
and $\Pphi_h:=(\pphi_h,\ssigma_h)\in [V_h^k]^{d+1}$
with $\ssigma_h\cdot\nnu|_{\partial\Omega}=0$.

We further note that the use of Gauss-Legendre integral rule along with the (nodal) Gauss-Legendre bases for the discontinuous spaces $W_h^{k-1}$ and $M_h^{k-1}$ significantly simplifies the optimization problem \eqref{SD-scalarh1}, as no 
explicit degrees of freedom coupling is introduced among the physical variables $\uu_h$ and 
$\rrho_{T,h}$.
In particular, given a fixed Lagrange multiplier $\Pphi_h$, the nonlinear optimization problem \eqref{SD-scalarh1} for $\uu_h$ and 
$\rrho_{T,h}$ can be solved in a pointwise fashion per integration point.
Moreover, positivity of the (space-time) density $\rrho_h$ and the terminal density 
$\rrho_{T,h}$ is guaranteed on all quadrature points by design as the admissible set requires
their coefficients to be non-negative: $\uprho_{ij}, \uprho_{T,i}\ge0$.
The scheme \eqref{SD-scalarh1} does not prescribe the choice of basis functions for the $H^1$-conforming space $V_h^k$. The results are independent of the specific choice of bases for $V_h^k$. We use a nodal Gauss-Lobatto base in our implementation.  

The scheme \eqref{SD-scalarh1} is our fully discrete high-order finite element discretization to the scalar constraint optimization problem \eqref{model-scalar}.
Next, we present its optimization solver using the ALG2 algorithm \cite{FortinBook}, which is also referred to as the alternating direction method of multipliers (ADMM) method \cite{ADMM}.

\subsubsection{Duality and augmented Lagrangian}
We introduce the dual variables 
\[
\uu_h^*=(\rrho_h^*, \mm_h^*, \so_h^*, \nn_h^*)\in [W_h^{k-1}]^{2d+2}, \text{ and }
\rho_{T,h}^*
=\sum_{i=1}^{k^dN_S}\uprho_{T,i}^*\varphi_i(x)
\in M_h^{k-1}. 
\]
To further simplify the notation, we denote 
\[
\uu_h=\;\sum_{j=1}^{kN_T}
\sum_{i=1}^{k^dN_S}{\bm{\mathrm{u}}}_{ij}\,\psi_j(t)\varphi_i(x),\text{ and }
\uu_h^*=\;\sum_{j=1}^{kN_T}
\sum_{i=1}^{k^dN_S}{\bm{\mathrm{u}}}_{ij}^*\,\psi_j(t)\varphi_i(x),
\]
where 
\[
\bm{\mathrm{u}}_{ij}:=
(\uprho_{ij}, \bm{\mathrm{m}}_{ij},\mathrm{s}_{ij}, \bm{\mathrm{n}}_{ij})
\in \mathbb{R}^{2d+2}, \text{ with }\uprho_{ij}\ge 0, 
\]
and 
\[
\bm{\mathrm{u}}_{ij}^*:=
(\uprho_{ij}^*, \bm{\mathrm{m}}_{ij}^*,\mathrm{s}_{ij}^*, \bm{\mathrm{n}}_{ij}^*)
\in \mathbb{R}^{2d+2}.
\]
Introducing the following Legendre transforms:
\begin{subequations}
    \label{duals}
\begin{align}
\HH^*(\uui_{ij}^*):=&\; \sup_{\uui_{ij}\in \mathbb{R}^{2d+2}, 
\uprho_{ij}\ge 0} \uui_{ij}\cdot\uui_{ij}^*-
\HH(\uui_{ij}),\\
\GG^*(\rti^*):=&\; \sup_{\rti\in \mathbb{R}_+} \rti\cdot\rti^*-
\GG(\rti),
\end{align}
by duality, there holds
\begin{align}
\label{dual-H}
\HH(\uui_{ij}):=&\; \sup_{\uui_{ij}^*\in \mathbb{R}^{2d+2}} \uui_{ij}\cdot\uui_{ij}^*-
\HH^*(\uui_{ij}^*),\\
\label{dual-G}
\GG(\rti):=&\; \sup_{\rti^*\in \mathbb{R}} \rti\cdot\rti^*-
\GG^*(\rti^*).
\end{align}
\end{subequations}
Plugging the relations \eqref{dual-H} and \eqref{dual-G} back to the scheme \eqref{SD-scalarh1}, we have the following dual formulation of \eqref{SD-scalarh1}: 
Find the critical point of
\begin{equation}\label{SD-scalarh2}
\begin{split}
\sup_{\uu_h, \,\rrho_{T,h}} \inf_{\Pphi_h,\uu_h^*, \,\rrho_{T,h}^*}~& \stint{\HH^*(\uu_h)
+\uu_h\cdot(\mathcal{D}(\Pphi_h)-\uu_h^*)}
\\
 & \hspace{-0.2cm}+
 \sint{\GG^*(\rrho_{T,h}^*)-\rrho_{T,h}\, (\pphi_h(T,\cdot)+\rrho_{T,h}^*)
+\rrho^0\,\pphi_h(0,\cdot)}.
\end{split}
\end{equation} 
It is clear that in the above formulation \eqref{SD-scalarh2}, 
the variable $\uu_h$ is the Lagrange multiplier for the constraint $\mathcal{D}(\Pphi_h)-\uu_h^*=0$, while the variable $\rrho_{T,h}$ is the Lagrange multiplier for the terminal constraint $\pphi_h(T,x)+\rho_{T,h}^*(x) = 0$.

Hence, the critical point of the system \eqref{SD-scalarh2} is equivalent to the 
following augmented Lagrangian system 
\begin{equation}\label{SD-scalarh3}
\begin{split}
\sup_{\uu_h, \,\rrho_{T,h}} \inf_{\Pphi_h,\uu_h^*, \,\rrho_{T,h}^*}~& \stint{\HH^*(\uu_h^*)
+\uu_h\cdot(\mathcal{D}(\Pphi_h)-\uu_h^*)}
\\
 & \hspace{-0.2cm}+
 \sint{\GG^*(\rrho_{T,h}^*)-\rrho_{T,h}\, (\pphi_h(T,\cdot)+\rrho_{T,h}^*)
+\rrho^0\,\pphi_h(0,\cdot)}\\
&\hspace{-0.2cm}+
\frac{r}{2}\stint{(\mathcal{D}(\Pphi_h)-\uu_h^*)\cdot(\mathcal{D}(\Pphi_h)-\uu_h^*)}\\
&\hspace{-0.2cm}+
\frac{r}{2}\sint{(\pphi_h(T,\cdot)+\rrho_{T,h}^*)^2},
\end{split}
\end{equation} 
where $r>0$ is the augmented Lagrangian parameter.

\subsubsection{The ALG2 algorithm and its efficient (modified) implementation}
The ALG2 solves the optimization problem \eqref{SD-scalarh3} in a splitting fashion. One update from iteration level $\ell-1$ to $\ell$ contains three steps:
\begin{itemize}
\item Step A: Given data $\uu_h^{\ell-1},\rrho_{T,h}^{\ell-1},\uu_h^{*,\ell-1},\rrho_{T,h}^{*,\ell-1}$, 
compute $\Phi_h^{\ell}$ by optimizing the target functional in \eqref{SD-scalarh3}
with respect to $\Phi_h$.
\item Step B: Given data
$\uu_h^{\ell-1},\rrho_{T,h}^{\ell-1},\Pphi_h^{\ell}$, 
compute $\uu_h^{*,\ell}$ and $\rrho_{T,h}^{*,\ell}$ by optimizing the target functional in \eqref{SD-scalarh3} with respect to $\uu_h^*$ and $\rrho_{T,h}^*$.
\item Step C: Update the Lagrange multipliers 
$\uu_h^{\ell},\rrho_{T,h}^{\ell}$ by the following explicit expressions:
\begin{subequations}
    \label{sd-lag}
\begin{align}
\label{sd-lag1}
\uu_h^{\ell}=&\;\uu_h^{\ell-1}+r(\mathcal{D}(\Pphi_h^{\ell})-\uu_h^{*,\ell}),\\
\label{sd-lag2}
\rrho_{T,h}^{\ell}=&\;\rrho_{T,h}^{\ell-1}-r(\pphi_h^{\ell}(T,\cdot)+\rrho_{T,h}^{*,\ell}).
\end{align}
\end{subequations}
\end{itemize}
Here Step A contains a global constant-coefficient linear system solver for the variables $\Pphi_h=(\pphi_h, \ssigma_h)$, 
Step B is a pointwise nonlinear equation solved per quadrature point that involves the dual functions $\HH^*$ and $\GG^*$, and Step C is a simple pointwise update.
In practical implementation, we apply two further simplifications to the above ALG2 algorithm: (1) apply splitting to solve $\pphi_h$ and $\ssigma_h$ sequentially; (2) use duality to solve pointwise nonlinear problems for $\uu_h$ and $\rrho_{T,h}$, and then apply a simple pointwise update  for the dual variables $\uu_h^*$ and $\rrho_{T,h}^*$.
 We note that the second modification avoids the 
explicit computation of the dual functions, which has already been used in our previous work
 on ALG2 for variational time implicit schemes for reaction-diffusion systems \cite{FuOsherLi}.
One iteration of the modified ALG2 algorithm is documented in Algorithm
\ref{alg:1} below. 
We skip the detailed derivation of each step. Since the whole algorithm is very similar to  Algorithm 2 in \cite{FuOsherLi} with simple modifications.

\begin{algorithm}
\caption{One iteration of (modified) ALG2  for \eqref{SD-scalarh3}.}
\label{alg:1}
\begin{algorithmic}[1]
\STATE Step A1.
Find $\pphi_h^{\ell}\in V_h^k$ 
such that, for all $\psi_h\in V_h^k$,
\begin{align*}
&    \stint{\partial_t\pphi_h^\ell\cdot\partial_t\psi_h+
\nabla\pphi_h^\ell\cdot\nabla\psi_h+\pphi_h^\ell\cdot\psi_h}
+\sint{\pphi_h^\ell(T,\cdot)\psi_h(T,\cdot)}\\
&=
\stint{(\rrho_h^{*,\ell-1}-\frac{\rrho_h^{\ell-1}}{r}-\beta\nabla\cdot\ssigma_h^{\ell-1})\partial_t\psi_h+
(\mm_h^{*,\ell-1}-\frac{\mm_h^{\ell-1}}{r})\cdot\nabla{\psi}_h
}\\
&\;+\stint{
(\so_h^{*,\ell-1}-\frac{\so_h^{\ell-1}}{r})\cdot{\psi}_h
}
-\sint{(\rrho_{T,h}^{*,\ell-1}-\frac{\rrho_{T,h}^{\ell-1}}{r})\cdot\psi_h(T,\cdot)}
-\sint{\frac{\rrho^0}{r}\cdot\psi_h(0,\cdot)}.
\end{align*}
\vspace{-.5cm}
\STATE  Step A2. 
Find $\ssigma_h^{\ell}\in [V_h^k]^d$ with $\ssigma_h^\ell\cdot\nnu|_{\partial\Omega}=0$ 
such that 
\begin{align*}
&    \stint{\beta^2(\nabla\cdot\ssigma_h^\ell)(\nabla\cdot\ttau_h)
+\ssigma_h^\ell\cdot\ttau_h}
\\
&=
\stint{\beta(\rrho_h^{*,\ell-1}-\frac{\rrho_h^{\ell-1}}{r}-\partial_t\pphi_h^{\ell})\nabla\cdot\ttau_h+
(\nn_h^{*,\ell-1}-\frac{\nn_h^{\ell-1}}{r})\cdot\ttau_h
}
\end{align*}
for all $\ttau_h\in [V_h^k]^d$ with $\ttau_h\cdot\nnu|_{\partial\Omega}=0$.
\STATE Step A3. Evaluate $\mathcal{D}(\Pphi_h^\ell)$ on each space-time quadrature point $(\zeta_j, \bm\xi_i)$, and $\pphi_h^\ell(T,\cdot)$ on each spatial quadrature point $\bm\xi_i$.
Denote these values as 
\begin{align*}
\mathcal{D}\Pphi_{ij}^\ell:=(\mathrm{d}\uprho_{ij}^\ell, 
\bm{\mathrm{dm}}_{ij}^\ell, 
\mathrm{ds}_{ij}^\ell, 
\bm{\mathrm{dn}}_{ij}^\ell),\quad 
\mathrm{d}\uprho_{T,i}^\ell=\pphi_h^\ell(T,\bm\xi_i).
\end{align*}
Set
$\overline{\uui}_{ij}^\ell=(\overline{\uprho}_{ij}^\ell,
\overline{\bm{\mathrm{m}}}_{ij}^\ell,
\overline{{\mathrm{s}}}_{ij}^\ell,
\overline{\bm{\mathrm{n}}}_{ij}^\ell):=
\mathcal{D}\Phi_{ij}^\ell+\uui_{ij}^{\ell-1}/r, 
$
and 
$\overline{\uprho}_{T,i}^\ell:=
\rti^{\ell-1}/r-\mathrm{d}\uprho_{T,i}^\ell.
$

\STATE Step B1. For each quadrature point $(\zeta_j,\bm\xi_i)$, find the nonnegative 
density coefficient $\uprho_{ij}^\ell\in \mathbb{R}_+$ such that it minimizes the function \begin{align*}
L_{ij}(\uprho):=
\frac{r^2|\overline{\bm{\mathrm{m}}}_{ij}^\ell|^2}{r+V_1(\uprho)}
+\frac{r^2|\overline{{\mathrm{s}}}_{ij}^\ell|^2}{r+V_2(\uprho)}
+\frac{r^2|\overline{\bm{\mathrm{n}}}_{ij}^\ell|^2}{r+V_3(\uprho)}
+\frac{(\uprho-r\overline{\uprho}_{ij}^\ell)^2}{r}
-2\FF(\uprho),
\end{align*}
and find the nonnegative terminal density coefficient 
$\uprho_{T,i}^\ell\in \mathbb{R}_+$ such that it minimizes the function 
$
L_{T,i}(\uprho):=
2\GG(\uprho)+{(\uprho-r\overline{\uprho}_{T,i}^\ell)^2}/r.
$
\STATE Step B2. Update the following coefficients:
\begin{align*}
\bm{\mathrm{m}}_{ij}^\ell=\;\frac{rV_1(\uprho_{ij}^\ell)}{r+V_1(\uprho_{ij}^\ell)}
\overline{\bm{\mathrm{m}}}_{ij}^\ell,\;\;
{\mathrm{s}}_{ij}^\ell=\;\frac{rV_2(\uprho_{ij}^\ell)}{r+V_2(\uprho_{ij}^\ell)}
\overline{{\mathrm{s}}}_{ij}^\ell,\;\;
\bm{\mathrm{n}}_{ij}^\ell=\;\frac{rV_3(\uprho_{ij}^\ell)}{r+V_3(\uprho_{ij}^\ell)}
\overline{\bm{\mathrm{n}}}_{ij}^\ell.
\end{align*}

\STATE Step C. update the dual variables according to the following:
\begin{align*}
\uu_{ij}^{*,\ell} = \overline{\uu}_{ij}^{\ell}-\uu_{ij}^{\ell}/r, \quad\uprho_{T,i}^{*,\ell} =
\overline{\rrho}_{T,i}^{\ell}-\rrho_{T,i}^{\ell}/r.
\end{align*}
\end{algorithmic}
\end{algorithm}

\begin{remark}[Computational complexity for Algorithm \ref{alg:1}]
Step A1 involves solving a symmetric positive definite linear system
for a constant-coefficient reaction-diffusion equation, for which we use preconditioned conjugate gradient (PCG) method with a geometric multigrid preconditioner. It achieves a linear computational complexity $\mathcal{O}(N)$ with $N=\dim V_h^k$ being the total number of degrees of freedom.
Step A2 is an $H(div)$-elliptic linear system problem.
This system is well-conditioned for small parameter $\beta\ll 1$, which is usually 
the case in our applications. We use the PCG method with a Jacobi preconditioner, which achieves optimal linear computational complexity $\mathcal{O}(dN)$.
Step B and C involve pointwise updates, which again have a linear complexity of $\mathcal{O}(N)$.
Hence the overall computational complexity of applying one iteration of Algorithm \ref{alg:1}
is linear. 
It is also clear that Step B1 guarantees the positivity of density on all quadrature points.
\end{remark}

\begin{remark}[On unique solvability for Step B1]
Unique solvability of the one-dimensional minimization problem in Step B1 of Algorithm \ref{alg:1} can be guaranteed when the functions $L_{ij}$ and $L_{T,i}$
are strongly convex. 
Strong convexity is ensured if the mobility functions $V_1, V_2$, and $V_3$ are positive and concave functions,  the potential function $F(\uprho)$ is concave, and 
the function $G(\uprho)$ is convex, which are satisfied by the choices in Table \ref{table:1}.
\end{remark}

\subsection{System case}\label{sec3:system}
In this subsection, we discretize the model \eqref{vmfc2A} with $M$ species and $R$ reactions.
We note that due to our formulation, there are no additional technical difficulties for the system derivation compared with the scalar case. We still document details of the finite element scheme and its splitting optimization solver for the completeness of this paper.

Again, we introduce the vectors $\nn_i:=\beta\nabla\rrho_i$ to avoid derivative evaluation of the densities for the scheme.
To simplify the notation, we
denote the collection of densities $\brrho=(\rrho_1,\cdots,\rrho_M)$, 
fluxes $\cmm=(\mm_1,\cdots,\mm_M)$, vectors $\bnn=(\nn_1,\cdots,\nn_M)$, and sources $\bso=(\so_1,\cdots, \so_R)$.
Then optimization problem \eqref{vmfc2A} is rewritten as follows:
\begin{subequations}\label{model-system}
\begin{equation}\label{MY1}
\begin{split}
\inf_{\brrho,\cmm, \bso, \bnn} ~& \int_0^T\int_\Omega \left[\sum_{i=1}^M\left(\frac{\|\mm_i\|^2}{2 V_{1,i}(\rrho_i)}
+\frac{|\nn_i|^2}{2V_{3,i}(\rrho_i)}\right)
+ \sum_{p=1}^R\frac{|\sop|^2}{2V_{2,p}(\brrho)}\right] \, dxdt \\
&-
\int_0^T\int_\Omega \bFF(\brrho)\,dxdt
+
\int_{\Omega}\bGG(\brrho(T,x))dx, 
\end{split}
\end{equation}
subject to constraints
\begin{align}
\label{MY2}
\partial_t \rrho_i + \nabla\cdot \mm_i-\sum_{p=1}^R\gamma_{i,p}\so_p &\;=0, \\
\label{MY3}
\nn_i - \beta \nabla \rrho_i &\;= 0,
\end{align}
\end{subequations}
for all $1\le i\le M$ on $[0,T]\times \Omega$,
with fixed initial densities $\rrho_i(0,x)=\rrho_{i}^0(x)$ in $\Omega$
and no-flux boundary conditions $\mm_i\cdot\nnu=0$ on $[0,T]\times \partial\Omega$.

Similar to the scalar case in \eqref{SD-scalar}, we reformulate the constrained optimization problem \eqref{model-scalar}
into the following saddle-point system:
Find the critical point of  
\begin{equation}\label{SY-scalar}
\begin{split}
\inf_{\buu, \brrho_T} \sup_{\bPphi}~& \int_0^T\int_\Omega\Big[\bHH(\buu)
-\buu\cdot\underline{\mathcal{D}}(\bPphi)\Big]
\,dxdt \\
 & \hspace{-0.2cm}+
 \int_\Omega \Big[\bGG(\brrho_T)+\brrho_T\cdot \bpphi(T,\cdot)
-\brrho^0\cdot\bpphi(0,\cdot)\Big]
 dx.
\end{split}
\end{equation} 
Here 
the variables 
\begin{align}
\label{sy1}
&\buu:=\;(\brrho, \cmm, \bso, \bnn)\in [L^2(\Omega_T)]^{(2d+1)M+R}, \\
&\brrho_T:=\;(\rrho_{1,T}, \cdots, \rrho_{M,T})\in [L^2(\Omega)]^{M},\\
&\Pphi:=\;(\bpphi, \bssigma)\in [H^1(\Omega_T)]^M\times 
\left[
L^2([0,T])\otimes H(\mathrm{div}_0;\Omega)
\right]^M,
\end{align}
with 
\[
\bpphi=(\pphi_1,\cdots,\pphi_M), \text{ and }
\bssigma=(\ssigma_1,\cdots,\ssigma_M),
\]
and initial data $\brrho^0:=(\rrho_{1}^0,\cdots, \rrho_{M}^0)$.
Here the nonlinear function
\[
\bHH(\buu):=
\sum_{i=1}^M\left(\frac{\|\mm_i\|^2}{2 V_{1,i}(\rrho_i)}+\frac{|\nn_i|^2}{2V_{3,i}(\rrho_i)}\right)
+ \sum_{p=1}^R\frac{|\so_p|^2}{2V_{2,p}(\brrho)}
-\bFF(\brrho),
\]
and the operator 
\[
\underline{\mathcal{D}}(\bPphi):=
(\partial_t\bpphi+\beta\nabla\cdot\bssigma, \nabla\bpphi, \Gamma^T\bpphi, \bssigma),
\]
where 
$
\nabla\cdot\bssigma:=(\nabla\cdot\ssigma_1,\cdots, \nabla\cdot\ssigma_M)
$
is the component-wise divergence and 
$
\nabla\bpphi:=(\nabla\pphi_1,\cdots, \nabla\pphi_M)
$
is the component-wise gradient, and 
$
\Gamma = (\gamma_{i,p})\in \mathbb{R}^{M\times R}
$
is the reaction coefficient matrix.
Note that the inner-product 
$\uu\cdot\mathcal{D}(\Pphi)$
has the following component-wise form
\[
\uu\cdot\mathcal{D}(\Pphi)=
\sum_{i=1}^M\left(\rrho_i (\partial_t \pphi_i+\beta\nabla\cdot\ssigma_i)
+\mm_i\cdot\nabla \pphi_i+\nn_i\cdot\ssigma_i\right)
+\sum_{p=1}^R\sum_{i=1}^M\so_p\gamma_{i,p}\pphi_i.
\]

Replacing the function spaces with appropriate (high-order) finite element spaces and applying numerical integration, 
we derive the following fully discrete scheme:
Find the critical point of the fully discrete system
\begin{equation}\label{SY-scalarh1}
\begin{split}
\inf_{\buu_h, \,\brrho_{T,h}} \sup_{\bPphi_h}~& \stint{\bHH(\buu_h)
-\buu_h\cdot\underline{\mathcal{D}}(\bPphi_h)}
\\
 & \hspace{-0.2cm}+
 \sint{\bGG(\brrho_{T,h})+\brrho_{T,h}\cdot \bpphi_h(T,\cdot)
-\brrho^0\cdot\bpphi_h(0,\cdot)}.
\end{split}
\end{equation} 
Here the unknowns $\brrho_{T,h}=\sum_{\ii=1}^{k^dN_S}{\bm\uprho}_{T,\ii}\varphi_\ii(x)$
with non-negative coefficient ${\bm\uprho}_{T,i}\in\mathbb{R}_+^M$,
$\buu_h:=(\brrho_h, \cmm_h, \bso_h, \bnn_h)$ with
\begin{alignat*}{2}
\brrho_h=&\;\sum_{\jj=1}^{kN_T}
\sum_{\ii=1}^{k^dN_S}{\bm\uprho}_{ij}\,\psi_j(t)\varphi_i(x),\quad
\cmm_h=&&\;\sum_{j=1}^{kN_T}
\sum_{i=1}^{k^dN_S}\underline{\bm{\mathrm{m}}}_{ij}\psi_j(t)\varphi_i(x),\\
\bso_h=&\;\sum_{j=1}^{kN_T}
\sum_{i=1}^{k^dN_S}{\bm{\mathrm{s}}}_{ij}\psi_j(t)\varphi_i(x),\quad 
\bnn_h=&&\;\sum_{j=1}^{kN_T}
\sum_{i=1}^{k^dN_S}\underline{\bm{\mathrm{n}}}_{ij}\psi_j(t)\varphi_i(x),
\end{alignat*}
where coefficients $\bm\uprho_{ij}\in \mathbb{R}_+^M$, $\underline{\bm{\mathrm{m}}}_{ij},\underline{\bm{\mathrm{n}}}_{ij}\in\mathbb{R}^{dM}$, and $\bm{\mathrm{s}}_{ij}\in\mathbb{R}^R$,
and $\bPphi_h:=(\bpphi_h,\bssigma_h)\in [V_h^k]^{(d+1)M}$
with $\bssigma_h\cdot\nnu|_{\partial\Omega}=0$.

We introduce the dual variables 
\[
\buu_h^*=(\brrho_h^*, \cmm_h^*, \bso_h^*, \bnn_h^*)\in [W_h^{k-1}]^{(2d+1)M+R}, 
\]
and
$\brrho_{T,h}^*
=\sum_{i=1}^{k^dN_S}\bm\uprho_{T,i}^*\varphi_i(x)
\in M_h^{k-1}$, and denote 
\[
\buu_h=\;\sum_{j=1}^{kN_T}
\sum_{i=1}^{k^dN_S}\underline{\bm{\mathrm{u}}}_{ij}\,\psi_j(t)\varphi_i(x),\text{ and }
\buu_h^*=\;\sum_{j=1}^{kN_T}
\sum_{i=1}^{k^dN_S}\underline{\bm{\mathrm{u}}}_{ij}^*\,\psi_j(t)\varphi_i(x),
\]
with coefficients 
\[
\underline{\bm{\mathrm{u}}}_{ij}:=
(\bm\uprho_{ij}, \underline{\bm{\mathrm{m}}}_{ij},\bm{\mathrm{s}}_{ij}, \underline{\bm{\mathrm{n}}}_{ij})
\in \mathbb{R}^{(2d+1)M+R}, \text{ with }\bm\uprho_{ij}\ge 0, 
\]
and 
\[
\underline{\bm{\mathrm{u}}}_{ij}^*:=
(\bm\uprho_{ij}^*, \underline{\bm{\mathrm{m}}}_{ij}^*,\bm{\mathrm{s}}_{ij}^*, \underline{\bm{\mathrm{n}}}_{ij}^*)
\in \mathbb{R}^{(2d+1)M+R}. 
\]
Using these dual variables,
we obtain the following augmented Lagrangian formulation of \eqref{SY-scalarh1}:
\begin{equation}\label{SY-scalarh3}
\begin{split}
\sup_{\buu_h, \,\brrho_{T,h}} \inf_{\bPphi_h,\buu_h^*, \,\brrho_{T,h}^*}~& \stint{\bHH^*(\buu_h^*)
+\buu_h\cdot(\underline{\mathcal{D}}(\bPphi_h)-\buu_h^*)}
\\
 & \hspace{-0.2cm}+
 \sint{\bGG^*(\brrho_{T,h}^*)-\brrho_{T,h}\cdot (\bpphi_h(T,\cdot)+\brrho_{T,h}^*)}\\
&\hspace{-0.2cm}+
\frac{r}{2}\stint{(\underline{\mathcal{D}}(\bPphi_h)-\buu_h^*)\cdot(\underline{\mathcal{D}}(\bPphi_h)-\buu_h^*)}\\
&\hspace{-0.2cm}+
\frac{r}{2}\sint{|\bpphi_h(T,\cdot)+\brrho_{T,h}^*|^2}
+\stint{\brrho^0\cdot\bpphi_h(0,\cdot)}
\end{split}
\end{equation} 

Finally, we introduce a splitting algorithm (modified ALG2) for 
the saddle-point system \eqref{SY-scalarh3}, where we sequentially compute each component of 
$\bPphi_h$ in the linear elliptic update (Step A) 
and of the densities $\brrho_h$ and $\brrho_{T,h}$ in the nonlinear update (Step B1); see also \cite{FuOsherLi}.

The coupled global elliptic linear system for $\bPphi_h=(\bpphi_h, \bssigma_h)$ in ALG2 takes the following form:
Find $(\bpphi_h, \bssigma_h)\in [V_h^{k}]^{(d+1)M}$ with 
$\bssigma_h\cdot\nnu|_{\partial\Omega}=0$,
such that, for all 
$(\bppsi_h, \bttau_h)\in [V_h^{k}]^{(d+1)M}$ with 
$\bttau_h\cdot\nnu|_{\partial\Omega}=0$, there exists
\begin{align}
\label{eqn-Phi}
\begin{split}
&\sum_{i=1}^M\stint{(\partial_t\pphi_{i,h}+\beta\nabla\cdot\ssigma_{i,h})
(\partial_t\psi_{i,h}+\beta\nabla\cdot\ttau_{i,h})}\\
&\quad+\stint{\nabla\pphi_{i,h}\cdot\nabla\psi_{i,h}+\ssigma_{i,h}\cdot\ttau_{i,h}}
+\sint{\pphi_{i,h}(T,\cdot)\psi_{i,h}(T,\cdot)}\\
&+\sum_{p=1}^R\stint{(\textstyle\sum_{i=1}^M\gamma_{i,p}\pphi_{i,h})(\sum_{i=1}^M\gamma_{i,p}\psi_{i,h})}
\\
=\;\;
&\sum_{i=1}^M\stint{(\rrho_{i,h}^{*}-\frac{\rrho_{i,h}}{r})
(\partial_t\psi_{i,h}+\beta\nabla\cdot\ttau_{i,h})}\\
&\quad+\stint{(\mm_{i,h}^*-\frac{\mm_{i,h}}{r})\cdot\nabla\psi_{i,h}
+(\nn_{i,h}^*-\frac{\nn_{i,h}}{r})\cdot\ttau_{i,h}}\\
&\quad
-\sint{(\rho_{i,T,h}^*-\frac{\rho_{i,T,h}}{r})\psi_{i,h}(T,\cdot)}
+\sint{(\rho_{i,0}^*-\frac{\rho_{i,0}}{r})\psi_{i,h}(0,\cdot)}\\
&+\sum_{p=1}^R\stint{((\so_{p,h}^*-\frac{\so_{p,h}}{r}))(\textstyle\sum_{i=1}^M\gamma_{i,p}\psi_{i,h})}.
\end{split}
\end{align}
This system is solved sequentially in practice to drive down the overall computational cost.
The $i$-th component scalar reaction-diffusion solver for $\pphi_{i,h}^\ell$ at $\ell$-th iteration reads as follows:
\begin{align}
\label{eqn-phi1}
\begin{split}
&\stint{\partial_t\pphi_{i,h}^\ell\partial_t\psi_{i,h}+\nabla\pphi_{i,h}^\ell\cdot\nabla\psi_{i,h}}
\\
&
+\sint{\pphi_{i,h}^\ell(T,\cdot)\psi_{i,h}(T,\cdot)}
+\sum_{p=1}^R\stint{\gamma_{i,p}^2\pphi_{i,h}^\ell\psi_{i,h}}
\\
=\;\;
&\stint{(\rrho_{i,h}^{*,\ell-1}-\frac{\rrho_{i,h}^{\ell-1}}{r}-\beta\nabla\cdot\ssigma_{i,h}^{\ell-1})
\partial_t\psi_{i,h}}\\
&
+\stint{(\mm_{i,h}^{*,\ell-1}-\frac{\mm_{i,h}^{\ell-1}}{r})\cdot\nabla\psi_{i,h}}
\\
&\quad
-\sint{(\rho_{i,T,h}^{*,\ell-1}-\frac{\rho_{i,T,h}^{\ell-1}}{r})\psi_{i,h}(T,\cdot)}
+\sint{(\rho_{i,0}^*-\frac{\rho_{i,0}}{r})\psi_{i,h}(0,\cdot)}\\
&+\sum_{p=1}^R\stint{(\so_{p,h}^*-\frac{\so_{p,h}}{r}-
\textstyle\sum_{j=1}^{i-1}\gamma_{j,p}\pphi_{j,h}^\ell
-\sum_{j=i+1}^{M}\gamma_{j,p}\pphi_{j,h}^{\ell-1}
)\gamma_{i,p}\psi_{i,h}}.
\end{split}
\end{align}
The $i$-th component $H(\mathrm{div})$-elliptic solver for $\ssigma_{i,h}^\ell$ reads as follows
\begin{align}
\label{eqn-sigma1}
\begin{split}
&\stint{\beta^2(\nabla\cdot\ssigma_{i,h}^\ell)
(\nabla\cdot\ttau_{i,h})
+\ssigma_{i,h}^\ell\cdot\ttau_{i,h}}\\
&\;=\;
\sum_{i=1}^M\stint{(\rrho_{i,h}^{*,\ell-1}-\frac{\rrho_{i,h}^{\ell-1}}{r}-\partial_t\pphi_{i,h}^\ell)
\beta\nabla\cdot\ttau_{i,h}
+(\nn_{i,h}^{*,\ell-1}-\frac{\nn_{i,h}^{\ell-1}}{r})\cdot\ttau_{i,h}}.
\end{split}
\end{align}

Meanwhile, the nonlinear update on each quadrature point for 
the densities $\bm{\uprho}_{ij}^\ell=({\uprho}_{1,ij}^\ell,\cdots, {\uprho}_{M,ij}^\ell)\in \mathbb{R}_+^M$ 
takes the following form:
find $\bm{\uprho}_{ij}^\ell\in\mathbb{R}_+^M$ that minimizes the function
\begin{align}
\label{eqn-rho1}
\begin{split}
L_{ij}(\bm\uprho):=&\;
\sum_{\mathsf{m}=1}^M\frac{r^2|\overline{\bm{\mathrm{m}}}_{\mathsf{m},ij}^\ell|^2}{r+V_{1,\mathsf{m}}(\uprho_{\mathsf{m}})}
+\frac{r^2|\overline{\bm{\mathrm{n}}}_{\mathsf{m},ij}^\ell|^2}{r+V_{3,\mathsf{m}}(\uprho_{\mathsf{m}})}
+\frac{(\uprho_{\mathsf{m}}-r\overline{\uprho}_{\mathsf{m},ij}^\ell)^2}{r}\\
&\;-2\bFF(\bm\uprho)
+\sum_{p=1}^R\frac{r^2|\overline{{\mathrm{s}}}_{p,ij}^\ell|^2}{r+V_{2,p}(\bm\uprho)}.
\end{split}
\end{align}
And the nonlinear update for the terminal densities 
$\bm{\uprho}_{T,i}^\ell=({\uprho}_{1,T,i}^\ell,\cdots, {\uprho}_{M,T,i}^\ell)\in \mathbb{R}_+^M$ 
takes the following form:
find $\bm{\uprho}_{T,i}^\ell\in\mathbb{R}_+^M$ that minimizes the function
\begin{align}
\label{eqn-rho2}
L_{T,i}(\bm\uprho):=&\;
2\bGG(\bm\uprho)+{|\bm\uprho-r\overline{\bm\uprho}_{T,i}^\ell|^2}/r.
\end{align}
Here the bar-values take the same form as in Step A3 of Algorithm~\ref{alg:1}, 
namely, 
\begin{align}
\label{barX}
\overline{\underline{\uui}}_{ij}^\ell:=(\overline{\bm\uprho}_{ij}^\ell,
\underline{\overline{\bm{\mathrm{m}}}}_{ij}^\ell,
\overline{\bm{\mathrm{s}}}_{ij}^\ell,
\overline{\underline{\bm{\mathrm{n}}}}_{ij}^\ell):=
\underline{\mathcal{D}}\bPphi_{ij}^\ell+\underline{\uui}_{ij}^{\ell}/r, 
\; 
\overline{\bm\uprho}_{T,i}^\ell:=
\bm{\uprho}_{T,i}^{\ell}/r-
\bm{\pphi}_h^\ell(T,\bm\xi_i).
\end{align}
The pointwise optimization problems in \eqref{eqn-phi1} and \eqref{eqn-rho2}
are $M$-dimensional problems. These problems are usually loosely coupled among the density components. We further reduce the computational cost by separately solving for each density component sequentially, which results in $M$ one-dimensional minimization problems per integration point. The positivity of the densities is guaranteed in our algorithm at the quadrature points.
After the densities $\bm{\uprho}_{ij}^\ell$ and $\bm{\uprho}_{T,i}^\ell$ have been computed, we then update the other physical variables on the quadrature point as follows:
\begin{subequations}
\label{ustarX}
\begin{alignat}{2}
\label{ustarX1}
\bm{\mathrm{m}}_{\mathsf{m},ij}^\ell=&\;\frac{rV_{1,\mathsf{m}}(\uprho_{\mathsf{m},ij}^\ell)}{r+V_{1,\mathsf{m}}(\uprho_{\mathsf{m},ij}^\ell)}
\overline{\bm{\mathrm{m}}}_{\mathsf{m},ij}^\ell,\;\;&&
\bm{\mathrm{n}}_{\mathsf{m},ij}^\ell=\;\frac{rV_{3,\mathsf{m}}(\uprho_{\mathsf{m},ij}^\ell)}{r+V_{3,\mathsf{m}}(\uprho_{\mathsf{m},ij}^\ell)}
\overline{\bm{\mathrm{n}}}_{\mathsf{m},ij}^\ell,\\
\label{ustarX2}
\bm{\mathrm{m}}_{\mathsf{m},ij}^{*,\ell}=&\;\frac{r}{r+V_{1,\mathsf{m}}(\uprho_{\mathsf{m},ij}^\ell)}
\overline{\bm{\mathrm{m}}}_{\mathsf{m},ij}^\ell,\;\;&&
\bm{\mathrm{n}}_{\mathsf{m},ij}^{*,\ell}=\;\frac{r}{r+V_{3,\mathsf{m}}(\uprho_{\mathsf{m},ij}^\ell)}
\overline{\bm{\mathrm{n}}}_{\mathsf{m},ij}^\ell,\\
\label{ustarX3}
{\bm\uprho}_{ij}^{*,\ell}=&\;\overline{\bm\uprho}_{ij}^{\ell}-{\bm\uprho}_{ij}^{\ell}/r,\;\;&&
{\bm\uprho}_{T,i}^{*,\ell}=\;\overline{\bm\uprho}_{T,i}^{\ell}-{\bm\uprho}_{T,i}^{\ell}/r,\\
\label{ustarX4}
\overline{{\mathrm{s}}}_{p,ij}^\ell=&\;\frac{rV_{2,p}(\bm\uprho_{ij}^\ell)}{r+V_{2,p}(\bm\uprho_{ij}^\ell)}
\overline{{\mathrm{s}}}_{p,ij}^\ell,\;\;&& 
{\mathrm{s}}_{p,ij}^{*,\ell}=\;\frac{r}{r+V_{2,p}(\bm\uprho_{ij}^\ell)}
\overline{{\mathrm{s}}}_{p,ij}^\ell.
\end{alignat}
\end{subequations}

For completeness, we collect the above procedures into the following algorithm.

\begin{algorithm}[H]
\caption{One iteration of (modified) ALG2  for \eqref{SY-scalarh3}.}
\label{alg:2}
\begin{algorithmic}[1]
\STATE Step A. For $i=1,\cdots, M$, solve the linear system \eqref{eqn-phi1} for the unknown $\pphi_{i,h}^\ell$, and solve the linear system \eqref{eqn-sigma1}
for the unknown $\ssigma_{i,h}^\ell$.
Interpolate the operator $\underline{\mathcal{D}}(\bPphi_h^\ell)$
on quadrature points, and compute the bar-values (on quadrature points) in \eqref{barX}.

\STATE Step B. For each quadrature point, sequentially compute the nonnegative density minimizers to \eqref{eqn-rho1} and \eqref{eqn-rho2}. 

\STATE Step C. Update the other variables on quadrature points according to \eqref{ustarX}.
\end{algorithmic}
\end{algorithm}

\section{Numerical results}\label{sec4}
In this section, we present one- and two-dimensional numerical results for the scalar scheme \eqref{SD-scalarh3}
and the system scheme \eqref{SY-scalarh3} using the finite element software MFEM \cite{MFEM}.
The spatial domain $\Omega=[0,1]^d$ is taken to be either a unit line segment ($d=1$) or a unit square ($d=2$). The terminal time is  $T=1$. 
Furthermore, we solve the mean field planning problem for all the numerical simulations, where the initial and terminal densities are prescribed. In this case, the terminal density is no longer an  unknown variable. Hence the scalar optimization problem \eqref{SD-scalarh3} simplifies as below
\begin{equation}\label{SD-scalarh3X}
\begin{split}
\sup_{\uu_h} \inf_{\Pphi_h,\uu_h^*}~& \stint{\HH^*(\uu_h^*)
+\uu_h\cdot(\mathcal{D}(\Pphi_h)-\uu_h^*)}
\\
 & \hspace{-0.2cm}+
 \sint{-\rrho^1\, \pphi_h(T,\cdot)
+\rrho^0\,\pphi_h(0,\cdot)}\\
&\hspace{-0.2cm}+
\frac{r}{2}\stint{(\mathcal{D}(\Pphi_h)-\uu_h^*)\cdot(\mathcal{D}(\Pphi_h)-\uu_h^*)},
\end{split}
\end{equation} 
and the system optimization problem \eqref{SY-scalarh3} simplifies as below
\begin{equation}\label{SY-scalarh3X}
\begin{split}
\sup_{\buu_h} \inf_{\bPphi_h,\buu_h^*}~& \stint{\bHH^*(\buu_h^*)
+\buu_h\cdot(\underline{\mathcal{D}}(\bPphi_h)-\buu_h^*)}
\\
 & \hspace{-0.2cm}+
 \sint{-\brrho^1\cdot \bpphi_h(T,\cdot)+\brrho^0\cdot\bpphi_h(0,\cdot)}\\
&\hspace{-0.2cm}+
\frac{r}{2}\stint{(\underline{\mathcal{D}}(\bPphi_h)-\buu_h^*)\cdot(\underline{\mathcal{D}}(\bPphi_h)-\buu_h^*)}.
\end{split}
\end{equation} 
We take the augmented Lagrangian parameter $r=1$ for all the simulations.

\subsection{Scalar MFC for reaction-diffusion: $\beta=0$}
\label{ex1}
We first consider a planning problem for  scalar MFC \eqref{smfc} with regularization parameter $\beta=0$ (no Fisher information functional). 
With $\beta=0$, we do not compute the $\nn_h, \nn_h^*$ and $\ssigma_h$ variables in the optimization problem \eqref{SD-scalarh3X}, since they are always {\it zero} and are decoupled from the other variables. 

We solve the problem \eqref{SD-scalarh3X} (with $\beta=0$) on 
 $\Omega_T=[0, 1]\times \Omega$ with $\Omega=[0,1]^d$ for $d=1,2$.
The initial (resp.\ terminal) densities are chosen to be: 
\begin{equation*}
\rho^0(x)=\exp(-50|x-x_A|^2),\quad \rho^1(x)=\exp(-50|x-x_B|^2), \forall x=(x_1,\cdots,x_d)\in\Omega.
\end{equation*}
Here $x_A=0.25$, $x_B=0.75$ for $d=1$, and  $x_A=(0.25, 0.25)$, $x_B=(0.75,0.75)$ for $d=2$. 
We fix $V_1(\rho)=\rho$ and potential function 
\begin{equation}
\label{ener1}
F(t,x,\rho) = -(0.01\rho\log(\rho)+0.4\rho\cos(4\pi t)\prod_{i=1}^d\cos(4\pi x_i)).
\end{equation}
Moreover, we take the following four choices of $V_2(\rho)$ to highlight different reaction effects of the model:
\begin{align}
\label{V2s}
 \begin{cases}
  \text{Case 1}: V_2(\rho)=0 \\[.2ex]
  \text{Case 2}: V_2(\rho)=20 \\[.2ex]
 \text{Case 3}: V_2(\rho)=20\rho \\[.2ex]
 \text{Case 4}: V_2(\rho)= 20\frac{\rho-1}{\log(\rho)}.
\end{cases}
\end{align}
The 1D results are computed on a $64\times 64$ uniform rectangular space-time mesh, and 
the 2D results are computed on a $16\times 64\times 64$ uniform cubic space-time mesh.
We use polynomial degree $k=4$, and apply 400 ALG iterations for all the simulations. 
The density contours for the $d=1$
on the space-time 2D domain $\Omega_T$ are shown in Figure~\ref{fig:den-ex1-2D}.
The snapshots of density contours on $\Omega=[0,1]^2$ at 
different times for the $d=2$ are shown in 
Figure~\ref{fig:den-ex1-3D}. We find the dynamics with and without reaction terms are completely different. The reaction function in Figure~\ref{fig:den-ex1-2D} (b) and (d) shows that the reaction mobility function $V_2$ dominates the path between $\rho^0$ and $\rho^1$. While the transport mobility function $V_1$ dominates the path between $\rho^0$ and $\rho^1$ in Figure~\ref{fig:den-ex1-2D} (a) and (c). In Figure~\ref{fig:den-ex1-3D}, we observe a different pattern formulation for various choices of general nonlinear reaction mobility function $V_2$. This behaves very differently from the classical optimal transport problem with $V_2=0$.

\begin{figure}[H]
\centering
\subfigure[$V_2(\rho) = 0$.]
{
\includegraphics[width=0.23\textwidth]{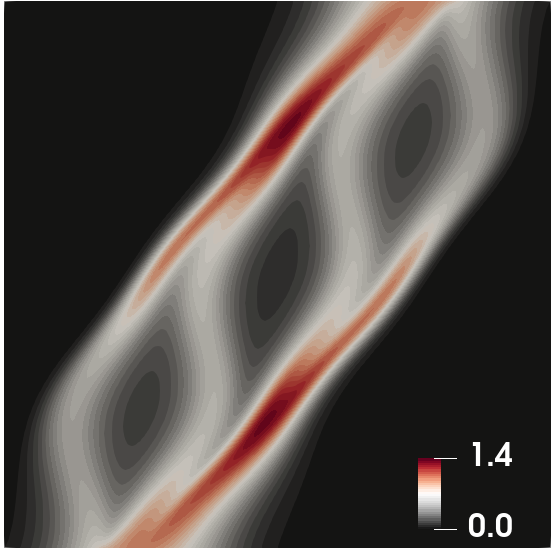}
}
\subfigure[$V_2(\rho) = 20$.]
{\includegraphics[width=0.23\textwidth]{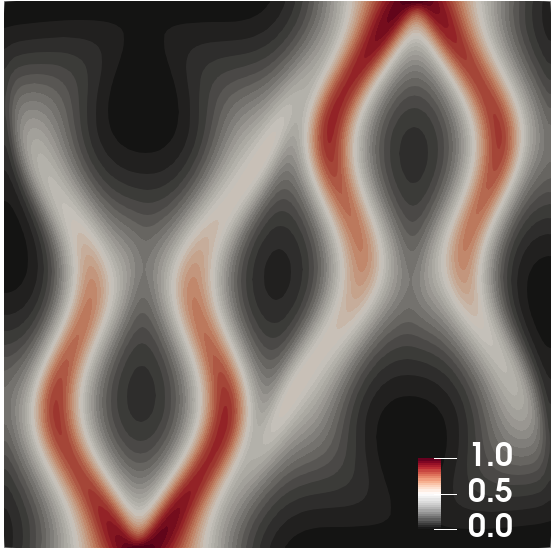}
}
\subfigure[$V_2(\rho) = 20\rho$.]
{\includegraphics[width=0.23\textwidth]{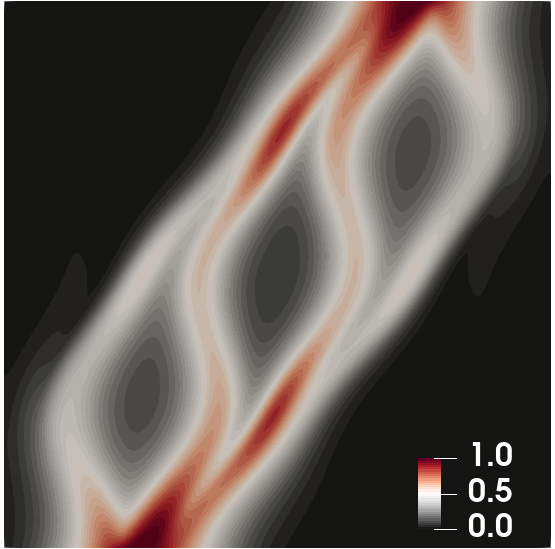}
}
\subfigure[$V_2(\rho) = 20\frac{\rho-1}{\log(\rho)}$.]
{\includegraphics[width=0.23\textwidth]{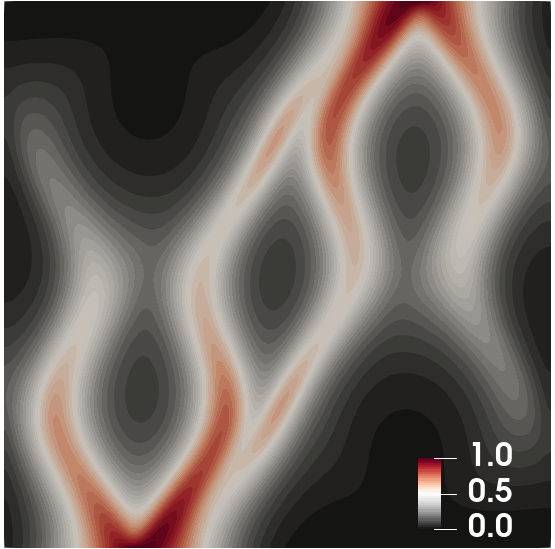}
}
\caption{Example \ref{ex1}. 
Snapshots of density contour on the space-time domain $\Omega_T=[0,1]^2$.
The vertical axis represents time.
}
\label{fig:den-ex1-2D}
\end{figure}

\begin{figure}[H]
\centering
\subfigure[$V_2(\rho) = 0$.]
{
\includegraphics[width=0.192\textwidth]{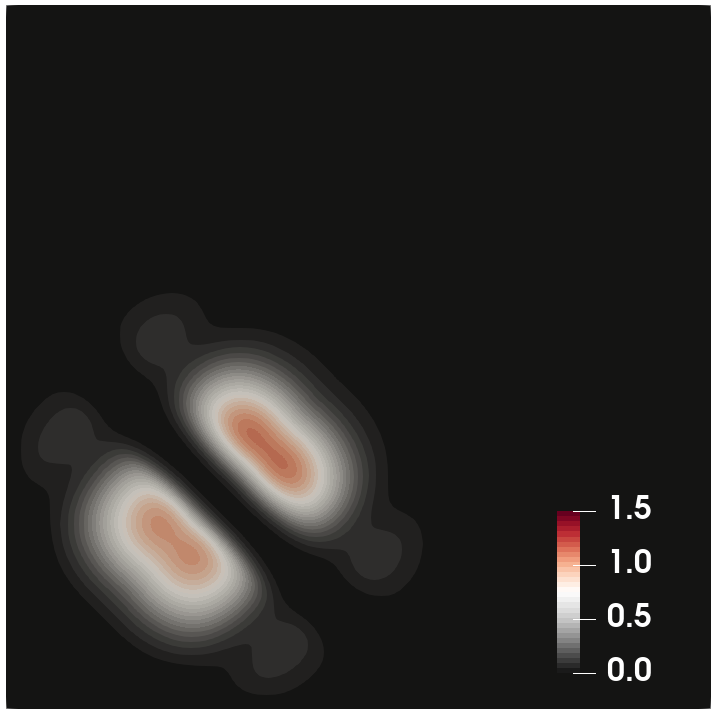}
\includegraphics[width=0.192\textwidth]{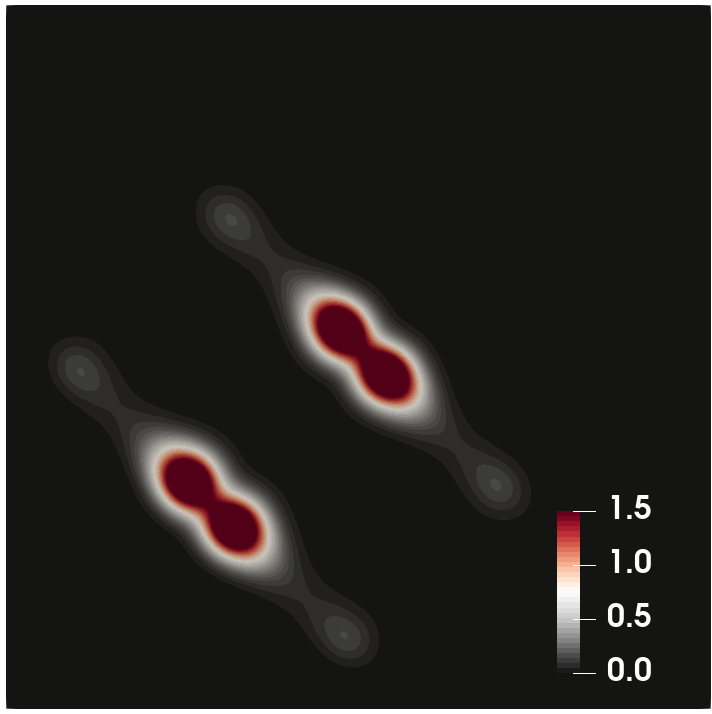}
\includegraphics[width=0.192\textwidth]{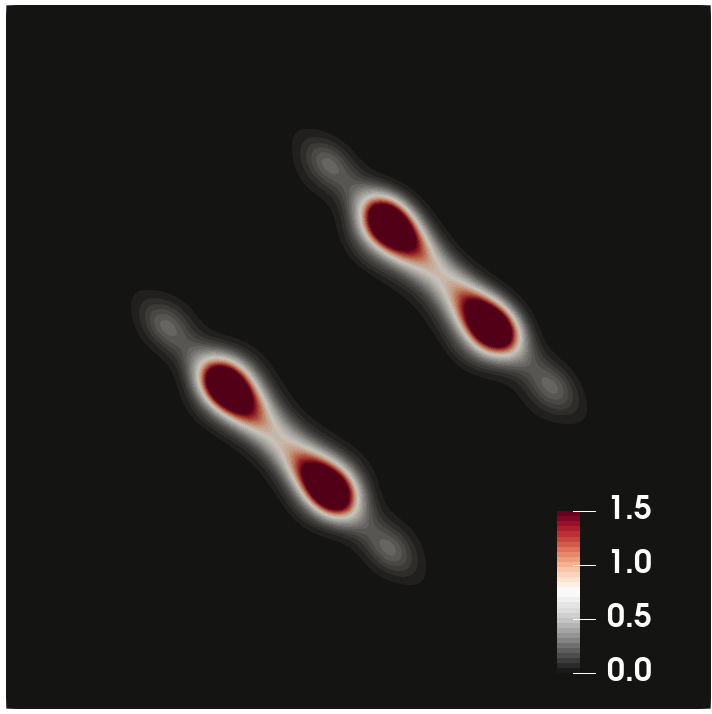}
\includegraphics[width=0.192\textwidth]{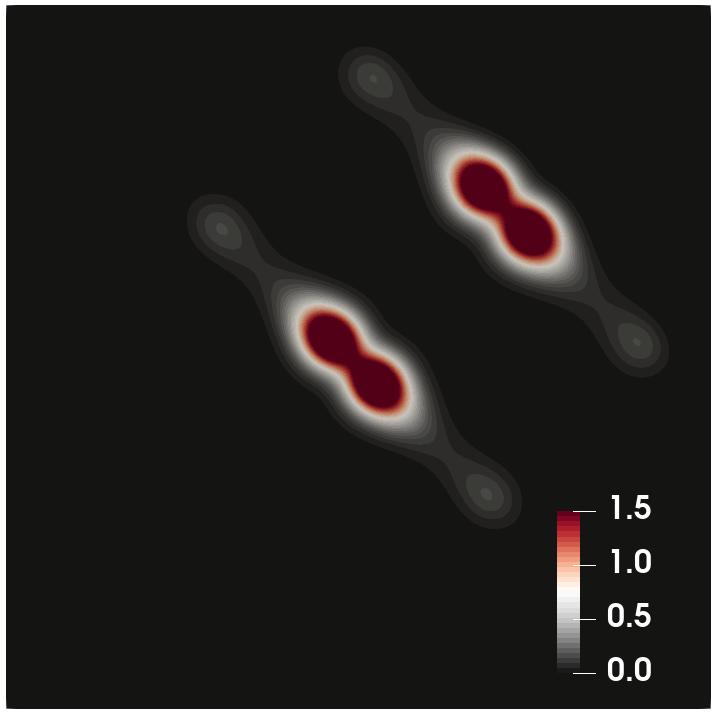}
\includegraphics[width=0.192\textwidth]{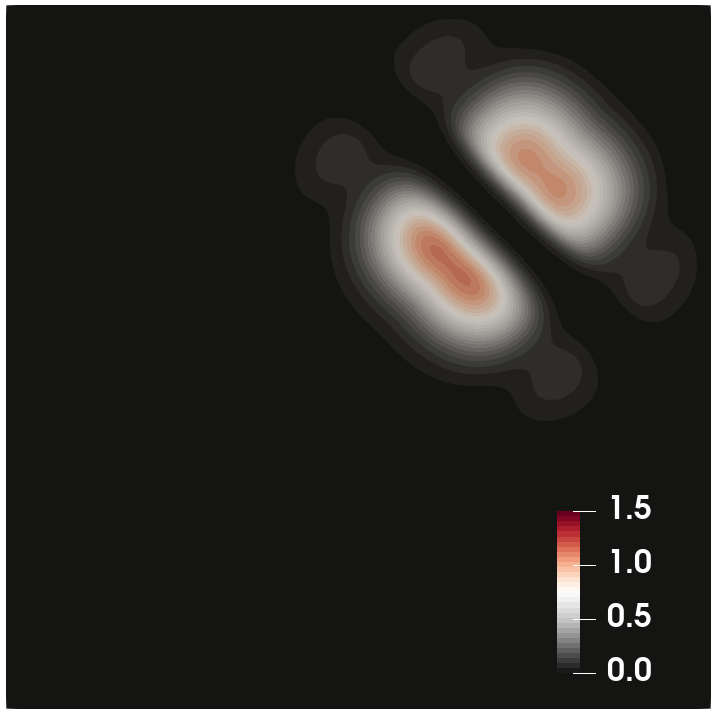}
}
\subfigure[$V_2(\rho) = 20$.]
{
\includegraphics[width=0.192\textwidth]{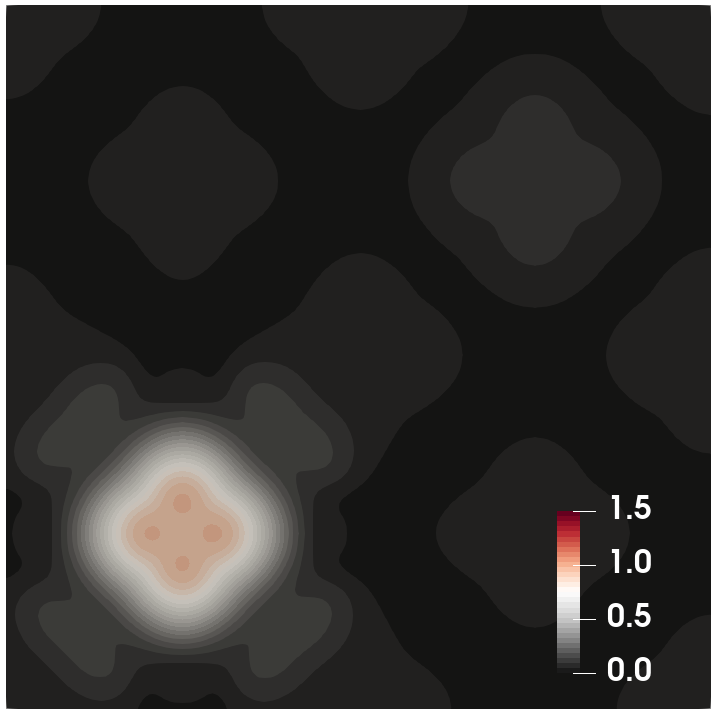}
\includegraphics[width=0.192\textwidth]{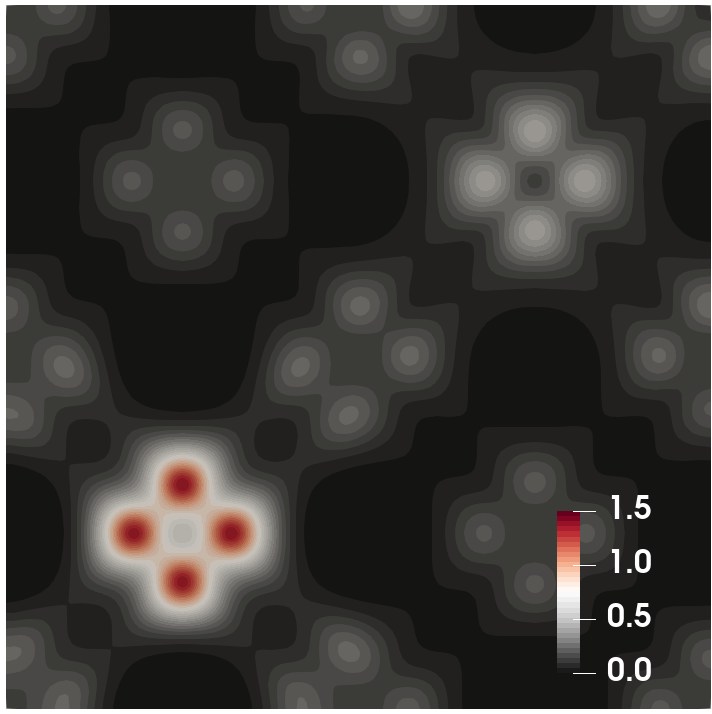}
\includegraphics[width=0.192\textwidth]{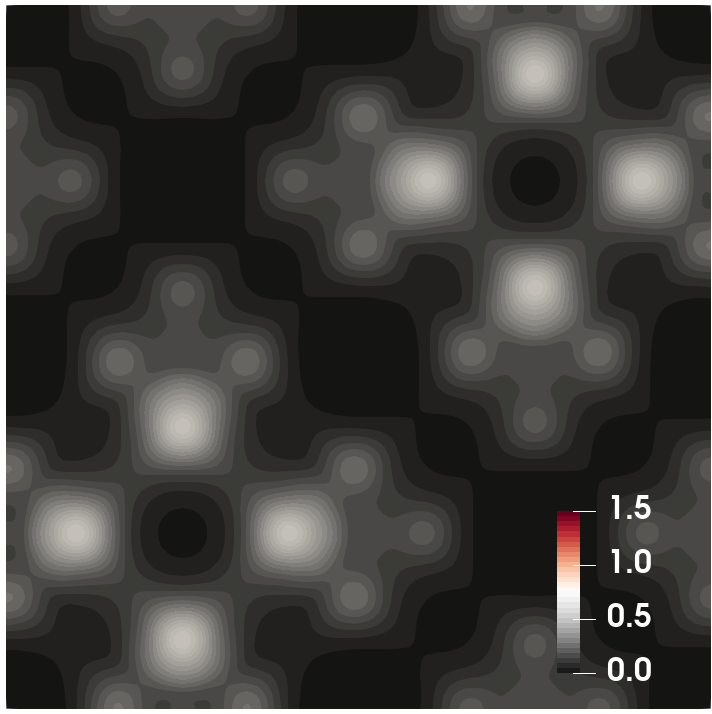}
\includegraphics[width=0.192\textwidth]{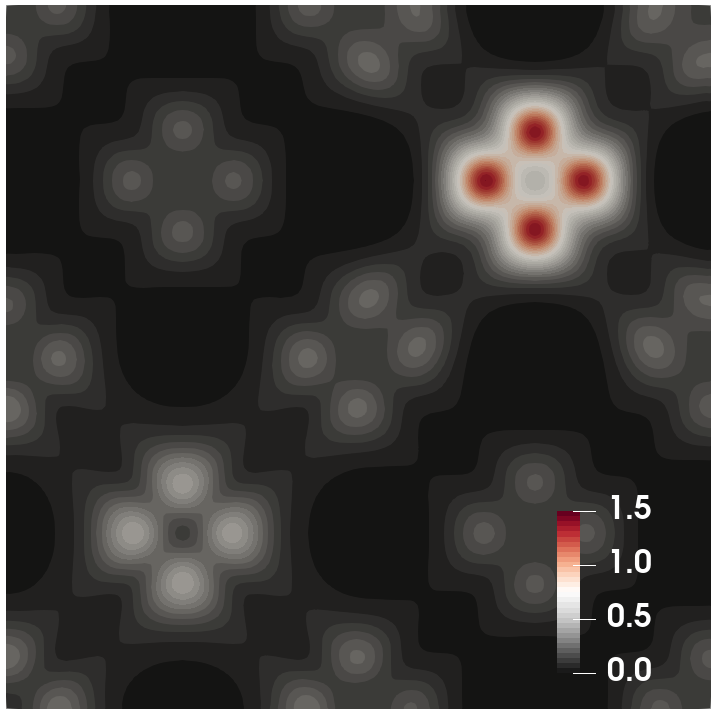}
\includegraphics[width=0.192\textwidth]{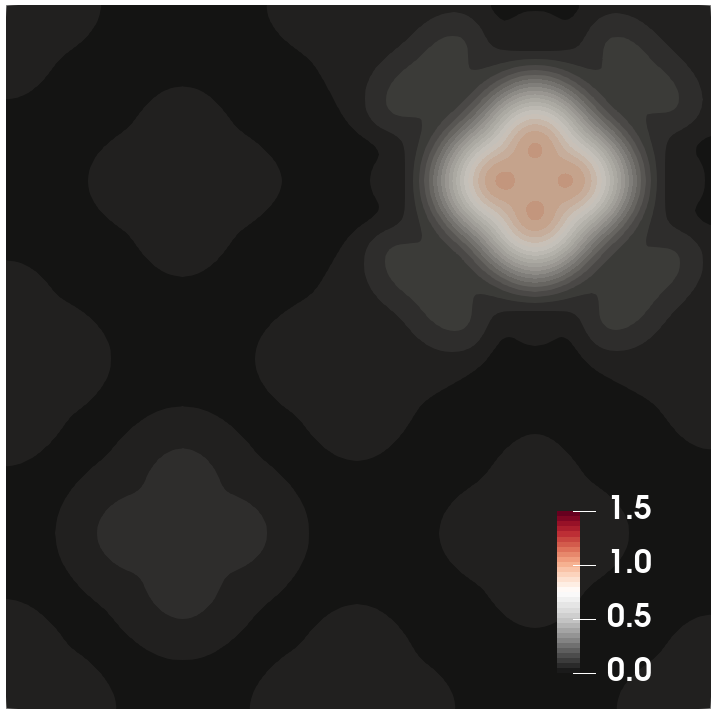}
}
\subfigure[$V_2(\rho) = 20\rho$.]
{
\includegraphics[width=0.192\textwidth]{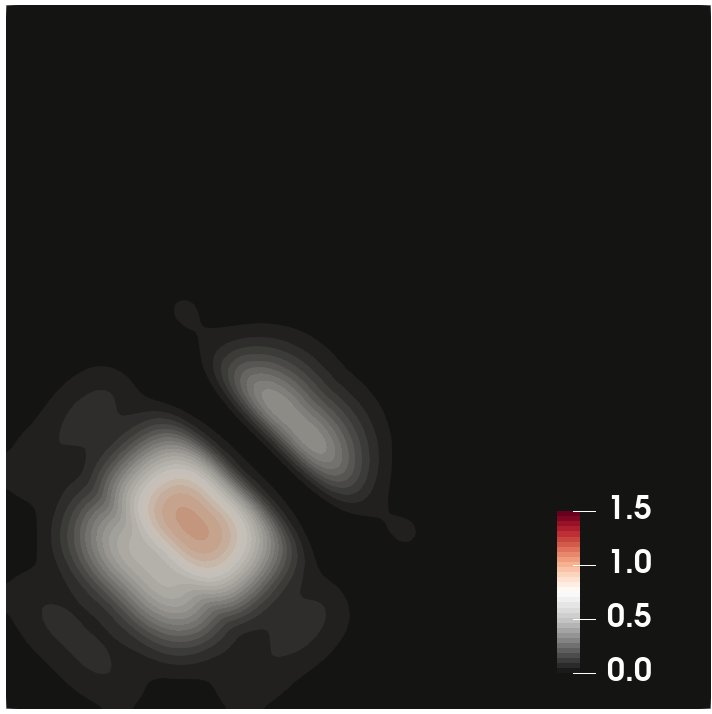}
\includegraphics[width=0.192\textwidth]{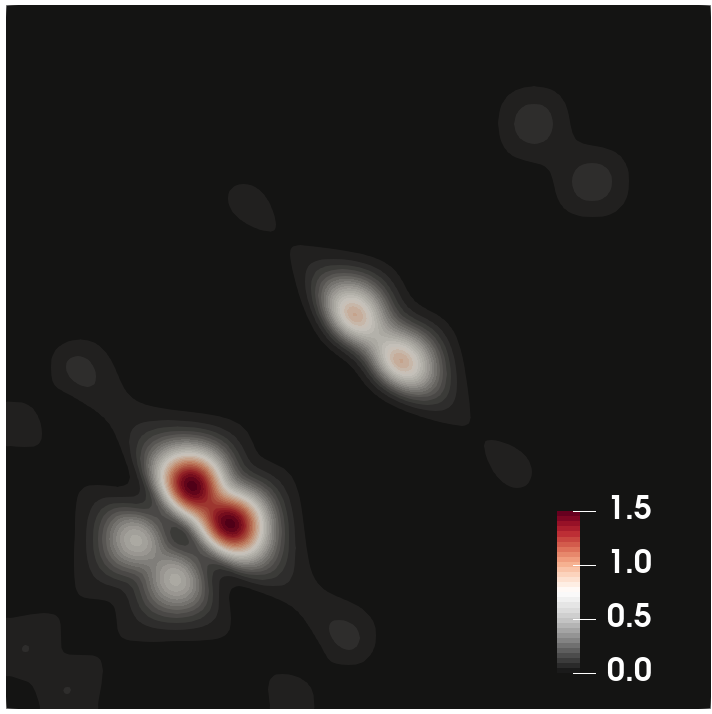}
\includegraphics[width=0.192\textwidth]{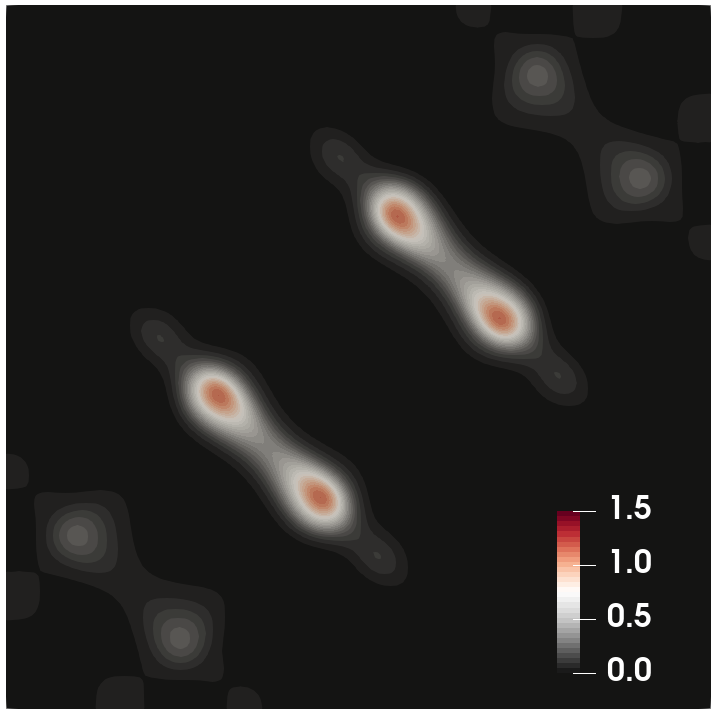}
\includegraphics[width=0.192\textwidth]{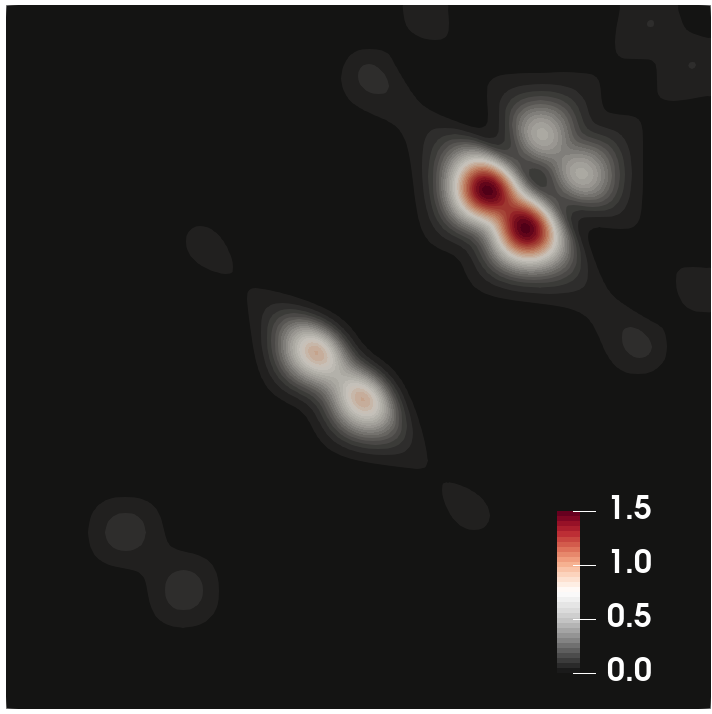}
\includegraphics[width=0.192\textwidth]{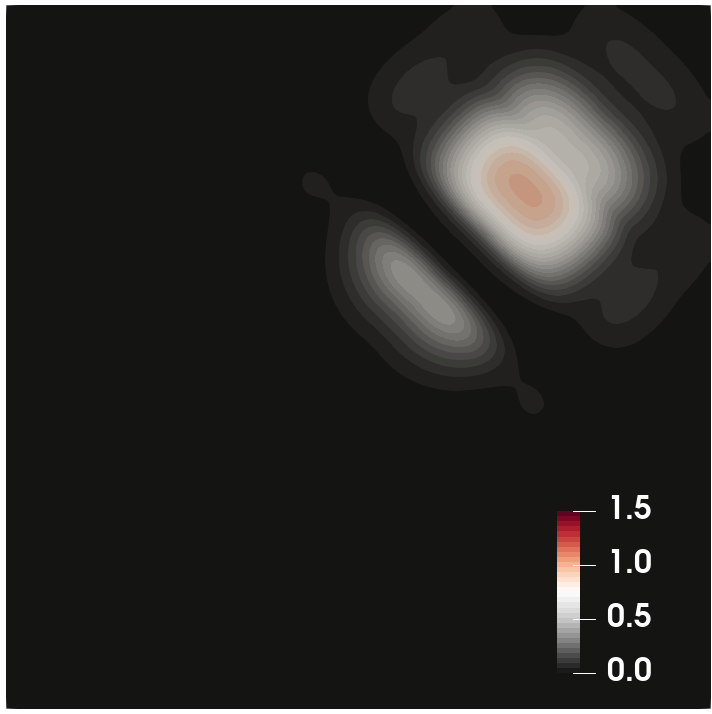}
}
\subfigure[$V_2(\rho) = 20\frac{\rho-1}{\log(\rho)}$.]
{
\includegraphics[width=0.192\textwidth]{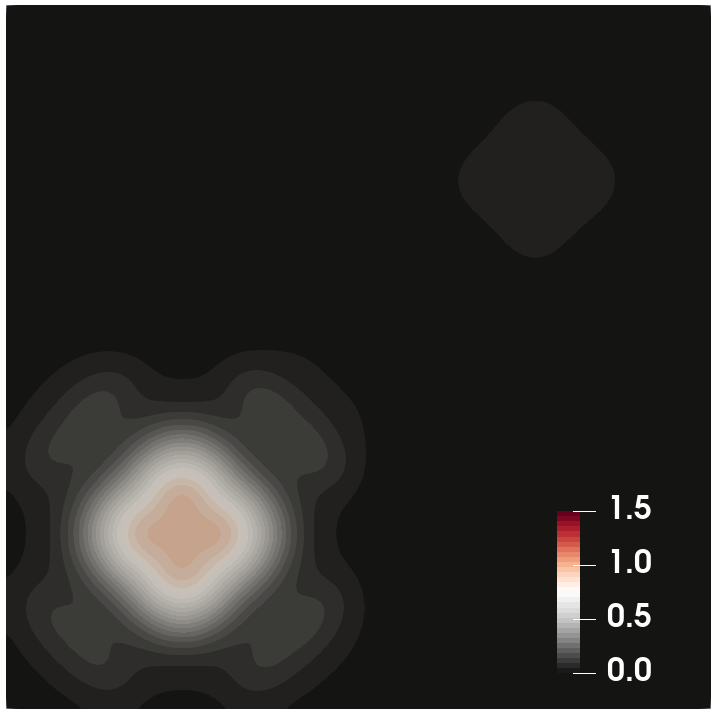}
\includegraphics[width=0.192\textwidth]{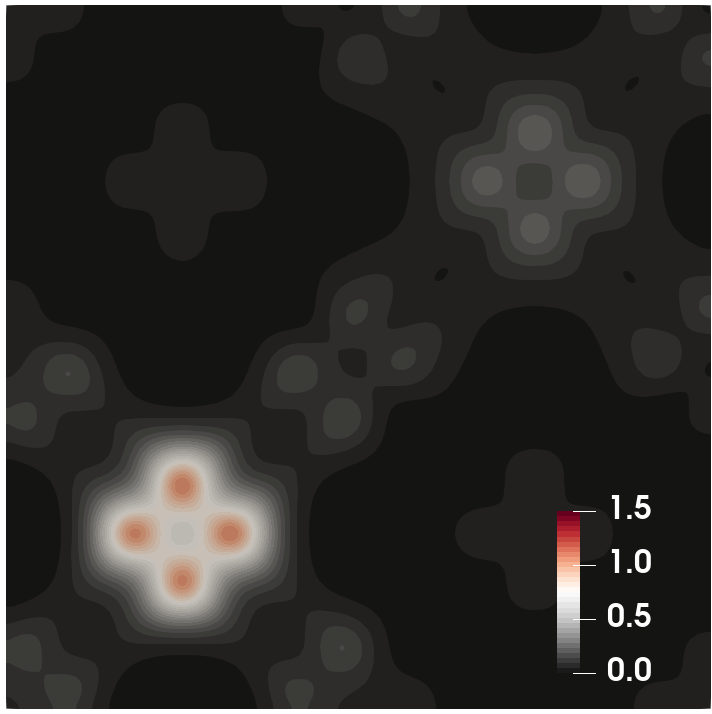}
\includegraphics[width=0.192\textwidth]{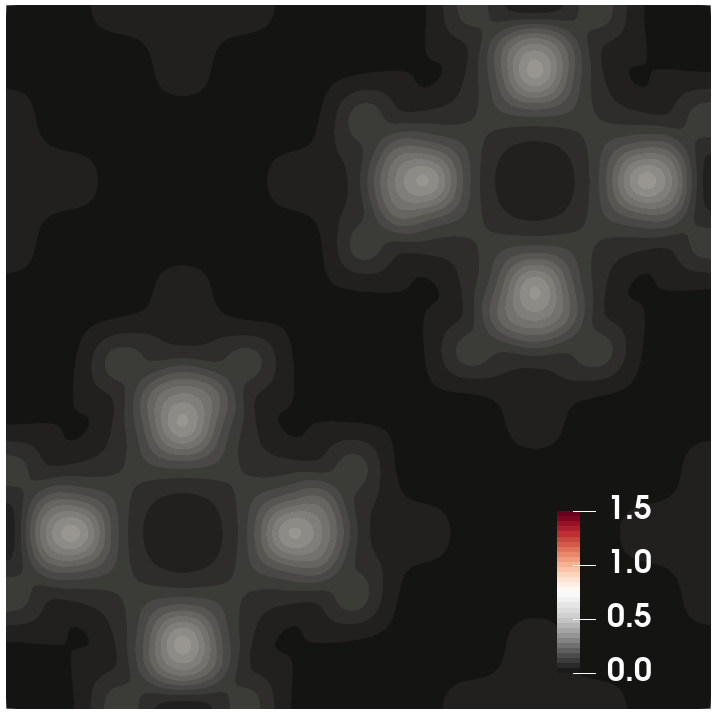}
\includegraphics[width=0.192\textwidth]{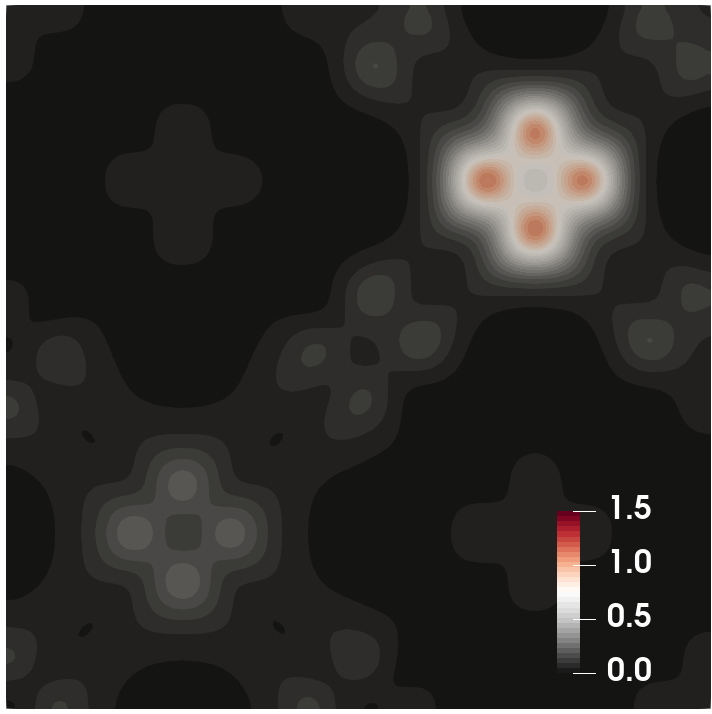}
\includegraphics[width=0.192\textwidth]{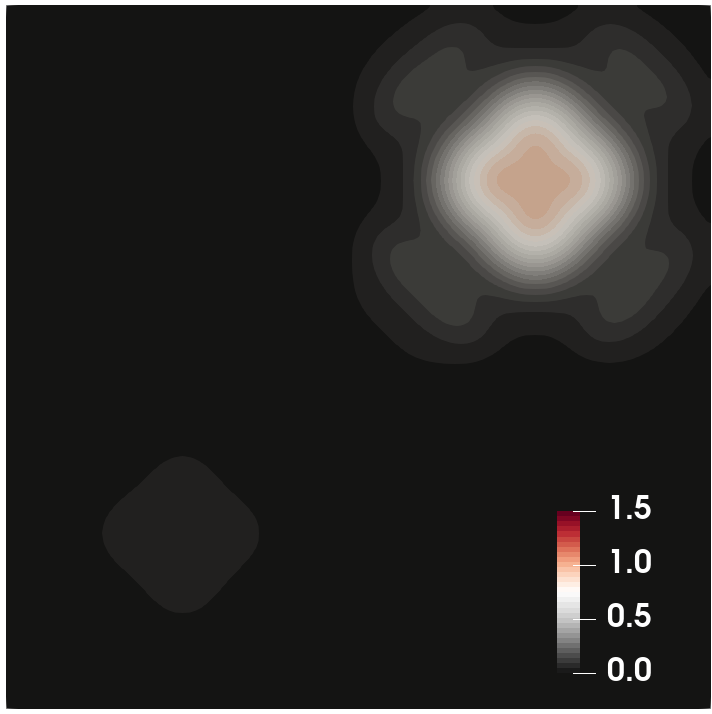}
}
\caption{Example \ref{ex1}. 
Snapshots of density contour at 
$t=$ 0.1,0.3,0.5, 0.7, 0.9 (left to right).
}
\label{fig:den-ex1-3D}
\end{figure}

\subsection{Scalar MFC for reaction-diffusion: effect of  $\beta$}
\label{ex2}
Here we use a similar setup as in Example \ref{ex1}, but study the effect of the regularization parameter $\beta$. 
We take take mobility functions $V_1(\rho) = \rho$, $V_2(\rho) = 20$, and potential function 
\begin{align}
\label{enerF1}
F(t,x,\rho) = -(0.005\rho\log(\rho)+0.4\rho\cos(4\pi t)\prod_{i=1}^d\cos(4\pi x_i)).
\end{align}
We further take the mobility function 
$V_3(\rho) = \rho$, and vary $\beta$, namely, 
$\beta=0$, 
$\beta=0.005$, and  
$\beta=0.01$.
The same discretization as in Example \ref{ex1} is used. 
The density contours for $d=1$
on the space-time domain $\Omega_T$ are shown in Figure~\ref{fig:den-ex2-2D}.
The snapshots of density contours at 
different times for $d=2$ are shown in 
Figure~\ref{fig:den-ex2-3D}. In numerical examples, we show that increasing regularization $\beta$ leads to a more diffusive density evolution. 

\begin{figure}[H]
\centering
\subfigure[$\beta=0$]{
\includegraphics[width=0.23\textwidth]{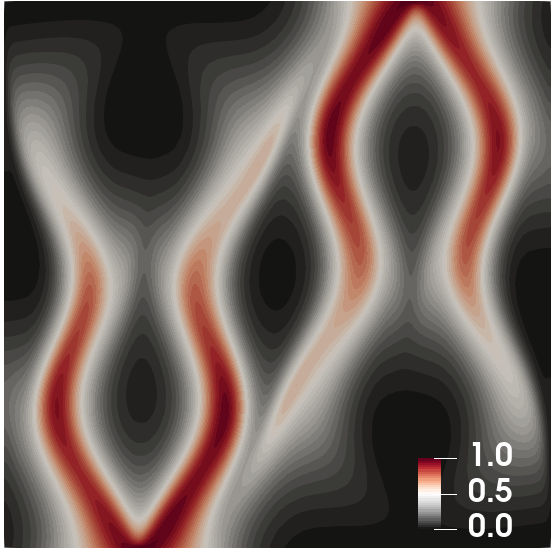}}
\subfigure[$\beta=0.005$]{
\includegraphics[width=0.23\textwidth]{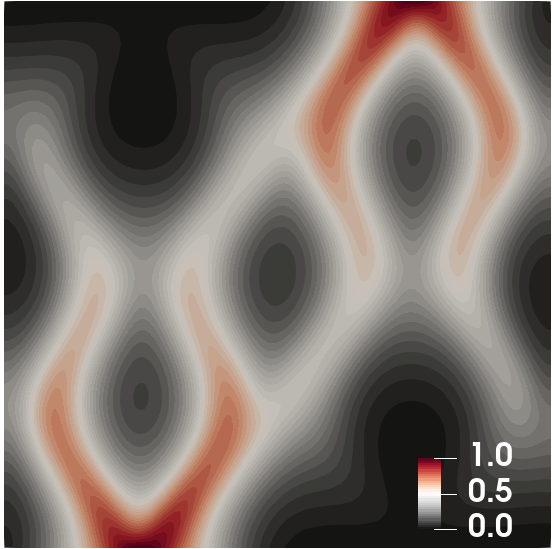}}
\subfigure[$\beta=0.01$]{
\includegraphics[width=0.23\textwidth]{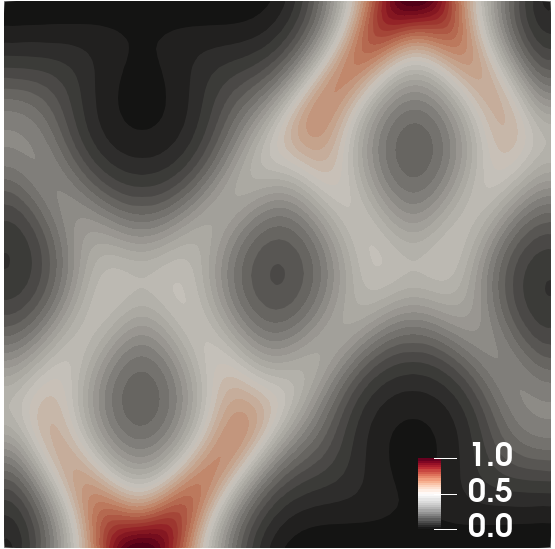}
}
\caption{Example \ref{ex2}. 
Snapshots of density contour on  $\Omega_T=[0,1]^2$.
}
\label{fig:den-ex2-2D}
\end{figure}

\begin{figure}[H]
\centering
\subfigure[$V_2(\rho) = 20, \beta=0$. (no regularization)]
{
\includegraphics[width=0.192\textwidth]{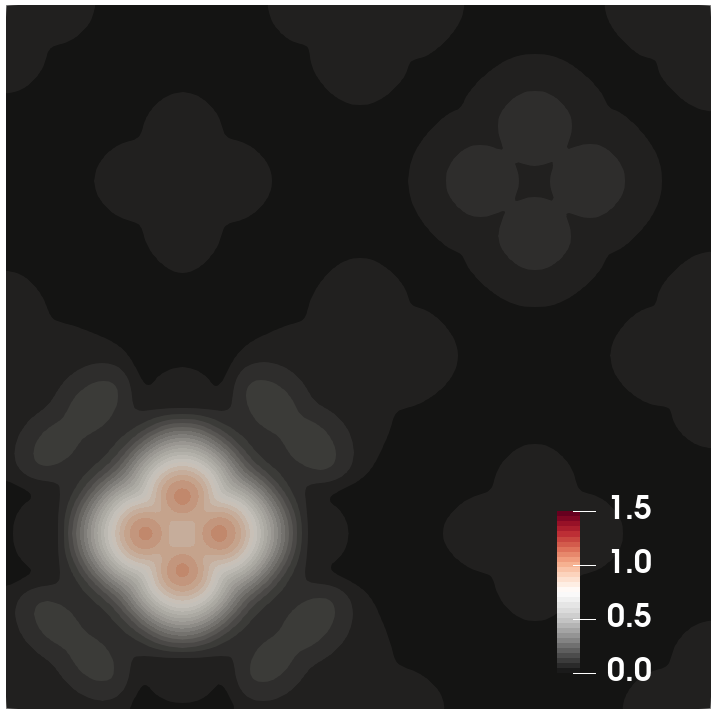}
\includegraphics[width=0.192\textwidth]{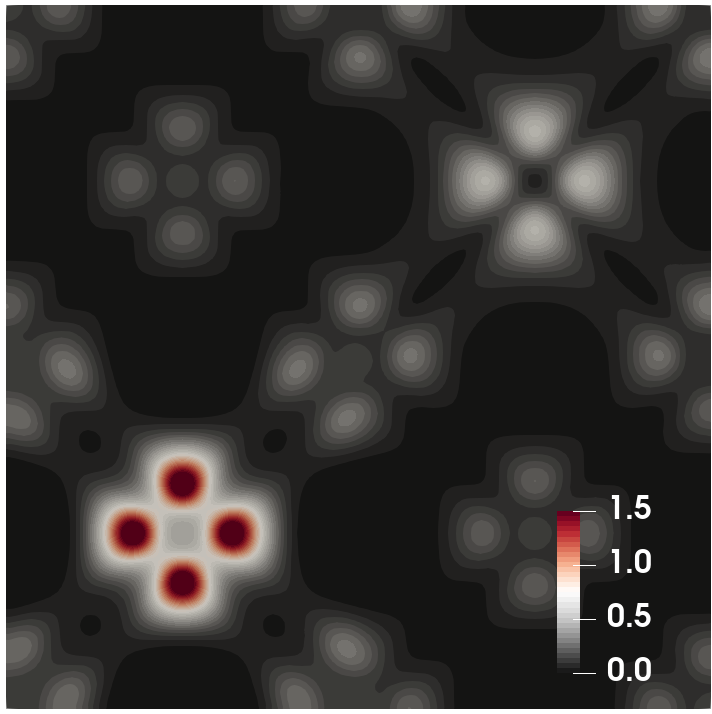}
\includegraphics[width=0.192\textwidth]{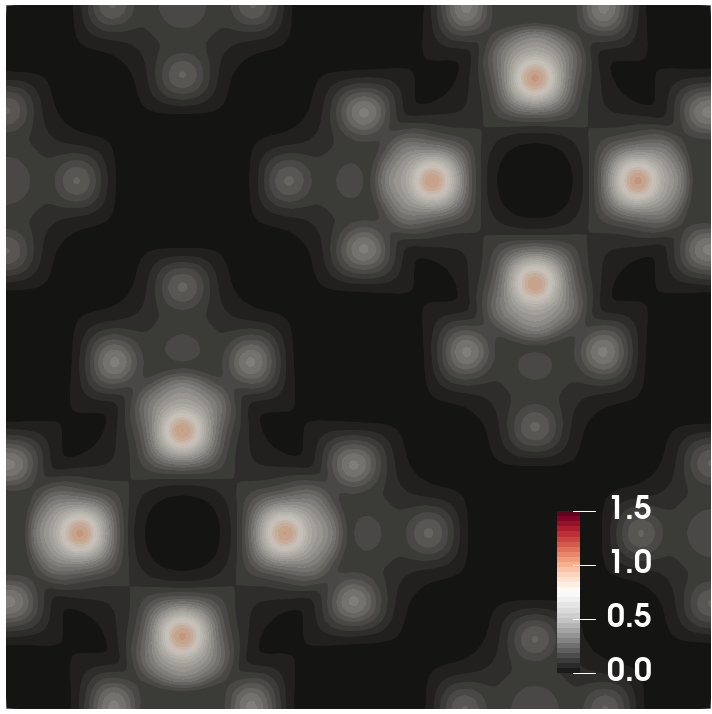}
\includegraphics[width=0.192\textwidth]{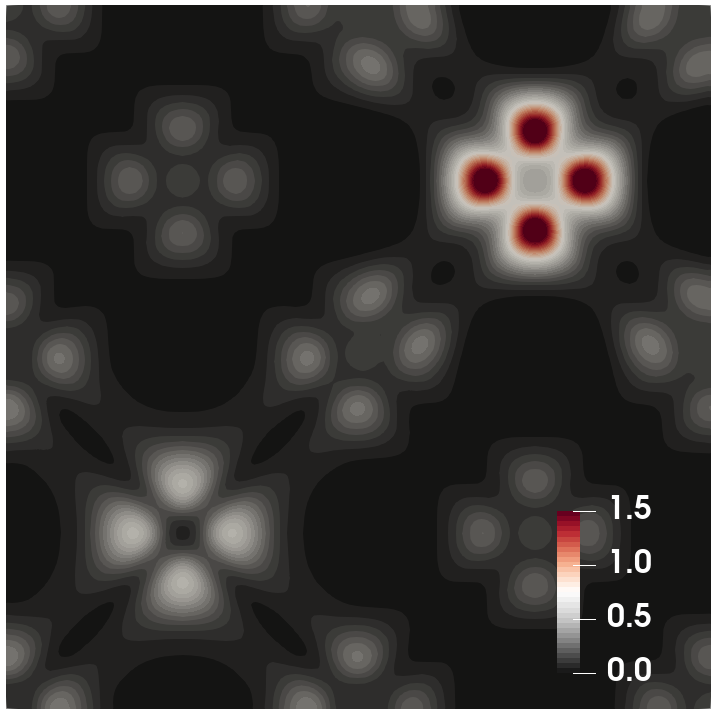}
\includegraphics[width=0.192\textwidth]{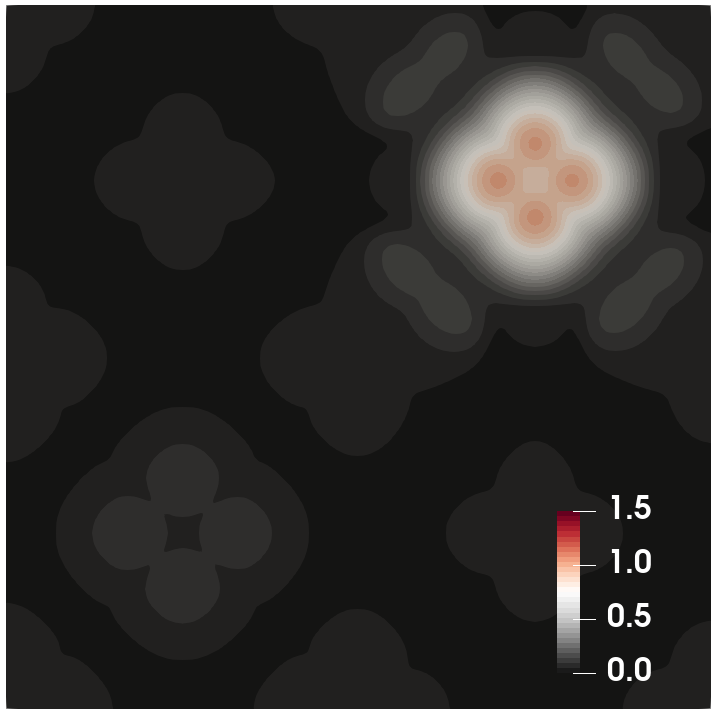}
}
\subfigure[$\beta=0.005$. (weak regularization)]
{
\includegraphics[width=0.192\textwidth]{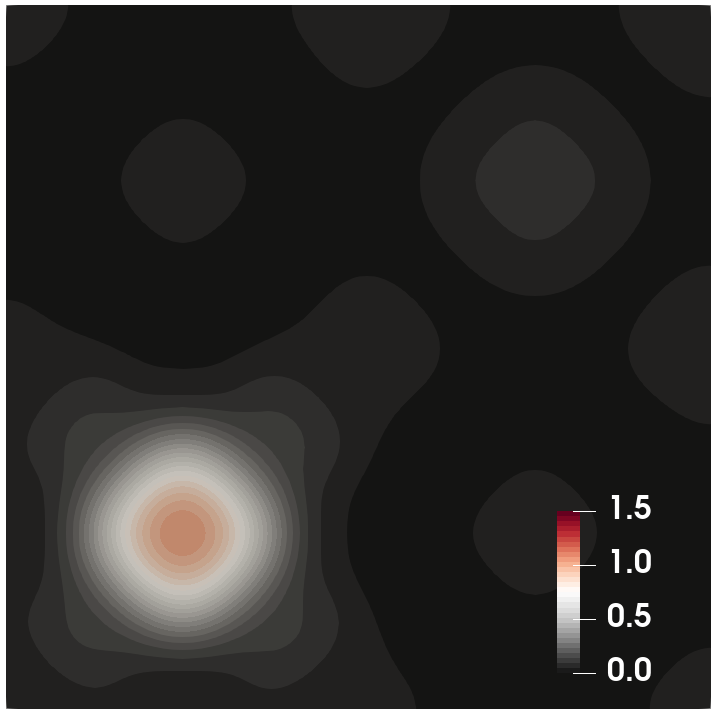}
\includegraphics[width=0.192\textwidth]{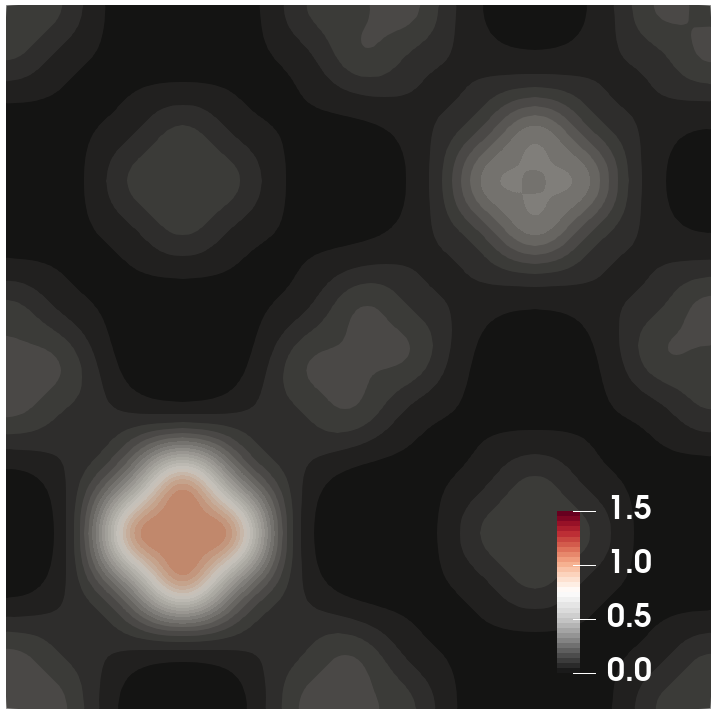}
\includegraphics[width=0.192\textwidth]{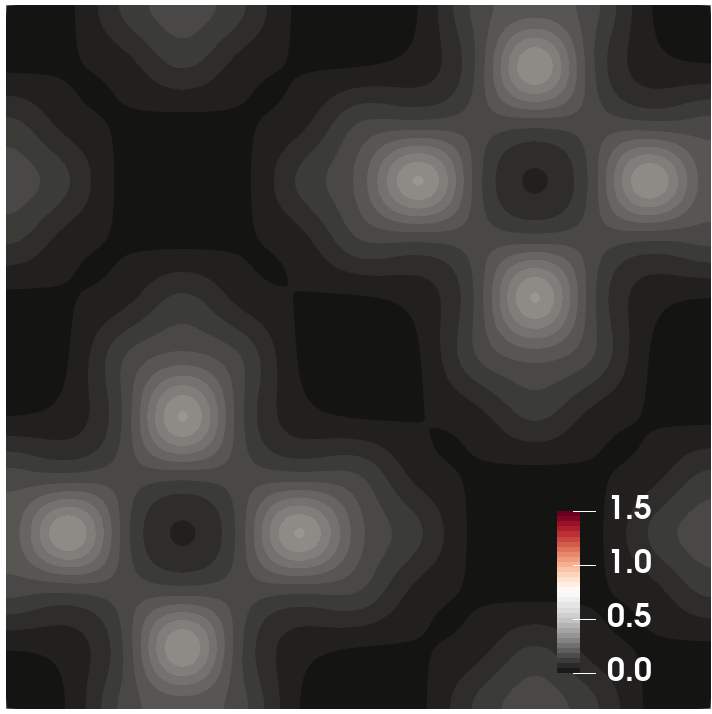}
\includegraphics[width=0.192\textwidth]{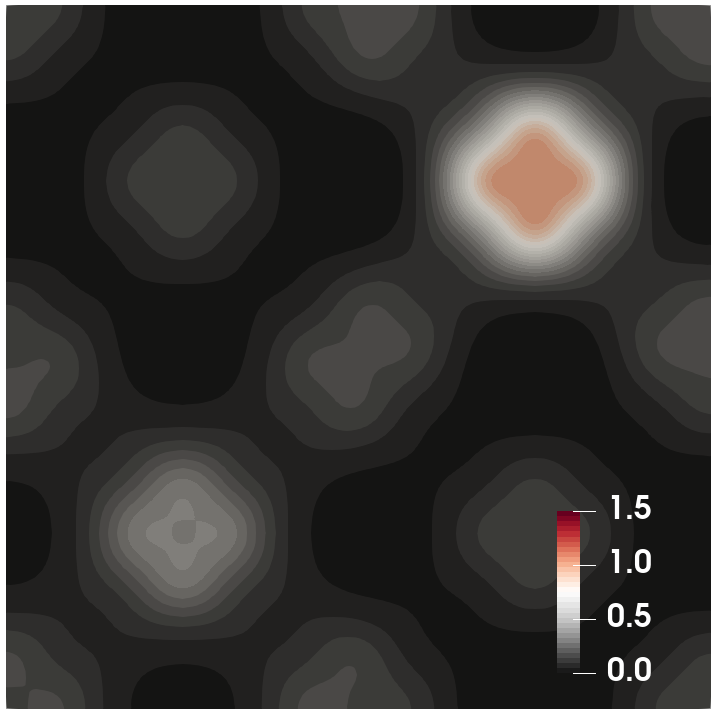}
\includegraphics[width=0.192\textwidth]{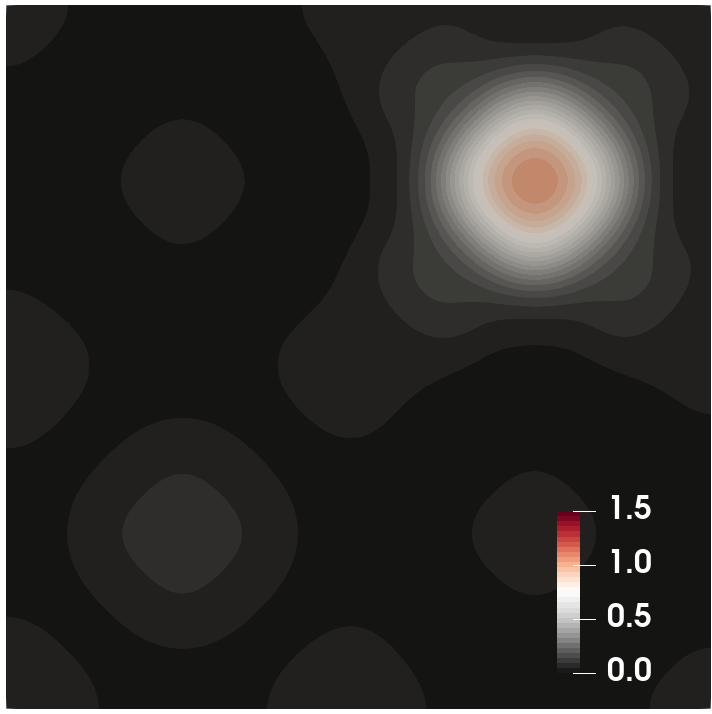}
}
\subfigure[$\beta=0.01$. (strong regularization)]
{
\includegraphics[width=0.192\textwidth]{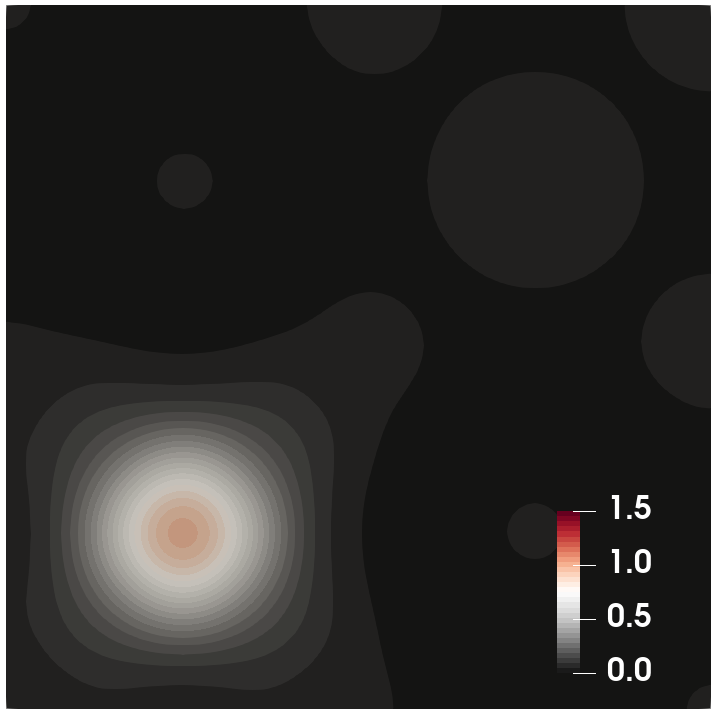}
\includegraphics[width=0.192\textwidth]{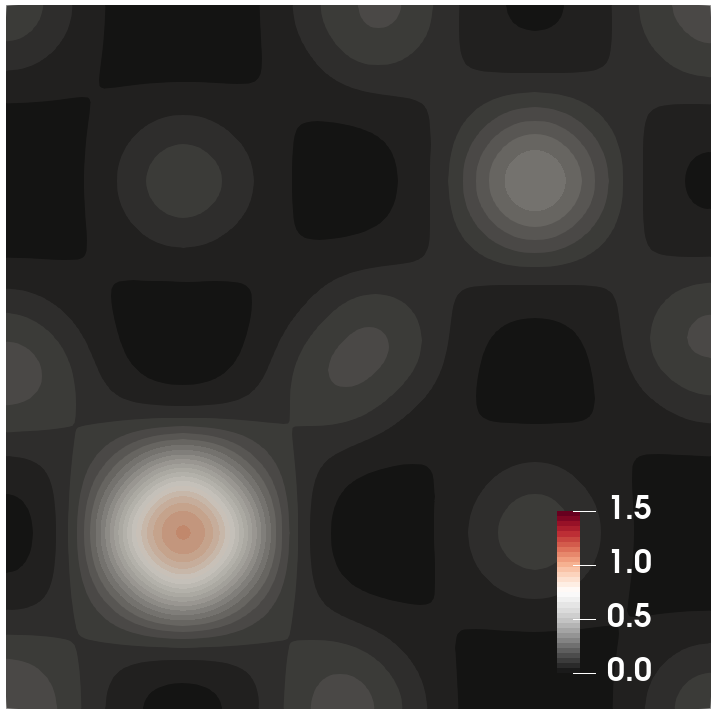}
\includegraphics[width=0.192\textwidth]{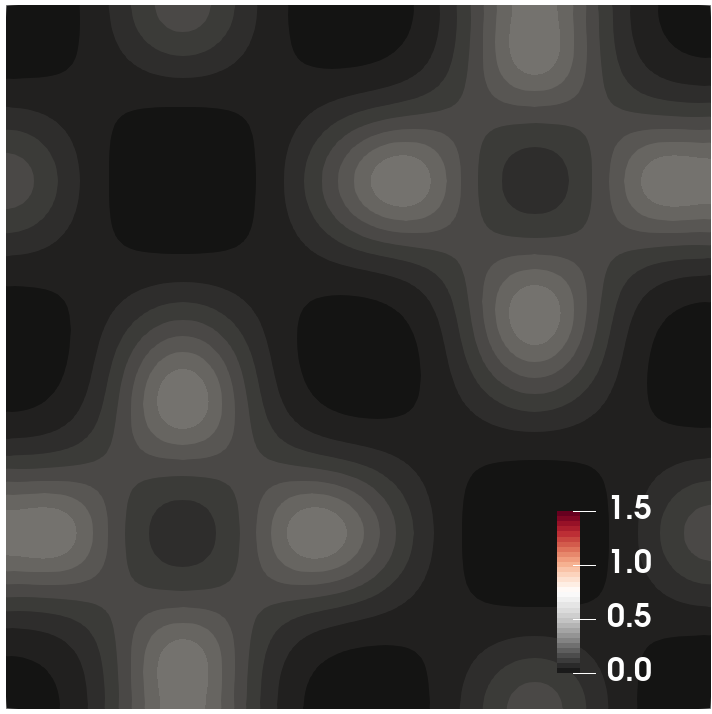}
\includegraphics[width=0.192\textwidth]{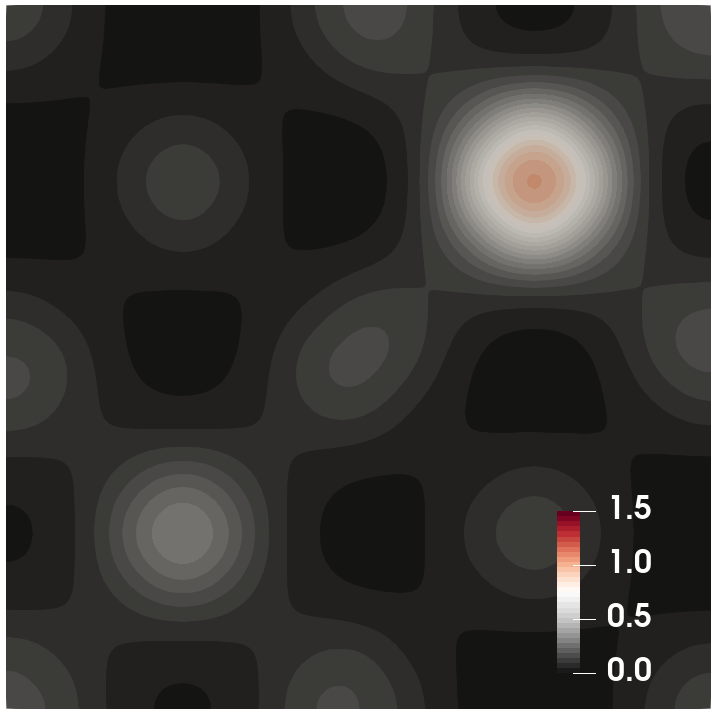}
\includegraphics[width=0.192\textwidth]{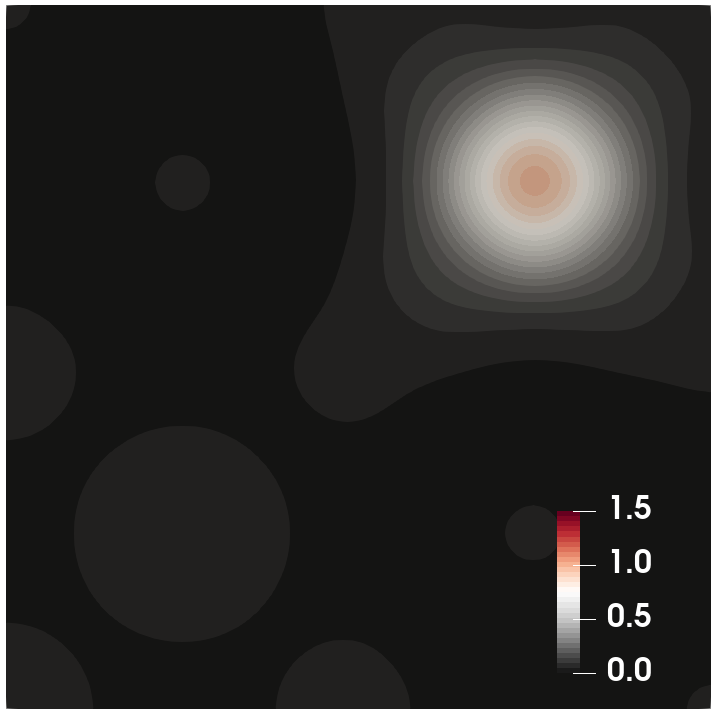}
}
\caption{Example \ref{ex2}. 
Snapshots of density contour at 
$t=$ 0.1,0.3,0.5,0.7,0.9 (left to right).
}
\label{fig:den-ex2-3D}
\end{figure}

\subsection{System MFC for reaction-diffusion ($M=2$, $R=1$)}
\label{ex3}
In this example, we consider a system model \eqref{vmfc} with $M=2$ species and $R=1$ reaction.
The initial/terminal densities for each component are given as:
\begin{align*}
\rho_1(0, x) = \exp(-50|x-x_A|^2), \;\;
\rho_1(1, x) = \exp(-50|x-x_B|^2), \\
\rho_2(0,x) = \exp(-50|x-x_B|^2), \;\;
\rho_2(1,x) = \exp(-50|x-x_A|^2).
\end{align*}
We take the mobilities $V_{1,i}(\rho_i)=V_{3,i}(\rho_i)=\rho_i$ for $i=1,2$, 
and the mobility
\[
V_{2,1}(\rho_1, \rho_2)=20\frac{\rho_1-\rho_2}{\log(\rho_1)-\log(\rho_2)}.
\]
The potential function is given as follows
\begin{align*}
\bm F(t,x,\bmr) = &\;- \Big(0.01\rho_1\log(\rho_1)
+0.4\rho_1\cos(4\pi t)\prod_{i=1}^d\cos(4\pi x_i)\\
&\;\hspace{1cm} +0.005\rho_2\log(\rho_2)\Big)
\end{align*}
The potential function for the first component has a drift and an entropy term in $\rho_1$, while that for the second component only has a smaller entropy term in $\rho_2$. We use the same discretization as in the previous two examples.
The density contours for the 1D results ($d=1$) are shown in Figure~\ref{fig:den-ex3-2D} for  $\beta=0$, 
$\beta=0.005$, and 
$\beta=0.01$.
The snapshots of density contours at 
different times for the 2D results ($d=2$) are shown in 
Figure~\ref{fig:den-ex3-3D}
for $\beta=0$ and 
$\beta=0.01$. Without reaction mobility $V_2$, the second component density path $\rho_2$ will be similar to a Gaussian translation. 
It is clear that the reaction mobility completely changes the density evolution. Moreover, increasing the regularization parameter $\beta$ leads to a diffusive density path, as expected.

\begin{figure}[H]
\centering
\subfigure[$\rho_1$. Left: $\beta=0$, middle: $\beta=0.005$, right: $\beta=0.01$.]
{
\includegraphics[width=0.23\textwidth]{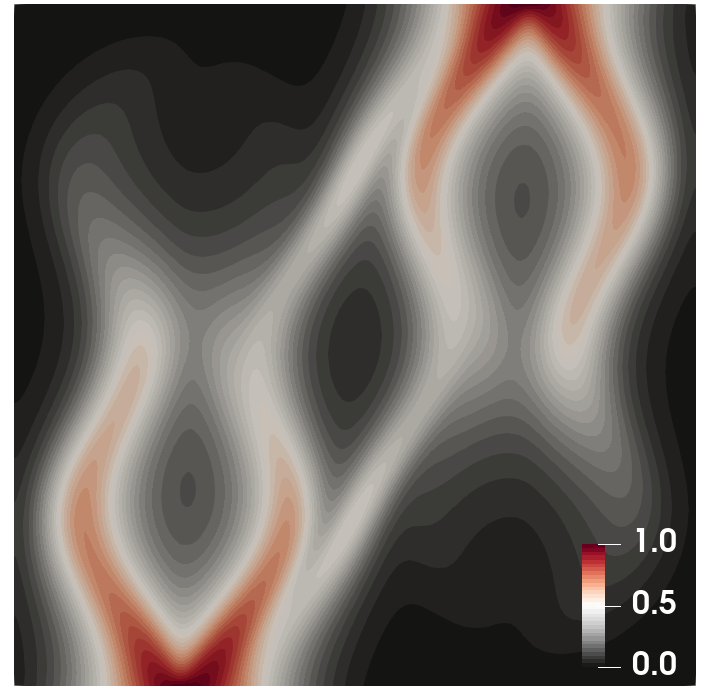}
\includegraphics[width=0.23\textwidth]{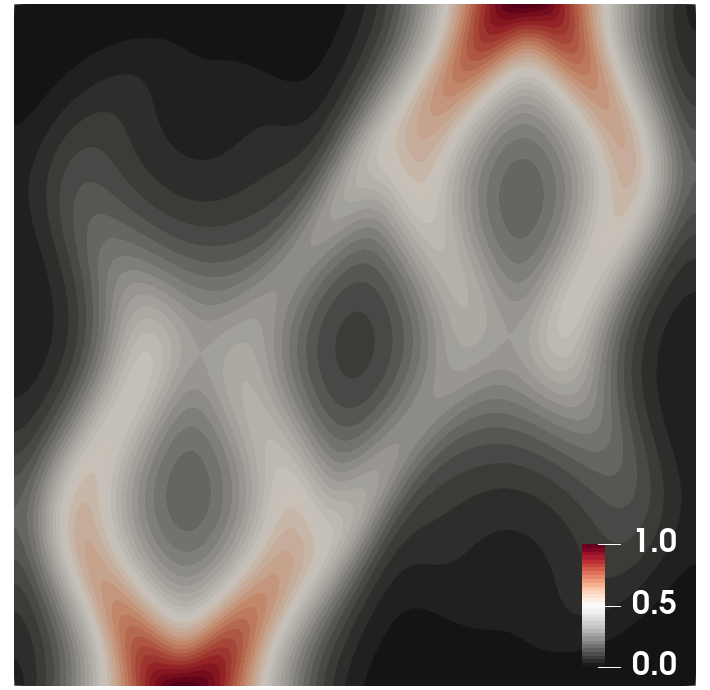}
\includegraphics[width=0.23\textwidth]{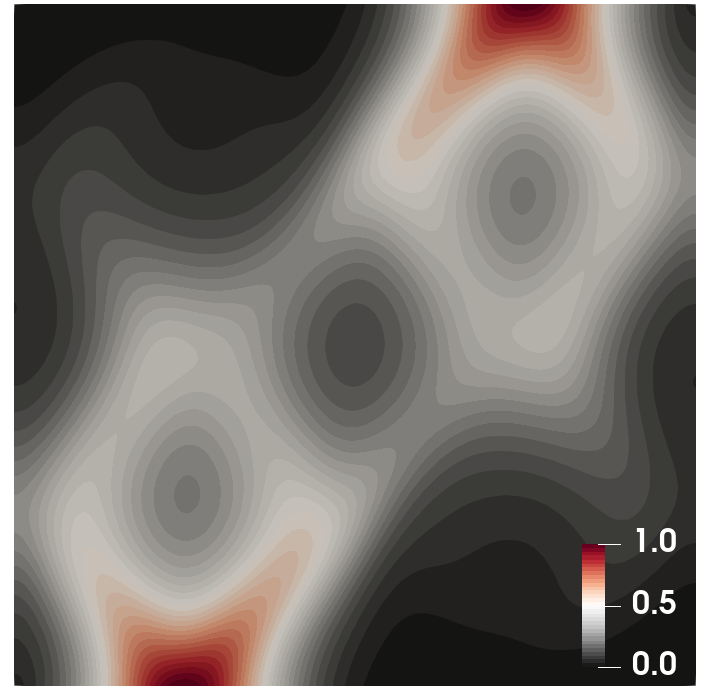}
}
\subfigure[$\rho_2$. Left: $\beta=0$, middle: $\beta=0.005$, right: $\beta=0.01$.]
{
\includegraphics[width=0.23\textwidth]{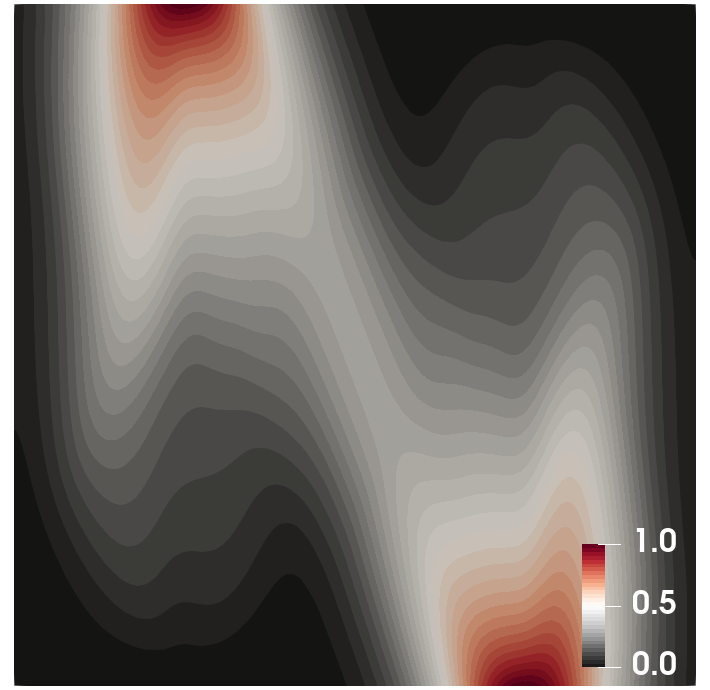}
\includegraphics[width=0.23\textwidth]{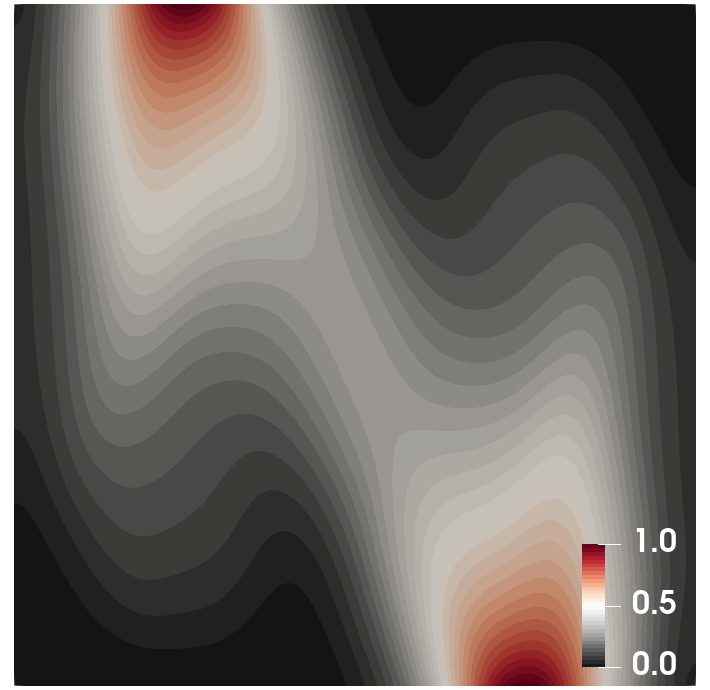}
\includegraphics[width=0.23\textwidth]{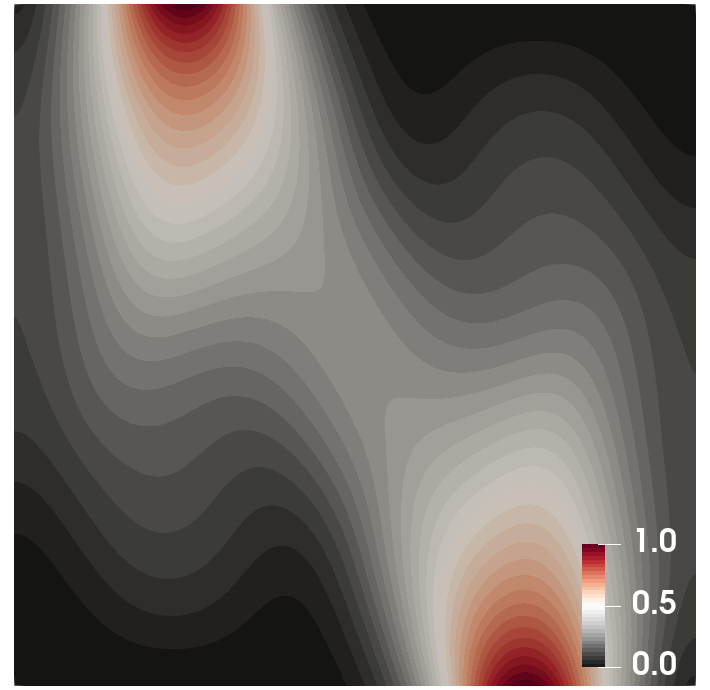}
}
\caption{Example \ref{ex3}. 
Snapshots of $\rho_1$ (top) and $\rho_2$ (bottom).
}
\label{fig:den-ex3-2D}
\end{figure}

\begin{figure}[H]
\centering
\subfigure[$\rho_1$. $\beta=0$. ]
{
\includegraphics[width=0.192\textwidth]{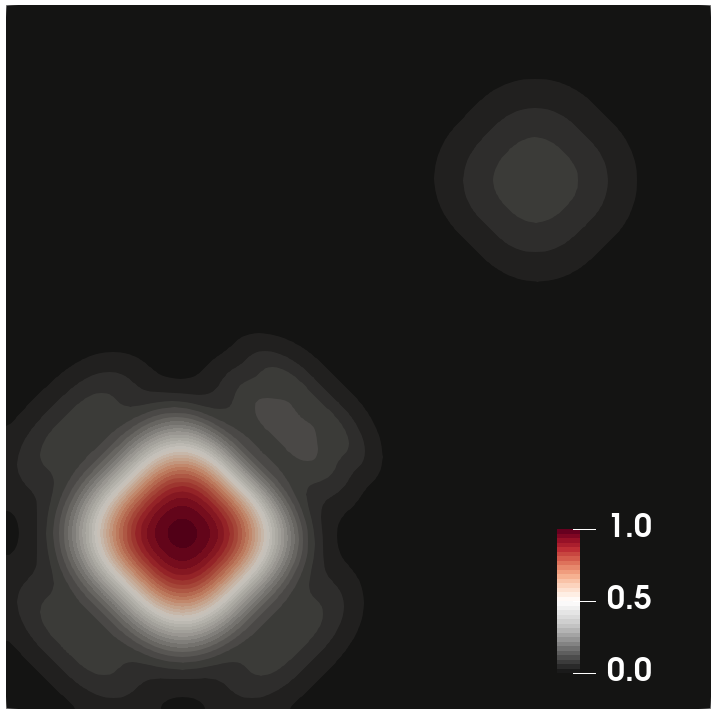}
\includegraphics[width=0.192\textwidth]{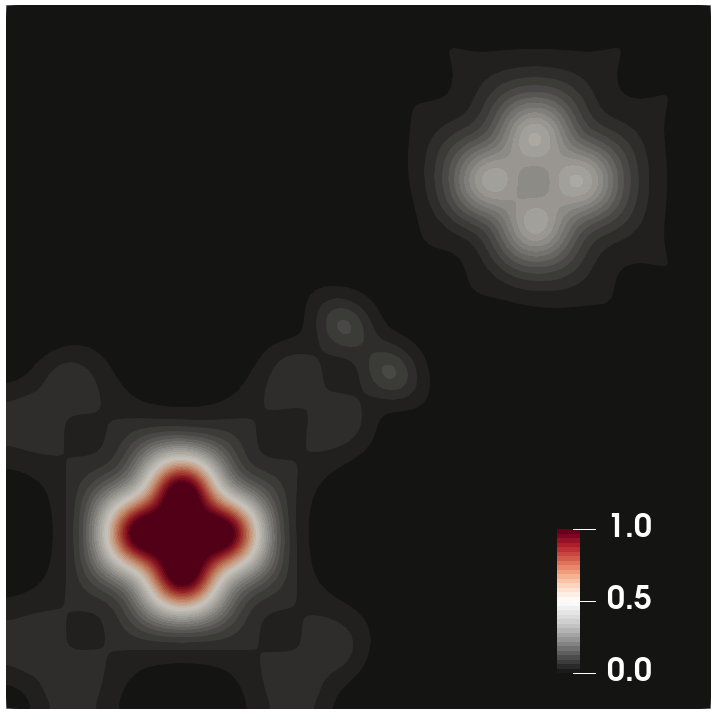}
\includegraphics[width=0.192\textwidth]{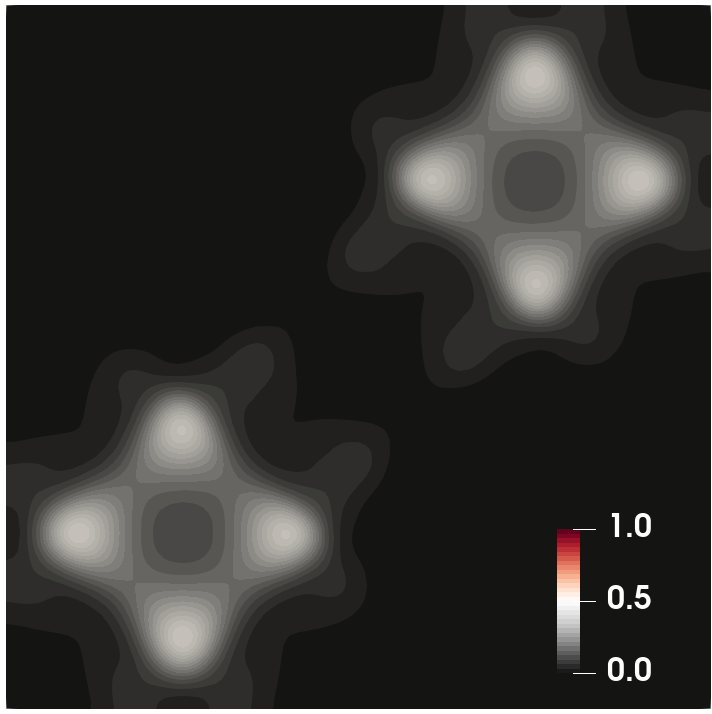}
\includegraphics[width=0.192\textwidth]{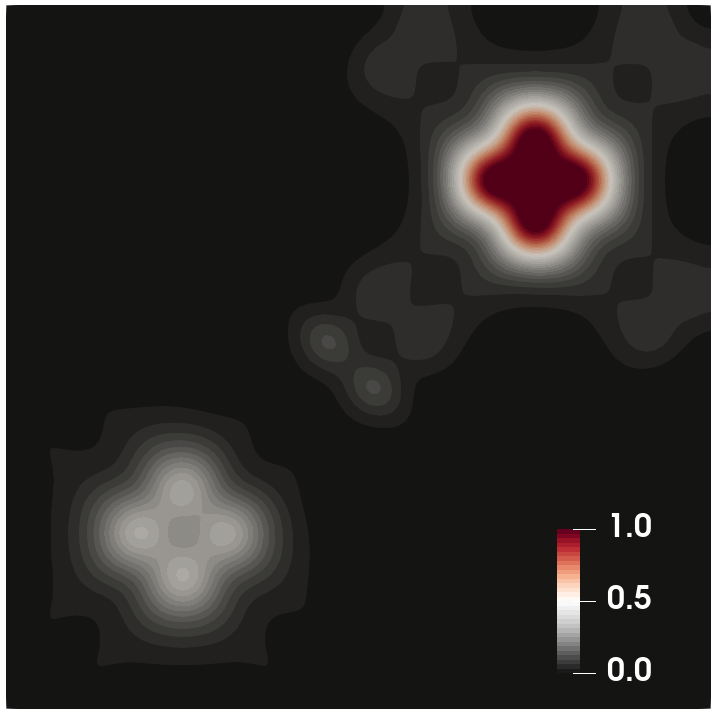}
\includegraphics[width=0.192\textwidth]{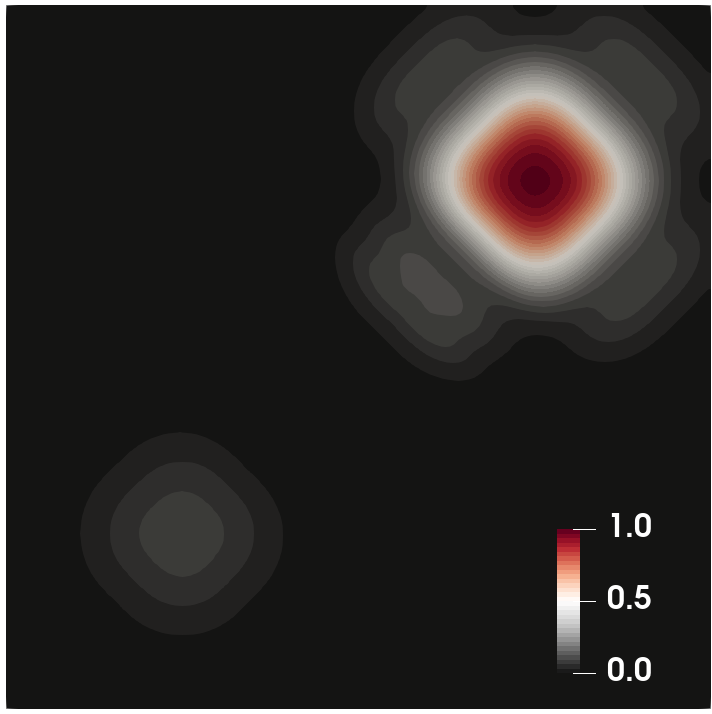}
}
\subfigure[$\rho_1$. $\beta=0.01$.]
{
\includegraphics[width=0.192\textwidth]{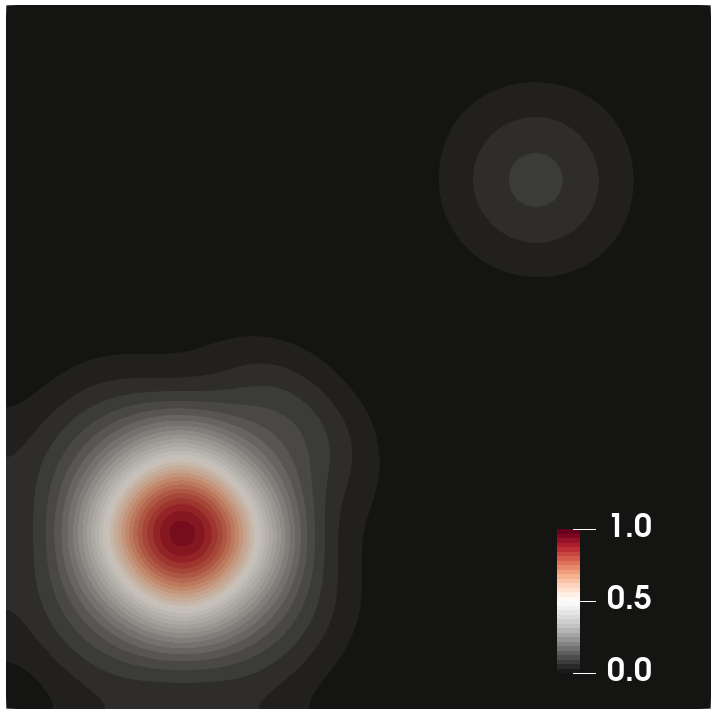}
\includegraphics[width=0.192\textwidth]{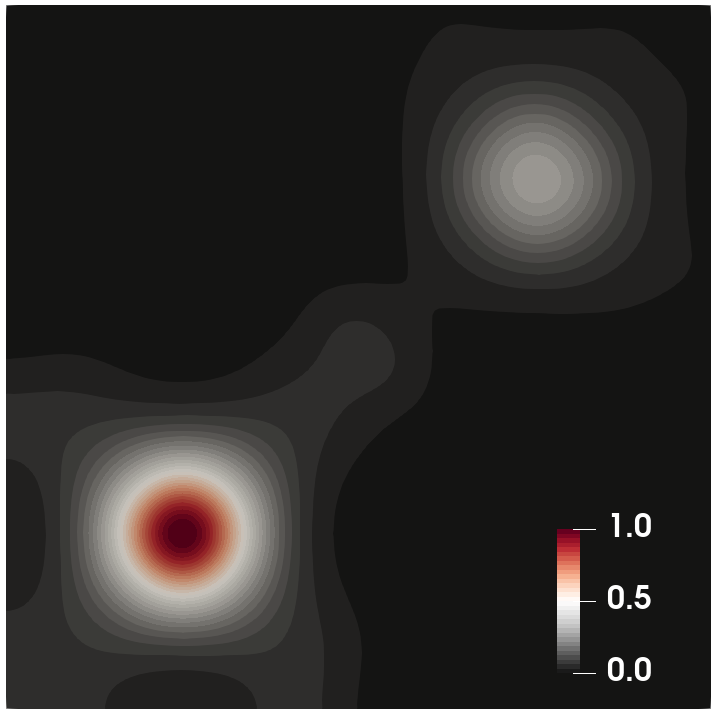}
\includegraphics[width=0.192\textwidth]{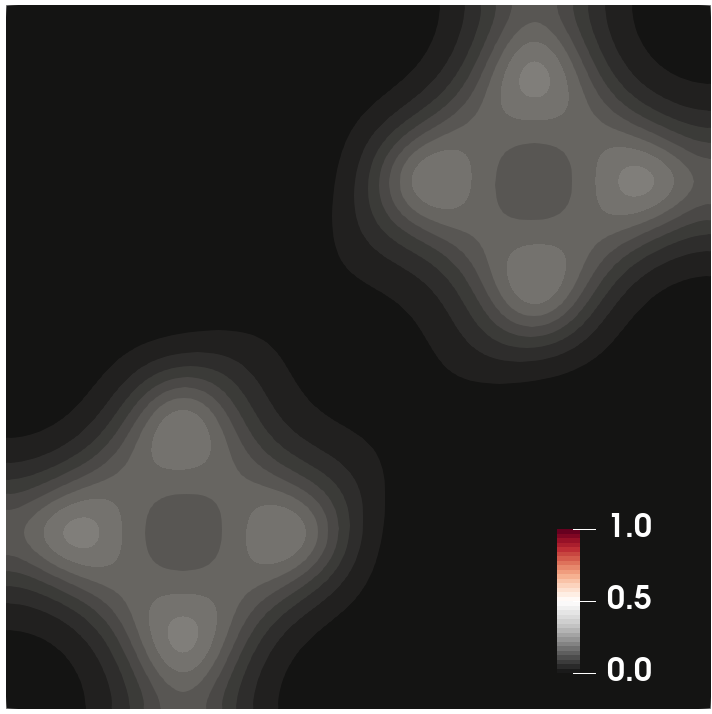}
\includegraphics[width=0.192\textwidth]{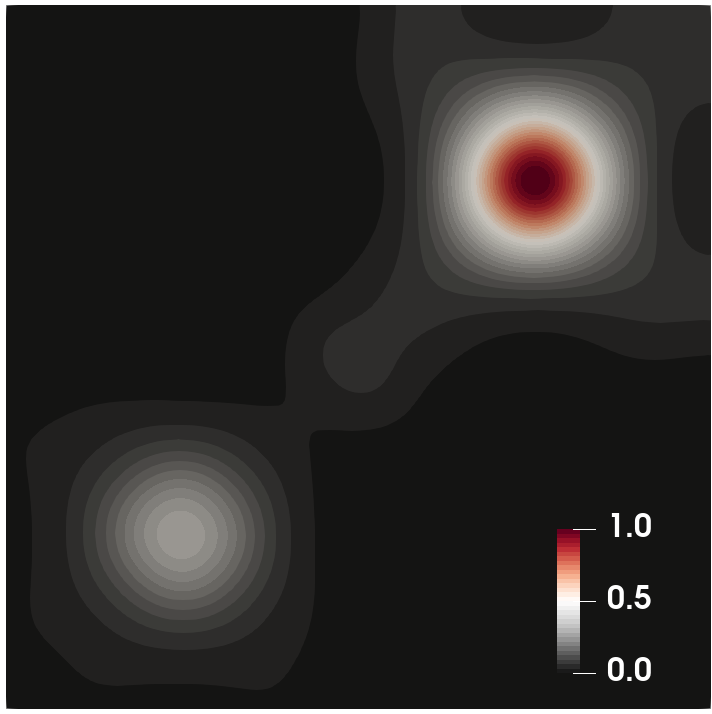}
\includegraphics[width=0.192\textwidth]{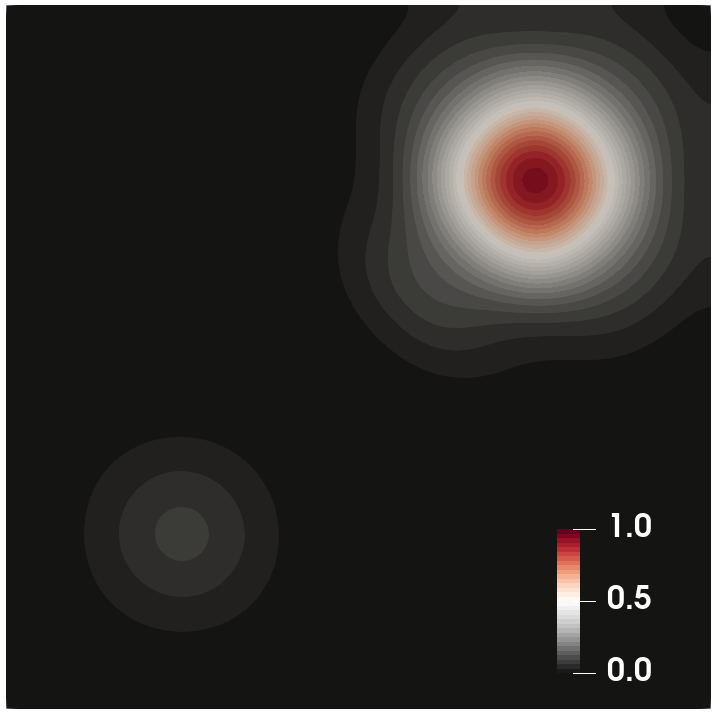}
}
\subfigure[$\rho_2$. $\beta=0$.]
{
\includegraphics[width=0.192\textwidth]{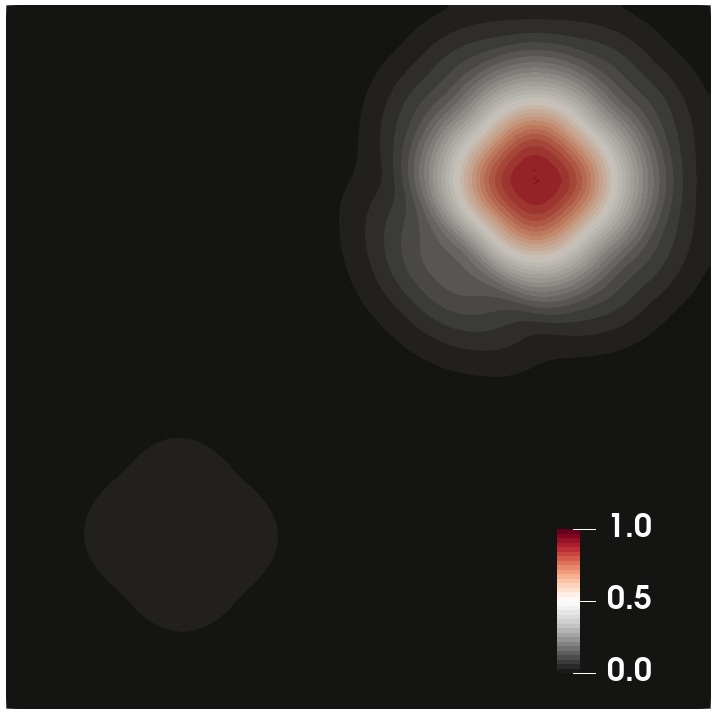}
\includegraphics[width=0.192\textwidth]{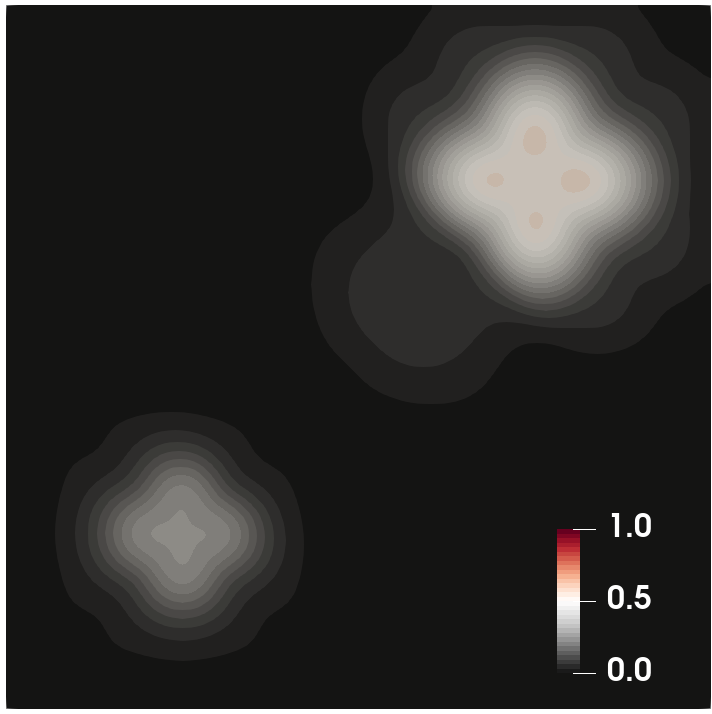}
\includegraphics[width=0.192\textwidth]{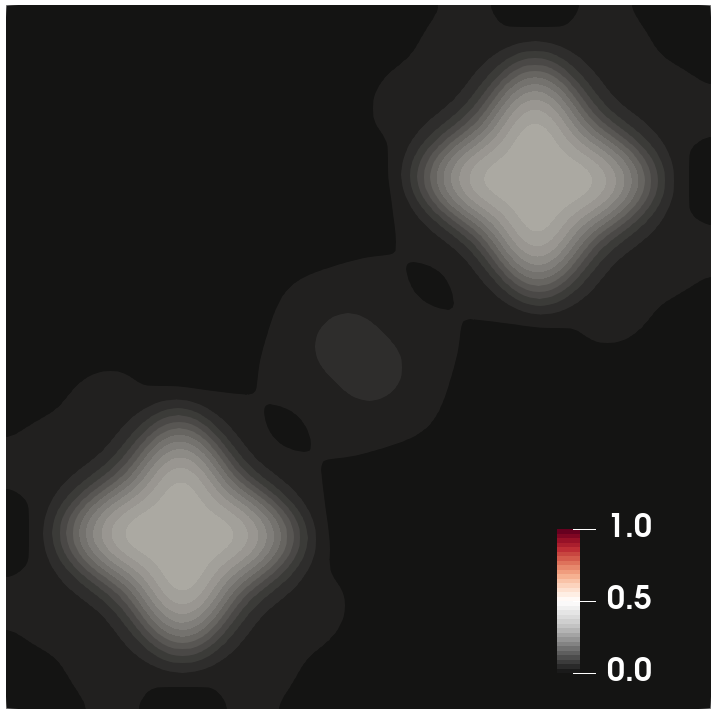}
\includegraphics[width=0.192\textwidth]{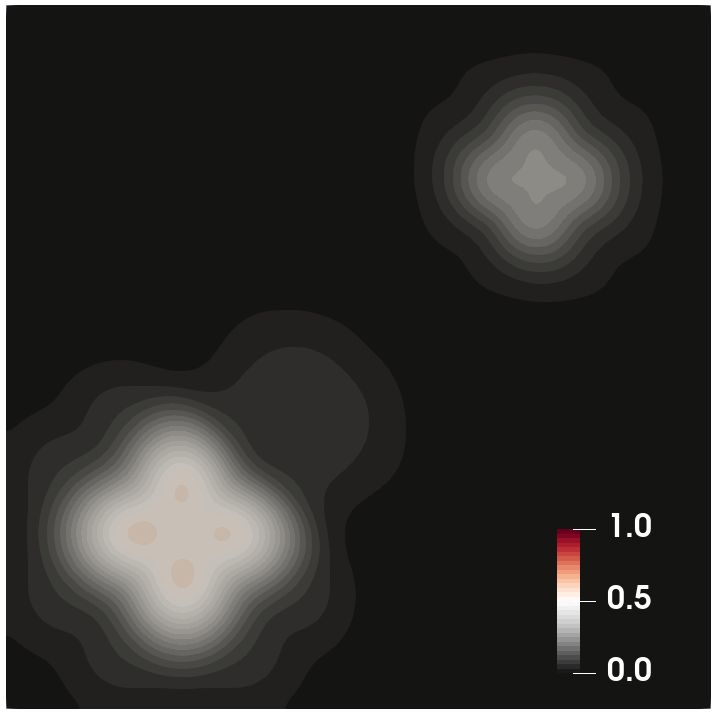}
\includegraphics[width=0.192\textwidth]{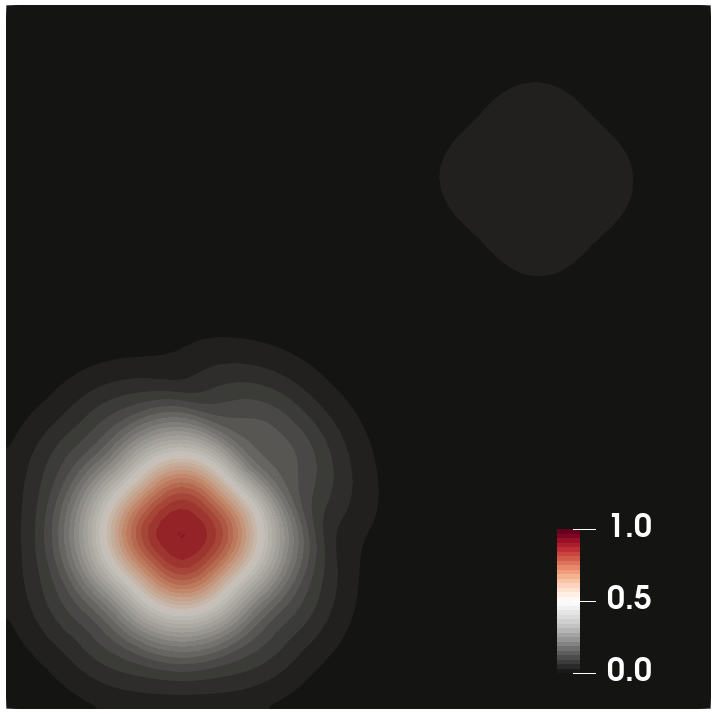}
}
\subfigure[$\rho_2$.  $\beta=0.01$. ]
{
\includegraphics[width=0.192\textwidth]{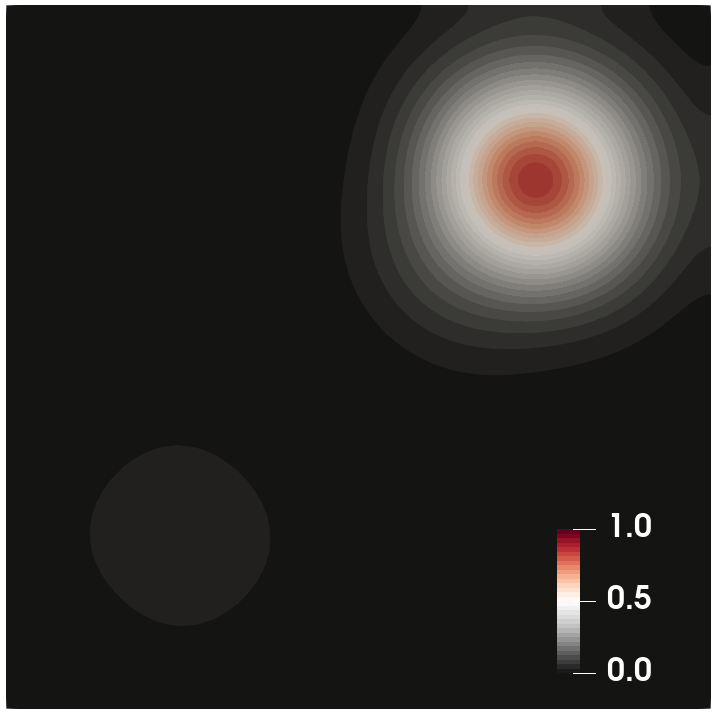}
\includegraphics[width=0.192\textwidth]{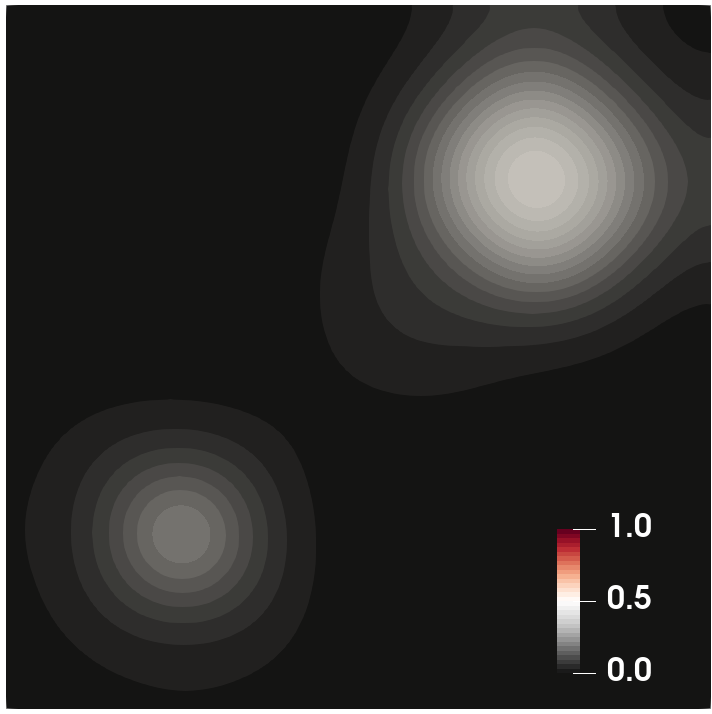}
\includegraphics[width=0.192\textwidth]{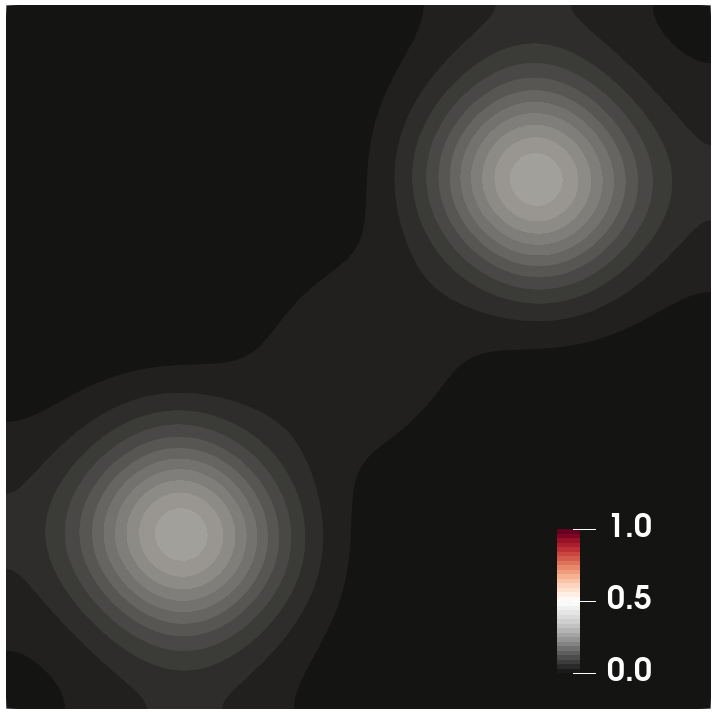}
\includegraphics[width=0.192\textwidth]{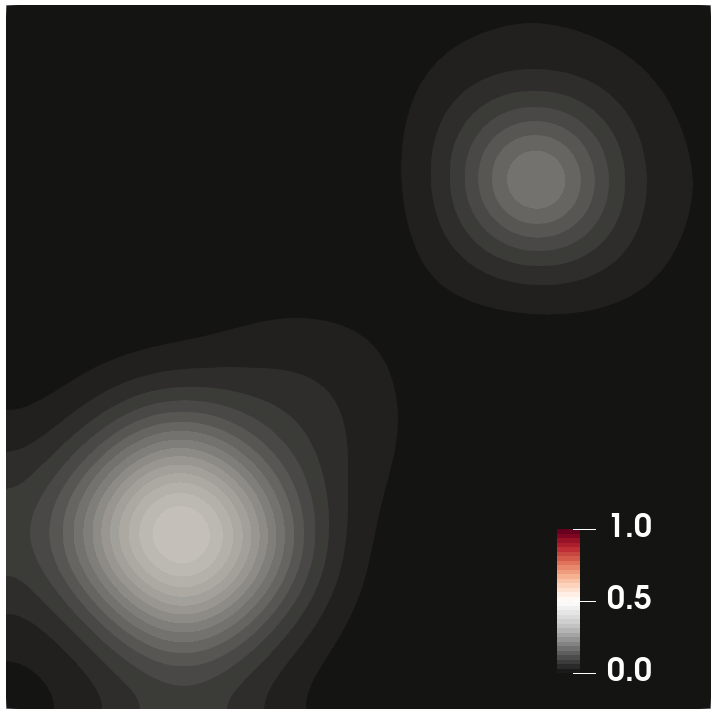}
\includegraphics[width=0.192\textwidth]{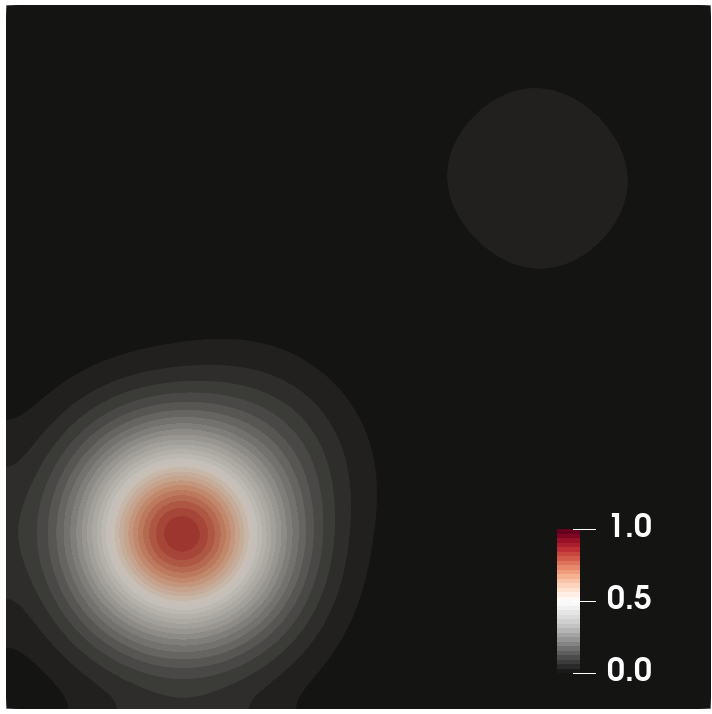}
}
\caption{Example \ref{ex3}. 
Snapshots of $\rho_1$ and $\rho_2$ at 
$t=$ 0.1,0.3,0.5,0.7,0.9 (left to right).
}
\label{fig:den-ex3-3D}
\end{figure}

\subsection{2D Scalar MFC for reaction-diffusion: image transfer}
\label{ex4}
In this example, we consider a 2D scalar mean field planning problem \eqref{smfc} with complex initial and terminal densities $\rho^0(x)$ and $\rho^1(x)$, and
a complex spatial coordinate dependent mobility $V_2(x, \rho)=20\rho^2(x)\rho$ and potential  $F(x,\rho)= -0.001\rho^3(x)\rho$. 
Here the four non-negative functions $\rho^0(x)$, $\rho^1(x)$, $\rho^2(x)$, and $\rho^3(x)$ are normalized mascot images as shown in Figure~\ref{fig:img}, which are logos from University of Notre Dame, UCLA, Portland State university, and University of South Carolina, respectively. 
\begin{figure}[h]
\centering
\subfigure[$\rho^0(x)$]
{\includegraphics[width=0.23\textwidth]{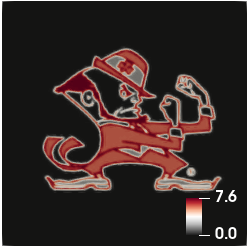}}
\subfigure[$\rho^1(x)$]
{\includegraphics[width=0.23\textwidth]{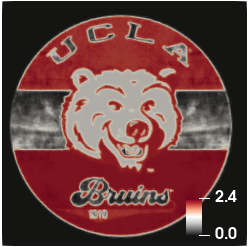}}
\subfigure[$\rho^2(x)$]
{\includegraphics[width=0.23\textwidth]{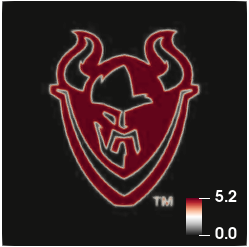}}
\subfigure[$\rho^3(x)$]
{\includegraphics[width=0.23\textwidth]{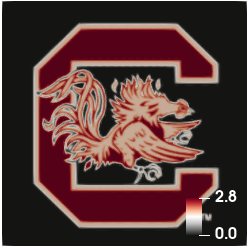}}
\caption{Example \ref{ex4}. 
The four functions $\rho^0(x)$, $\rho^1(x)$, $\rho^2(x)$, $\rho^3(x)$.
Initial density is $\rho^0$, terminal density is $\rho^1$, mobility $V_2(x,\rho) = 20\rho^2(x)\rho$
and potential  $F(x, \rho)=-0.001\rho^3(x)\rho$.
}
\label{fig:img}
\end{figure}
We further take $V_1(\rho)=V_3(\rho) = \rho$, and vary the regularization parameter $\beta=0$, $\beta=5\times 10^{-4}$, and 
$\beta=10^{-3}$.
We apply the scheme \eqref{SD-scalarh3X} with polynomial degree $k=4$ on a $16\times 64\times 64$ uniform mesh, and perform 2000 ALG iterations.
Snapshots of density contours at 
different times for different $\beta$
are shown in Figure~\ref{fig:den-ex4}. 
It is interesting to observe that the mobility coefficient $\rho^2(x)$ in $V_2$ and the interaction coefficient $\rho^3(x)$ in $F$
are imprinted in the density evolution. 
We also observe a strong diffusion effect when $\beta=10^{-3}$.

\begin{figure}[H]
\centering
\subfigure[$\beta=0$]
{
\includegraphics[width=0.192\textwidth]{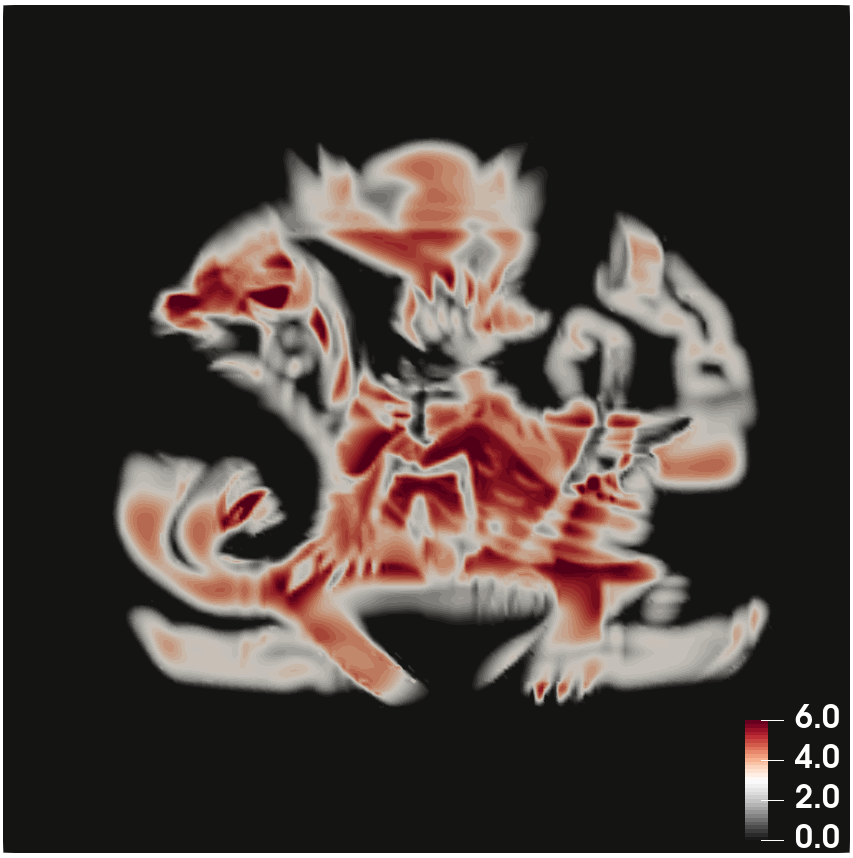}
\includegraphics[width=0.192\textwidth]{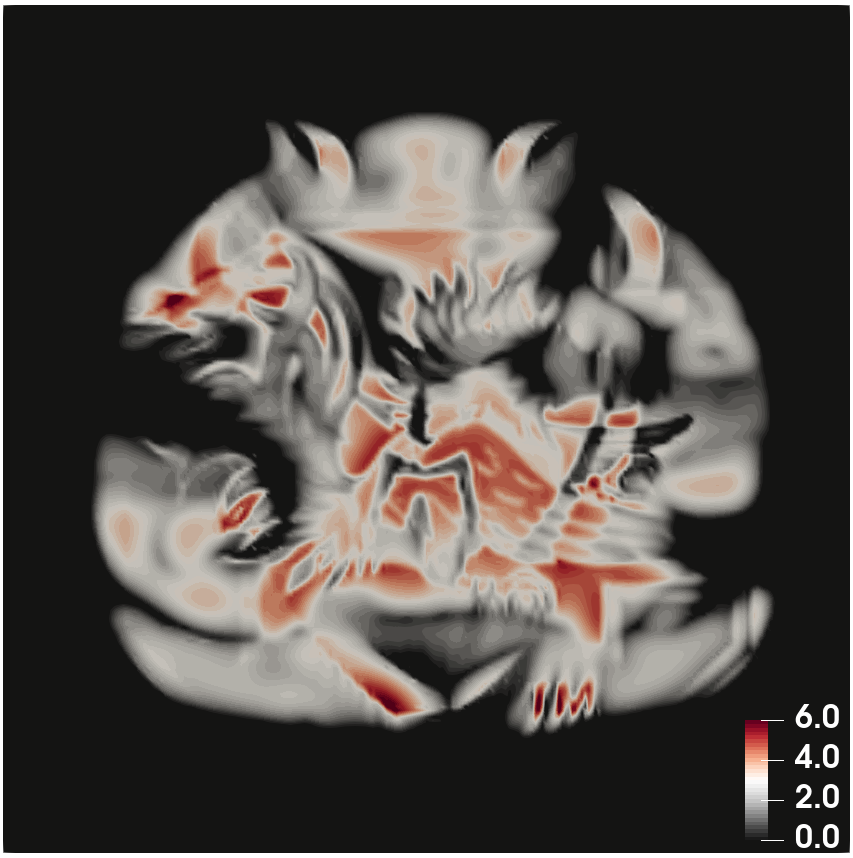}
\includegraphics[width=0.192\textwidth]{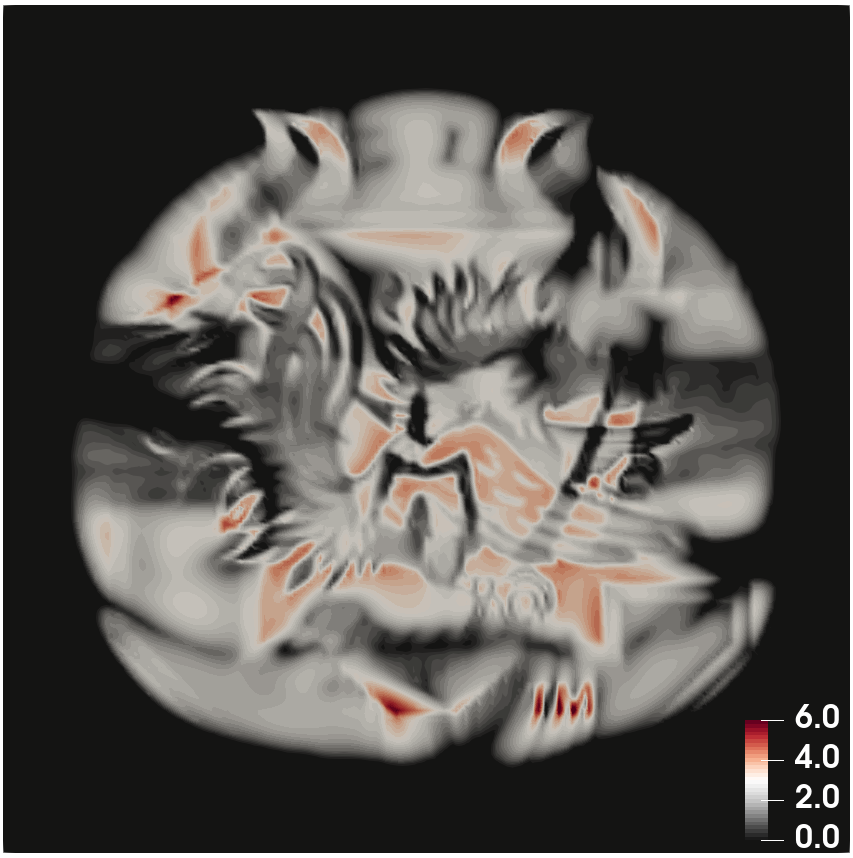}
\includegraphics[width=0.192\textwidth]{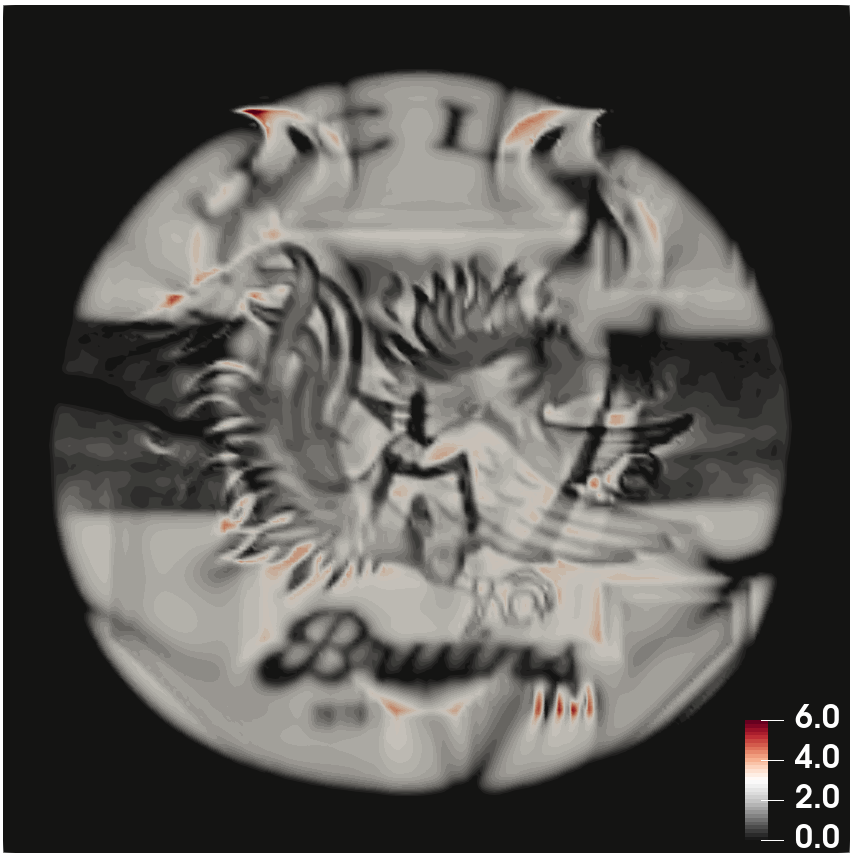}
}
\subfigure[$\beta=5\times 10^{-4}$]
{
\includegraphics[width=0.192\textwidth]{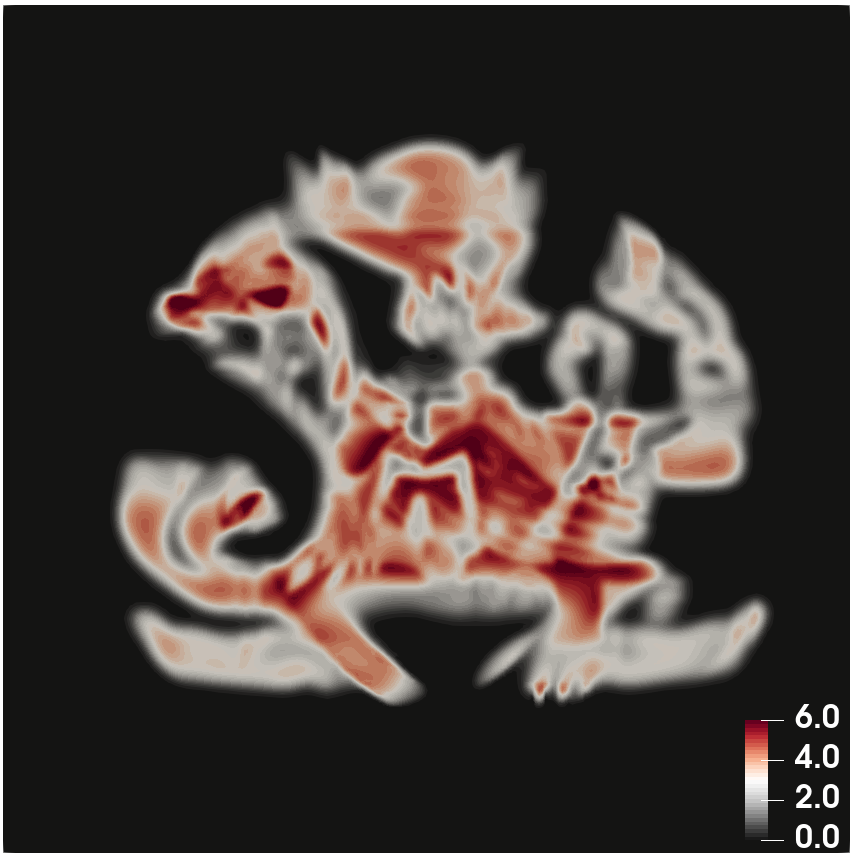}
\includegraphics[width=0.192\textwidth]{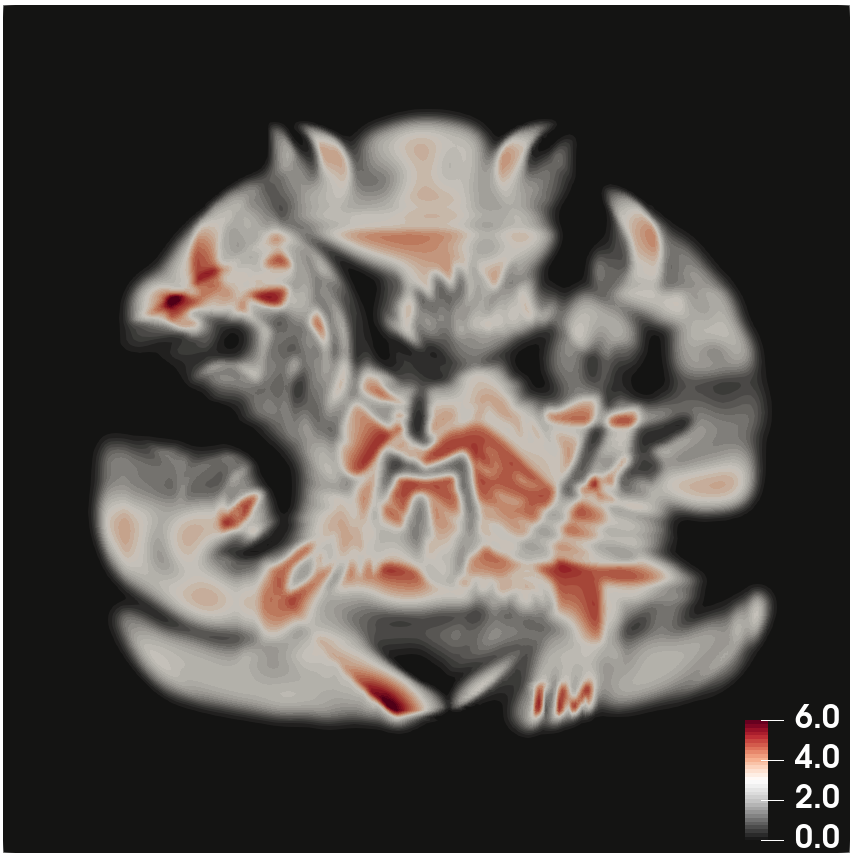}
\includegraphics[width=0.192\textwidth]{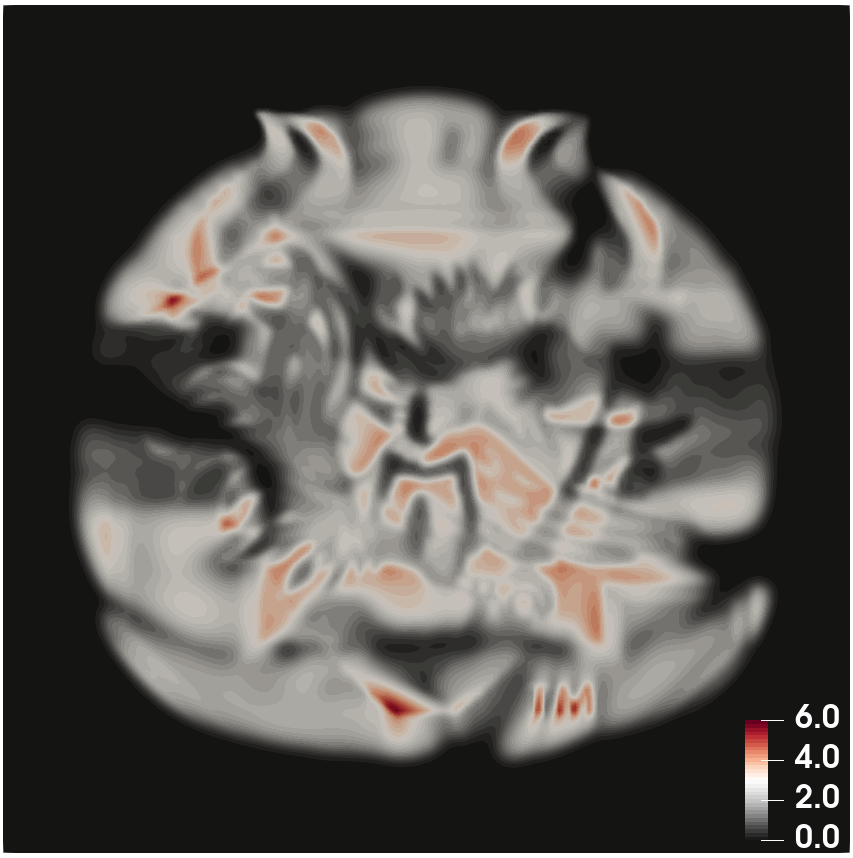}
\includegraphics[width=0.192\textwidth]{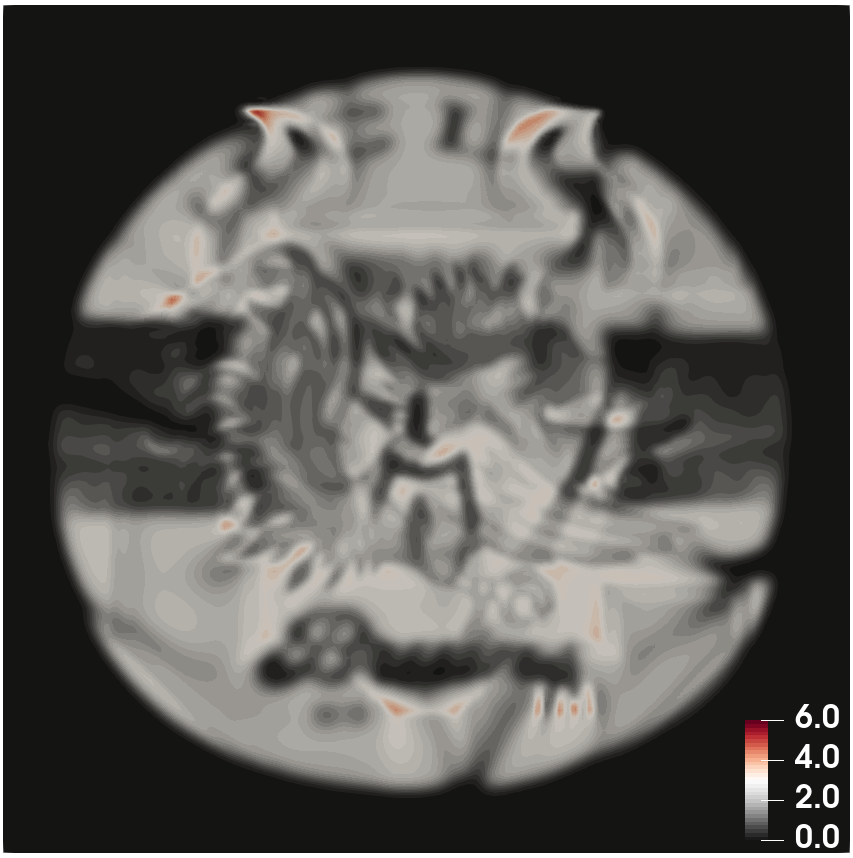}
}
\subfigure[$\beta=10^{-3}$]
{
\includegraphics[width=0.192\textwidth]{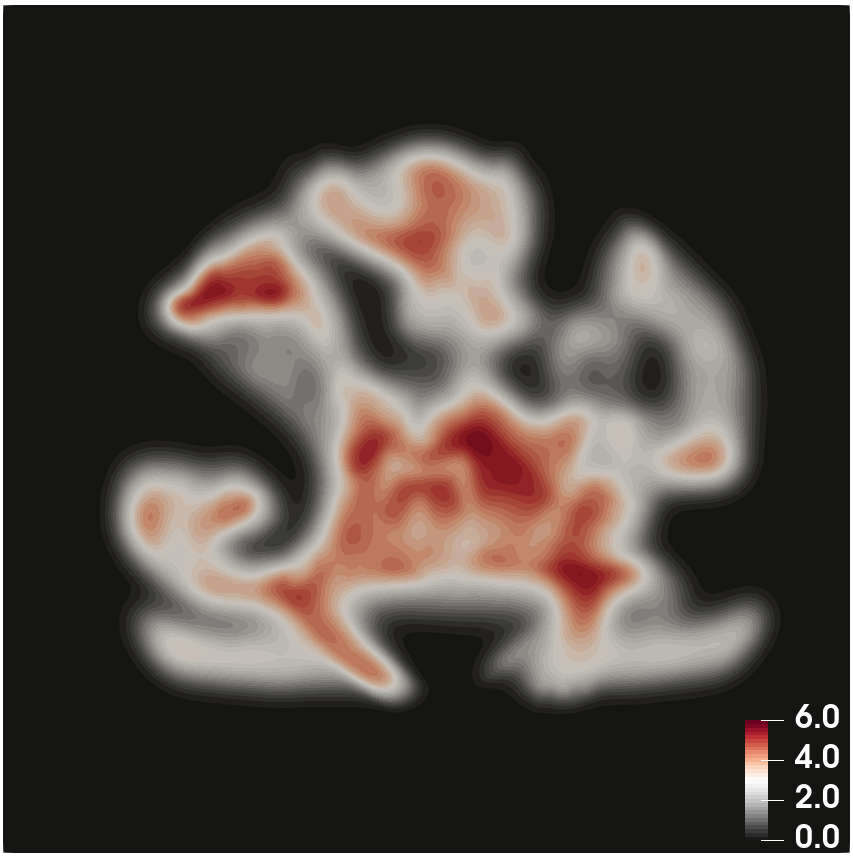}
\includegraphics[width=0.192\textwidth]{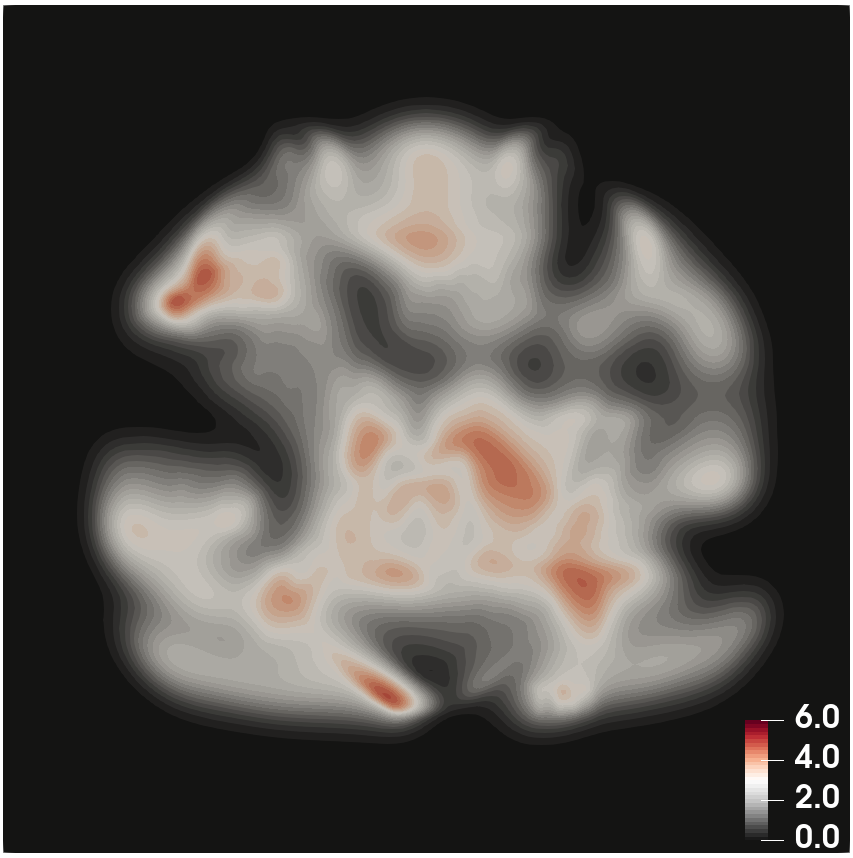}
\includegraphics[width=0.192\textwidth]{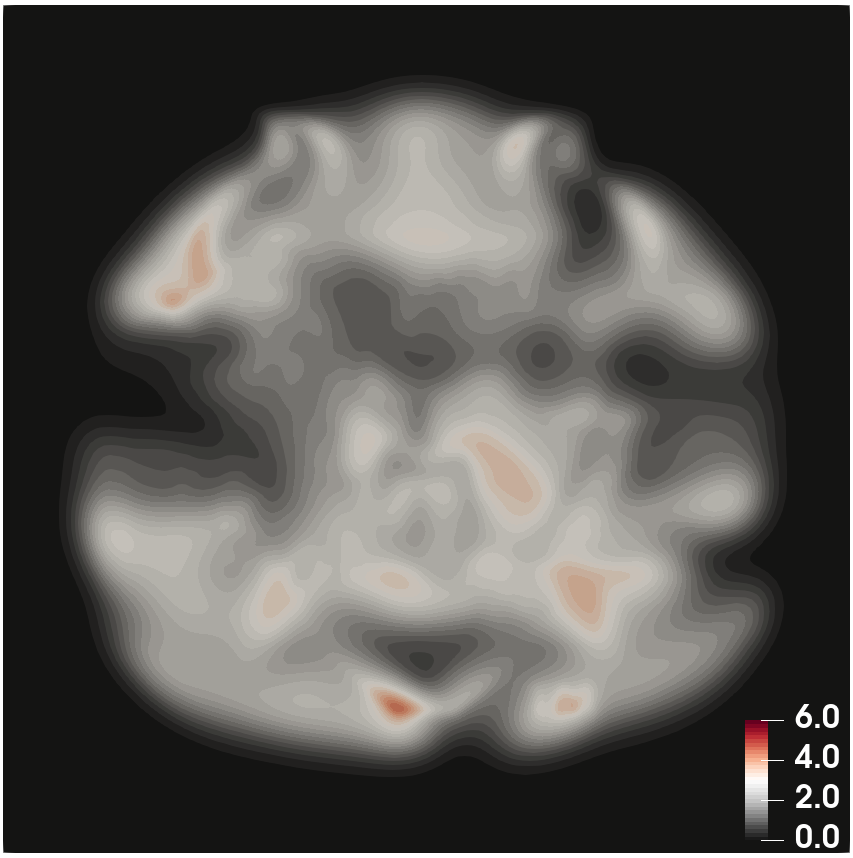}
\includegraphics[width=0.192\textwidth]{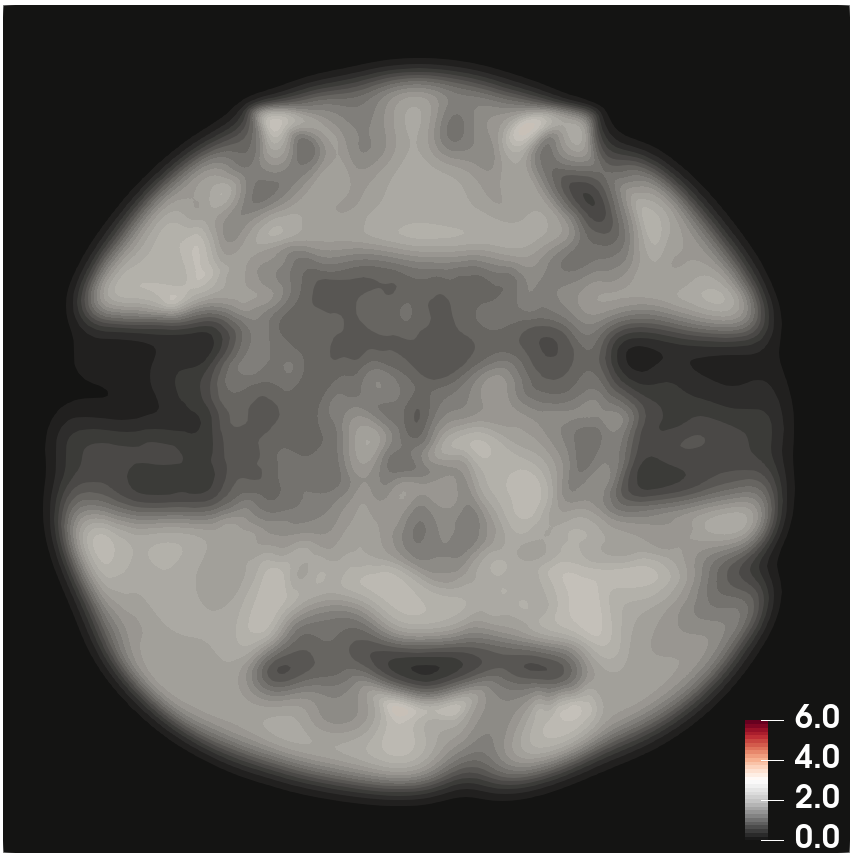}
}
\caption{Example \ref{ex4}. 
Snapshots of density contours at 
$t=$ 0.2,0.4,0.6,0.8 (left to right) for different $\beta$.
}
\label{fig:den-ex4}
\end{figure}

\subsection{2D System MFC for reaction-diffusion ($M=2$, $R=1$): image transfer}
\label{ex5}
In this example, we consider a system model \eqref{vmfc} with $M=2$ species and $R=1$ reaction.
We take the initial and terminal densities as the images in Figure~\ref{fig:img}. 
Specifically, the two initial densities $\rho_1^0 = \rho^0(x)$, $\rho_2^0=\rho^3(x)$, 
and the two terminal densities $\rho_1^1 = \rho^1(x)$ and $\rho_2^1=\rho^2(x)$.
We take $V_{1,i}(\rho)=\rho$, 
\[
V_{2,1}(\bmr) = 20\frac{\rho_1+\rho_2}{2}, 
\]
no potential $\bm F(\bmr) = 0$, and no regularization $\beta=0$. 
We use the same discretization as the previous example, and apply 2000 ALG iterations.
Snapshots of the density contour at different times are shown in Figure~\ref{fig:den-ex5}.
Here the reaction with mobility $V_{2,1}$ makes the density evolution different from a classical optimal transport path for each component.

\begin{figure}[H]
\centering
\subfigure[$\rho_1$]
{
\includegraphics[width=0.192\textwidth]{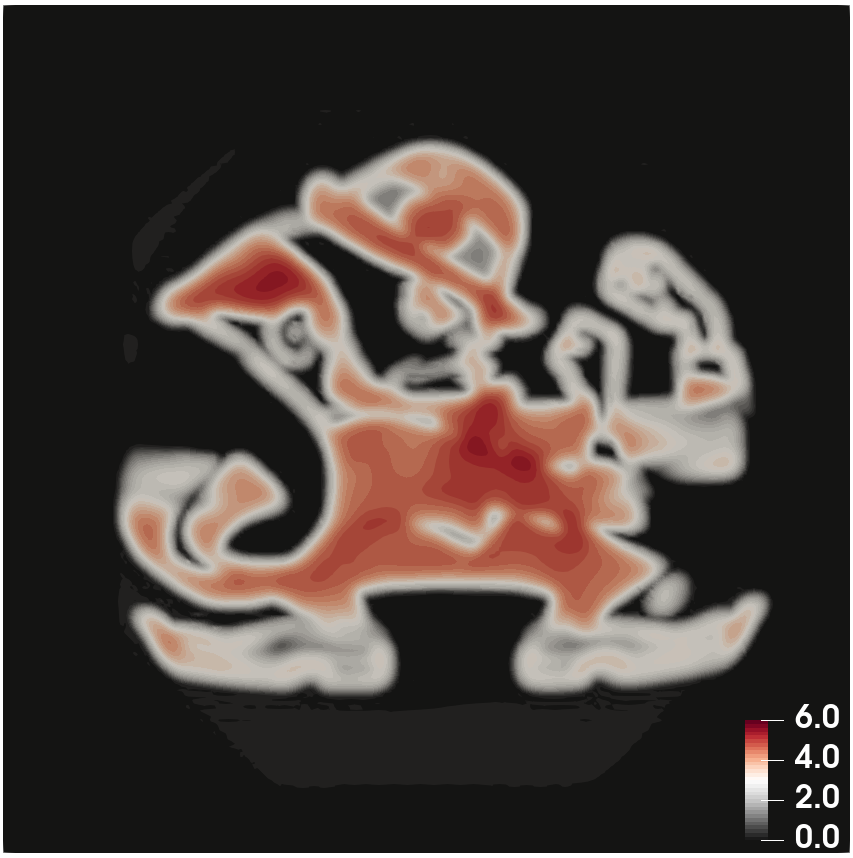}
\includegraphics[width=0.192\textwidth]{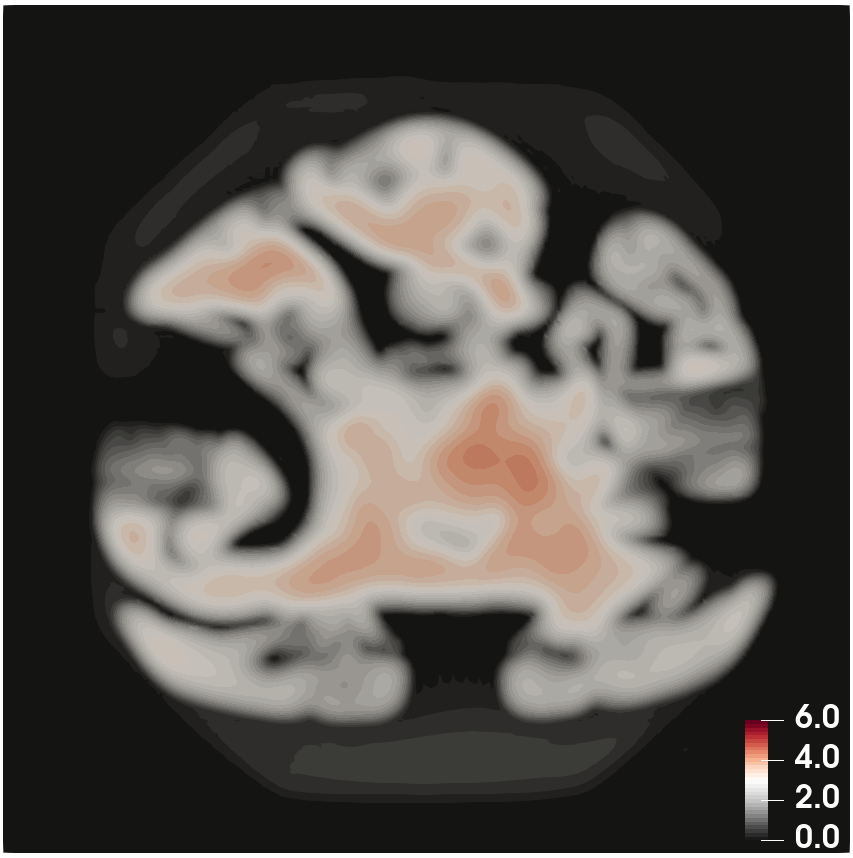}
\includegraphics[width=0.192\textwidth]{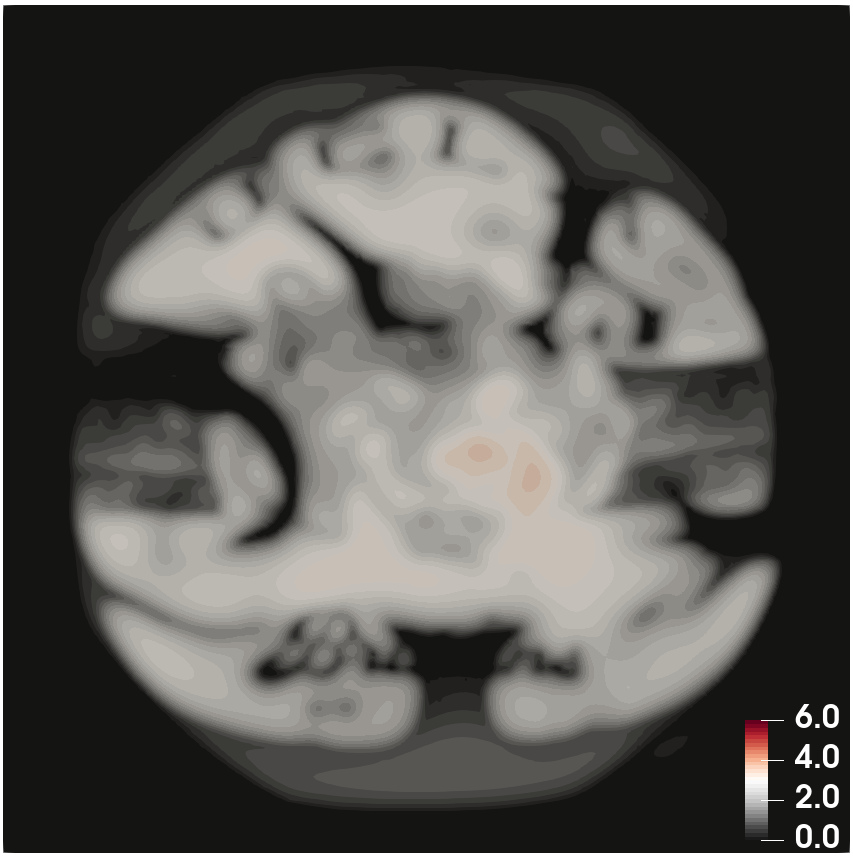}
\includegraphics[width=0.192\textwidth]{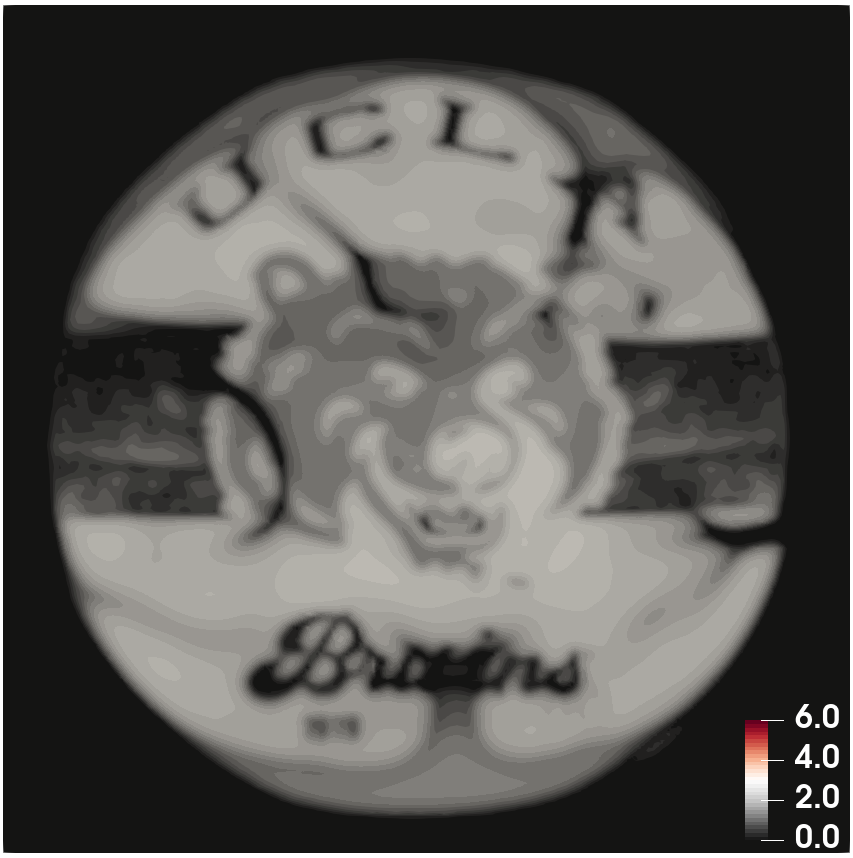}
}
\subfigure[$\rho_2$]
{
\includegraphics[width=0.192\textwidth]{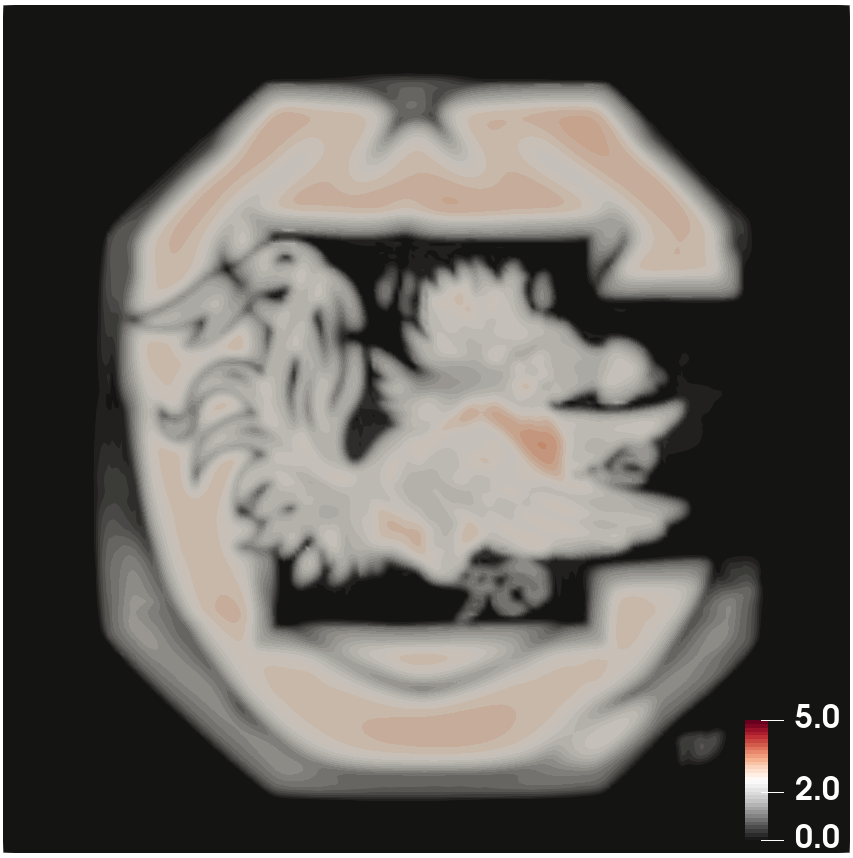}
\includegraphics[width=0.192\textwidth]{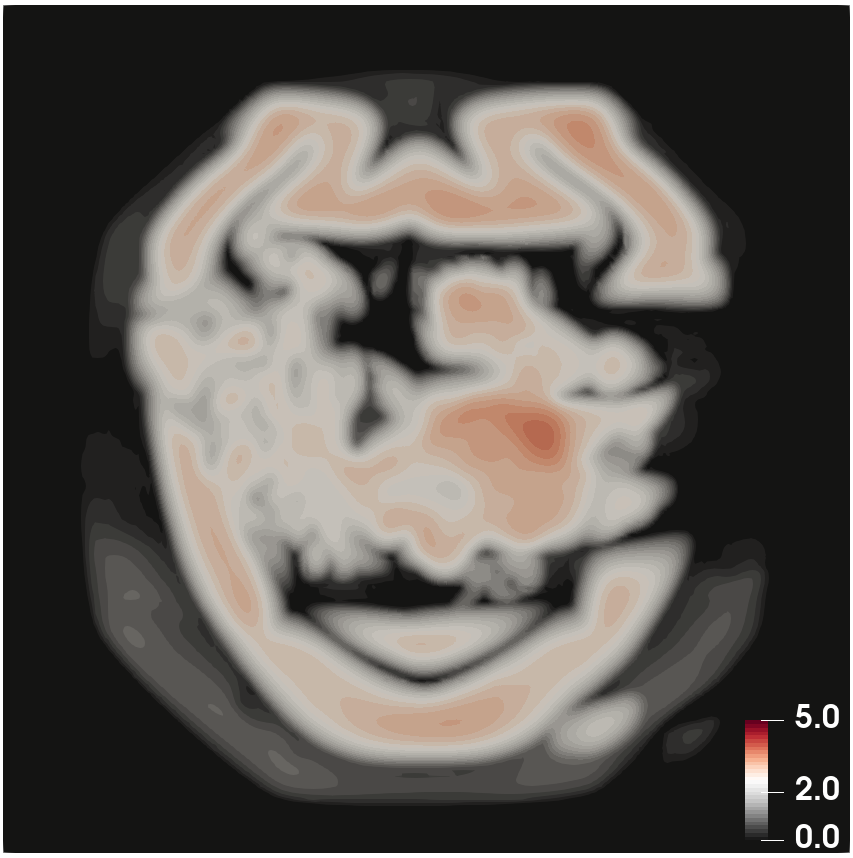}
\includegraphics[width=0.192\textwidth]{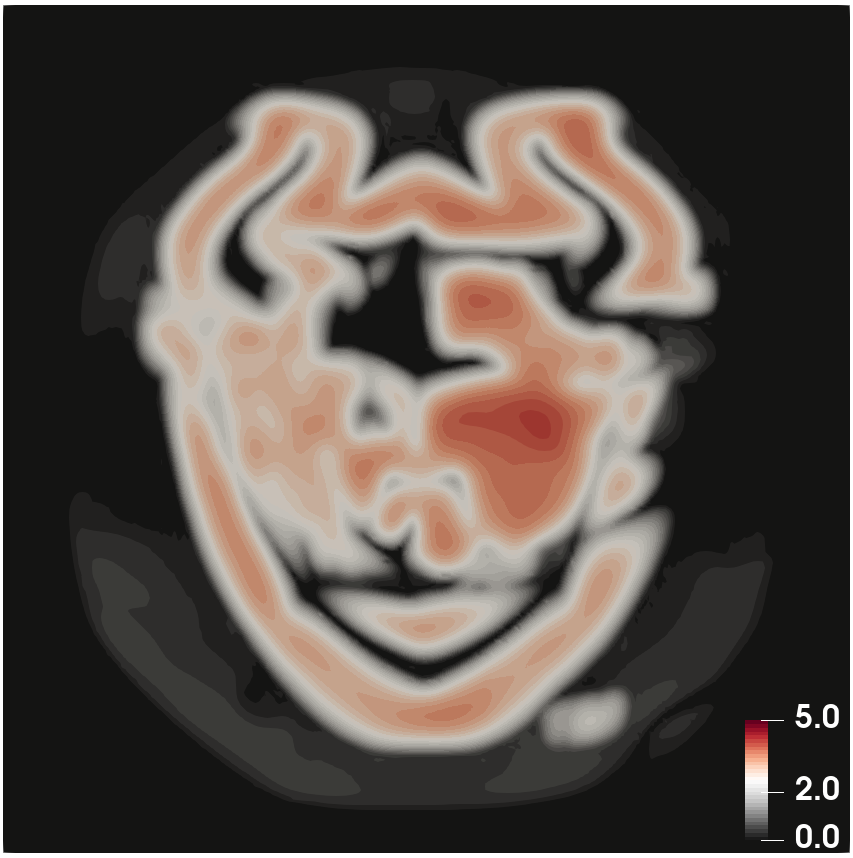}
\includegraphics[width=0.192\textwidth]{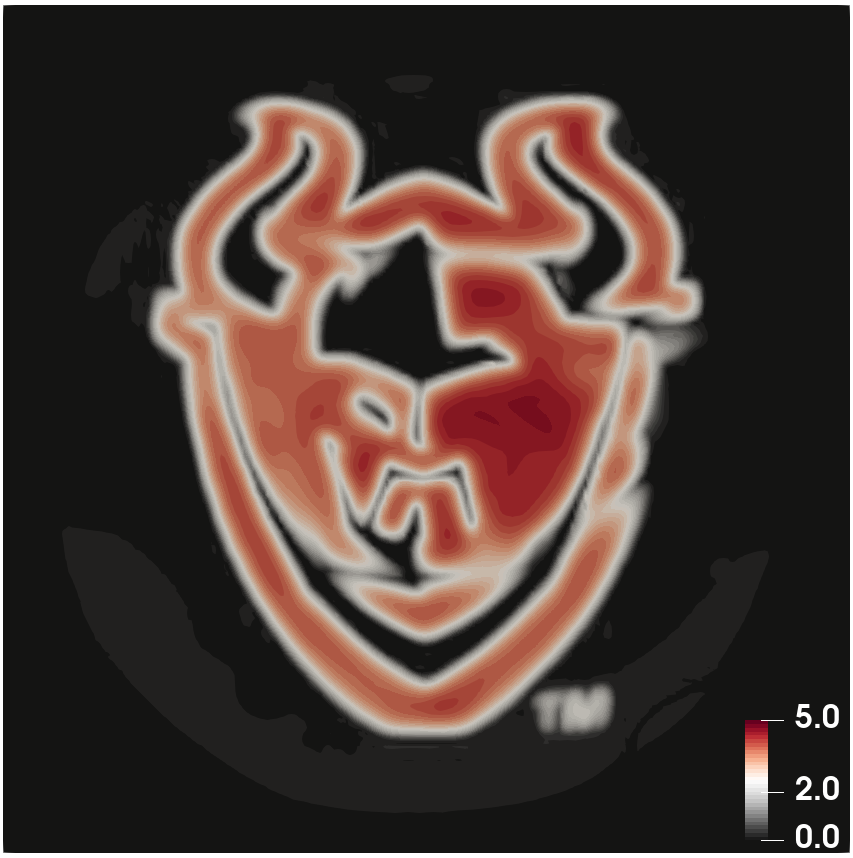}
}
\caption{Example \ref{ex5}. 
Snapshots of $\rho_1$ (top) and $\rho_2$ (bottom) at 
$t=$ 0.2,0.4,0.6,0.8 (left to right).
}
\label{fig:den-ex5}
\end{figure}

\subsection{2D System MFC for reaction-diffusion ($M=12$, $R=12$)}
\label{ex6}
In the last example, we consider a system model \eqref{vmfc} with $M=12$ species and $R=12$ reaction.
The initial densities are the ancient Chinese calligraphy in seal script
for the 12 Chinese zodiac animals, which are downloaded from Richard Sears' website \url{https://hanziyuan.net/}, while the terminal densities are their associated (gray-scale) images, which are generated by 
Baidu's text-to-image AI tool WenXin YiGe \url{https://yige.baidu.com/}. We rescale the images so that the maximal value is 1 and minimial value is 0; see Figure~\ref{fig:den-ex60}--\ref{fig:den-ex61} for the 12 initial ($t=0$) and terminal ($t=1$) density approximations in gray scale. 

Here we take $V_{1,i}(\rho)=\rho$, 
\[
V_{2,p}(\bmr) = 20 \frac{\rho_{p}-\rho_{p+1}}{\log(\rho_{p})-\log(\rho_{p+1})},\quad \forall 1\le p\le 12, 
\]
where the convention $\rho_{13}=\rho_1$ is used. The reaction patterns for this system are cyclic. Again, we set potential  $\bm F(\bmr) = 0$, and regularization $\beta=0$. 
We use the same discretization as the previous example and apply 2000 ALG iterations.
The results at time $t=0.0, 0.2, 0.5, 0.8$, and $1.0$ are shown in Figure~\ref{fig:den-ex60}--\ref{fig:den-ex61}.
The color range is gray-scale from 0 (black) to 1 (white). 
We observe interesting and complex densities' evolutions from these figures. 

For comparison purposes, we also plot the snapshots of densities at time $t=0.5$ for optimal transport without reaction ($V_{2,p}=0$)
in Figure~\ref{fig:den-ex62}. It is observed that the results in the middle row of Figure~\ref{fig:den-ex60}--\ref{fig:den-ex61} for the reaction-diffusion system model are very different from the scalar optimal transport results in Figure~\ref{fig:den-ex62}.
These differences come from the nonlinear reaction mobility functions. 

\begin{figure}[H]
\centering
\subfigure{
\includegraphics[width=0.16\textwidth]{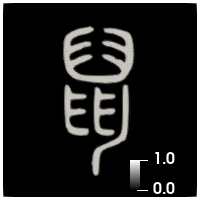}
\includegraphics[width=0.16\textwidth]{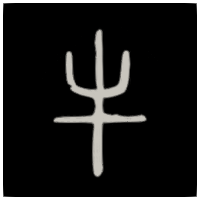}
\includegraphics[width=0.16\textwidth]{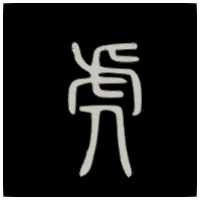}
\includegraphics[width=0.16\textwidth]{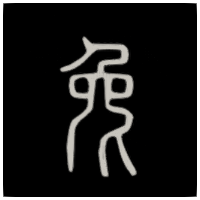}
\includegraphics[width=0.16\textwidth]{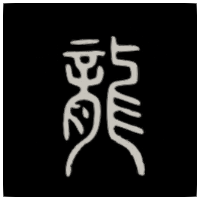}
\includegraphics[width=0.16\textwidth]{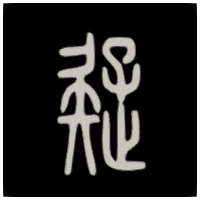}
}
\subfigure
{
\includegraphics[width=0.16\textwidth]{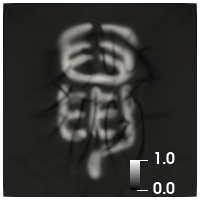}
\includegraphics[width=0.16\textwidth]{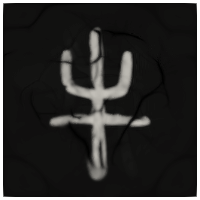}
\includegraphics[width=0.16\textwidth]{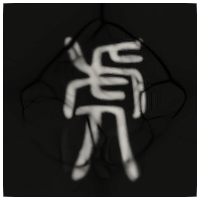}
\includegraphics[width=0.16\textwidth]{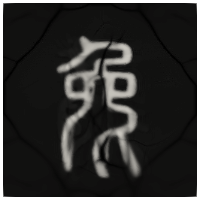}
\includegraphics[width=0.16\textwidth]{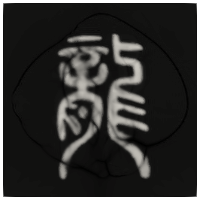}
\includegraphics[width=0.16\textwidth]{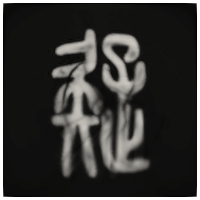}
}
\subfigure
{
\includegraphics[width=0.16\textwidth]{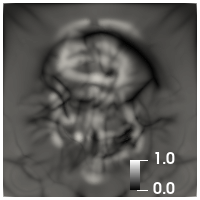}
\includegraphics[width=0.16\textwidth]{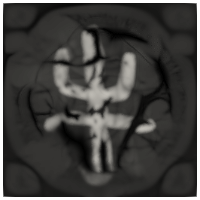}
\includegraphics[width=0.16\textwidth]{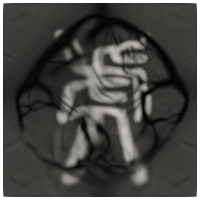}
\includegraphics[width=0.16\textwidth]{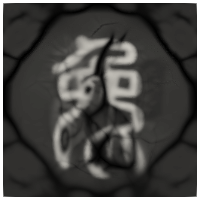}
\includegraphics[width=0.16\textwidth]{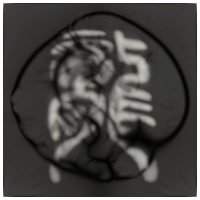}
\includegraphics[width=0.16\textwidth]{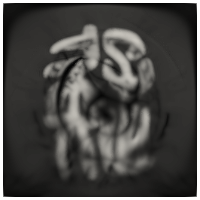}
}
\subfigure
{
\includegraphics[width=0.16\textwidth]{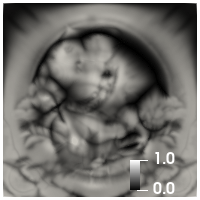}
\includegraphics[width=0.16\textwidth]{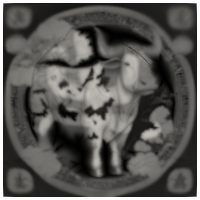}
\includegraphics[width=0.16\textwidth]{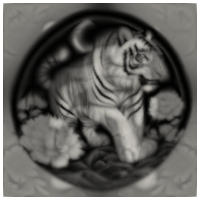}
\includegraphics[width=0.16\textwidth]{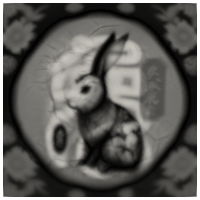}
\includegraphics[width=0.16\textwidth]{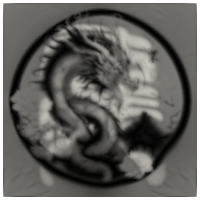}
\includegraphics[width=0.16\textwidth]{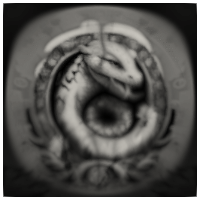}
}
\subfigure
{
\includegraphics[width=0.16\textwidth]{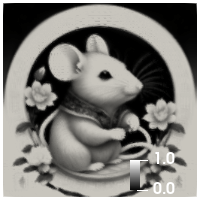}
\includegraphics[width=0.16\textwidth]{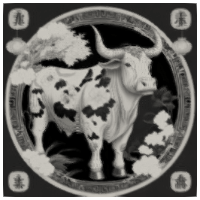}
\includegraphics[width=0.16\textwidth]{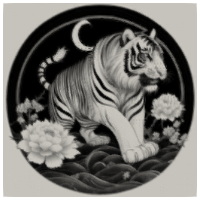}
\includegraphics[width=0.16\textwidth]{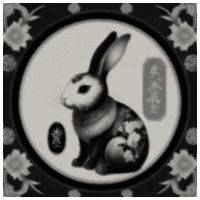}
\includegraphics[width=0.16\textwidth]{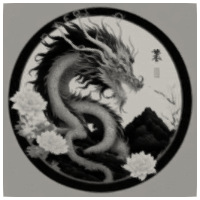}
\includegraphics[width=0.16\textwidth]{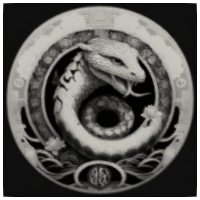}
}
\caption{Example \ref{ex6}. 
Snapshots of first 6 densities $\rho_1$ to $\rho_6$ (left to right) at times $t=0.0, 0.2, 0.5, 0.8, 1.0$ (top to bottom).
}
\label{fig:den-ex60}
\end{figure}

\begin{figure}[H]
\centering
\subfigure{
\includegraphics[width=0.16\textwidth]{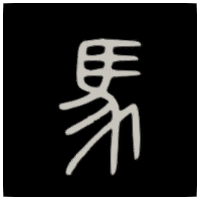}
\includegraphics[width=0.16\textwidth]{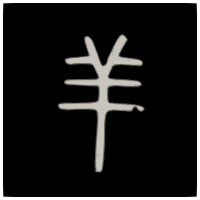}
\includegraphics[width=0.16\textwidth]{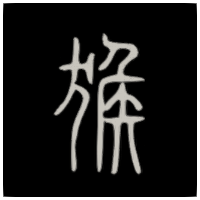}
\includegraphics[width=0.16\textwidth]{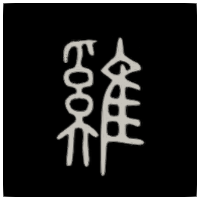}
\includegraphics[width=0.16\textwidth]{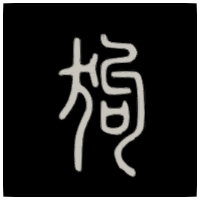}
\includegraphics[width=0.16\textwidth]{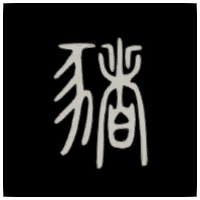}
}
\subfigure
{
\includegraphics[width=0.16\textwidth]{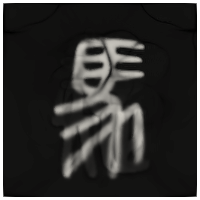}
\includegraphics[width=0.16\textwidth]{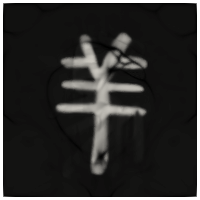}
\includegraphics[width=0.16\textwidth]{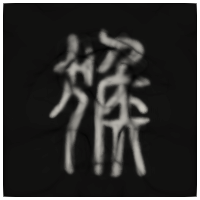}
\includegraphics[width=0.16\textwidth]{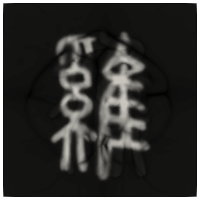}
\includegraphics[width=0.16\textwidth]{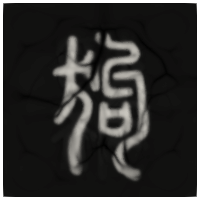}
\includegraphics[width=0.16\textwidth]{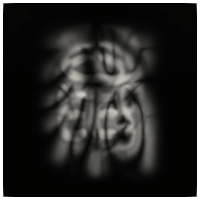}
}
\subfigure
{
\includegraphics[width=0.16\textwidth]{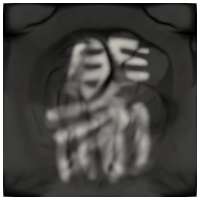}
\includegraphics[width=0.16\textwidth]{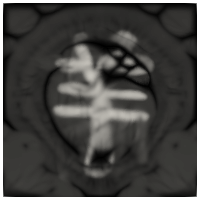}
\includegraphics[width=0.16\textwidth]{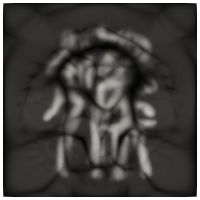}
\includegraphics[width=0.16\textwidth]{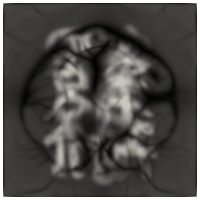}
\includegraphics[width=0.16\textwidth]{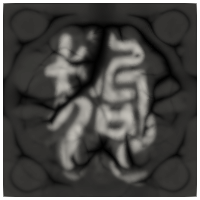}
\includegraphics[width=0.16\textwidth]{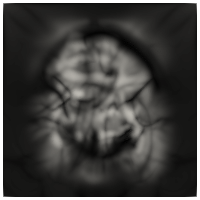}
}
\subfigure
{
\includegraphics[width=0.16\textwidth]{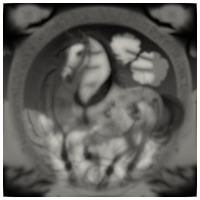}
\includegraphics[width=0.16\textwidth]{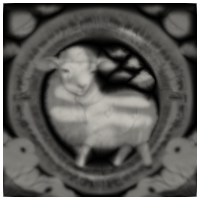}
\includegraphics[width=0.16\textwidth]{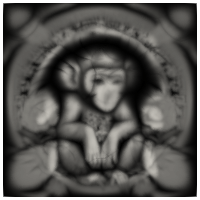}
\includegraphics[width=0.16\textwidth]{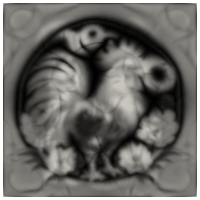}
\includegraphics[width=0.16\textwidth]{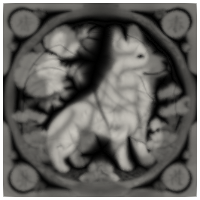}
\includegraphics[width=0.16\textwidth]{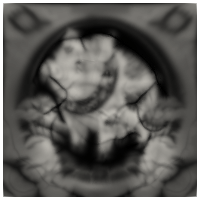}
}
\subfigure
{
\includegraphics[width=0.16\textwidth]{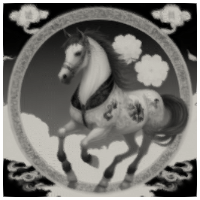}
\includegraphics[width=0.16\textwidth]{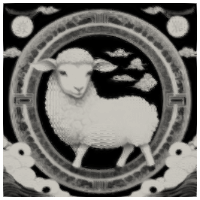}
\includegraphics[width=0.16\textwidth]{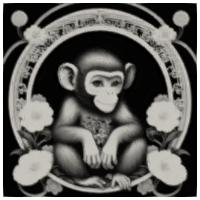}
\includegraphics[width=0.16\textwidth]{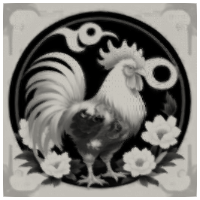}
\includegraphics[width=0.16\textwidth]{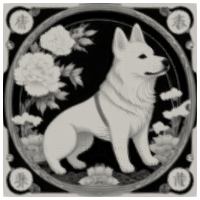}
\includegraphics[width=0.16\textwidth]{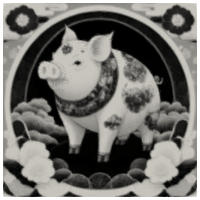}
}
\caption{Example \ref{ex6}. 
Snapshots of last 6 densities $\rho_7$ to $\rho_{12}$ (left to right) at times $t=0.0, 0.2, 0.5, 0.8, 1.0$ (top to bottom).
}
\label{fig:den-ex61}
\end{figure}

\begin{figure}[H]
\centering
\subfigure
{
\includegraphics[width=0.16\textwidth]{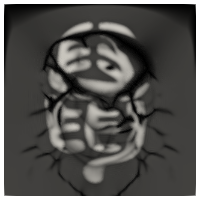}
\includegraphics[width=0.16\textwidth]{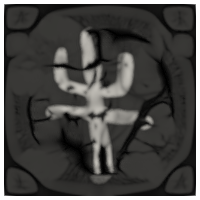}
\includegraphics[width=0.16\textwidth]{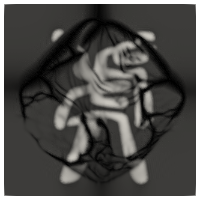}
\includegraphics[width=0.16\textwidth]{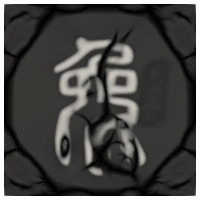}
\includegraphics[width=0.16\textwidth]{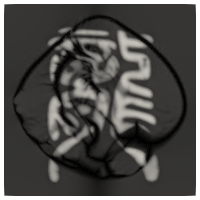}
\includegraphics[width=0.16\textwidth]{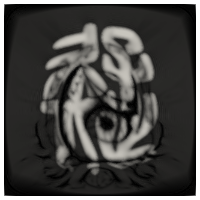}}
\subfigure
{
\includegraphics[width=0.16\textwidth]{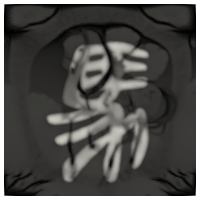}
\includegraphics[width=0.16\textwidth]{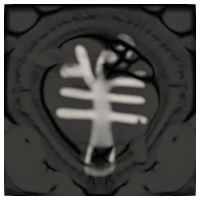}
\includegraphics[width=0.16\textwidth]{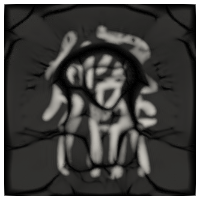}
\includegraphics[width=0.16\textwidth]{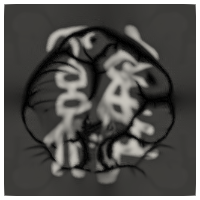}
\includegraphics[width=0.16\textwidth]{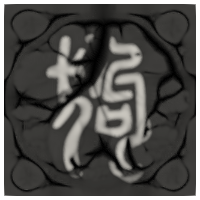}
\includegraphics[width=0.16\textwidth]{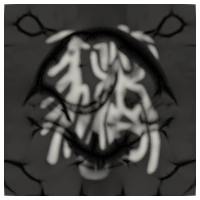}
}
\caption{Example \ref{ex6}. (No reaction $V_{2,p} = 0$.) 
Snapshots of densities at $t=0.5$. From left to right, top to bottom: $\rho_1$ to $\rho_{12}$.
}
\label{fig:den-ex62}
\end{figure}

\section{Discussions}
\label{sec5}
In this paper, we model and compute a class of generalized optimal transport and mean field control problems for reaction-diffusion equations and systems. The control problems are constructed by general choices of transport and reaction mobility functions, such as $V_1$, $V_2$, and $V_3$, derived from entropy dissipation properties. We apply a high-order spatial-time finite element method to discretize the spatial-time domain and use the ALG2 algorithm to compute mean field control problems. Numerical examples in Section \ref{sec4}, including transporting two Gaussian distributions and a system of $12$ images, demonstrate the effectiveness of the proposed mean field control models and computations. 

In future work, we shall study general modeling, computation, and inverse mean field control problems for reaction-diffusion systems. The generalized nonlinear reaction functions often represent complex behaviors between different populations, exhibiting patterns in social dynamical systems. The mean field control problem over transportation and reactions provides new patterns in population behaviors observed from our numerical examples. The analytical study of these new patterns could be a future research direction. We also expect that the mean field control problem of reaction-diffusion systems has vast applications in pandemic control, computer vision, and image processing problems. The other important direction is the parallel and high-order computation of generalized optimal transport and mean field control problems on three-dimensional spatial domains.  

\bibliographystyle{siam}
\bibliography{ams}

\section*{Appendix}
This section gives detailed proofs of Propositions \ref{smfc2}--\ref{vmfcKKT}.

\begin{proof}[Proof of Proposition \ref{smfc2}]
Denote $\mathcal{E}(\rho)=\int_\Omega E(\rho(x))dx$. Thus $\frac{\delta}{\delta\rho}\mathcal{E}(\rho)=E'(\rho)$. 
By the change of variable formula, we obtain 
\begin{equation*}
   \bmm=\tilde\bmm-\beta V_1(\rho)\nabla E'(\rho),\quad {s}=\tilde s-\beta V_2(\rho)E'(\rho). 
\end{equation*}
The constraint set \eqref{smfcA2} satisfies 
\begin{equation*}
\begin{split}
  0=&\partial_t \rho+\nabla\cdot(\tilde\bmm-\beta V_1(\rho)\nabla E'(\rho))-(\tilde s-\beta V_2(\rho)E'(\rho))\\ 
=&\partial_t \rho+\nabla\cdot \bmm-{s}.   
  \end{split}
\end{equation*}
Moreover, the terms in objective functional \eqref{smfcA1} satisfy 
\begin{equation*}
\begin{split}
&\int_0^T\int_\Omega \Big[\frac{\|\tilde \bmm(t,x)\|^2}{ 2V_1(\rho(t,x))}+ \frac{|\tilde s(t,x)|^2}{2V_2(\rho(t,x))}\Big] dxdt\\
=&\int_0^T\int_\Omega \Big[\frac{\| \bmm+\beta V_1(\rho)\nabla E'(\rho)\|}{ 2V_1(\rho)}+ \frac{|s+\beta V_2(\rho)E'(\rho)|^2}{2V_2(\rho)}\Big] dxdt\\
=&\int_0^T\int_\Omega \Big[\frac{\| \bmm\|^2+\beta^2 V_1(\rho)^2\|\nabla E'(\rho)\|^2}{ 2V_1(\rho)}+ \frac{|s|^2+\beta^2 V_2(\rho)^2|E'(\rho)|^2}{2V_2(\rho)}\Big] dxdt\\
&+\beta\int_0^T\int_\Omega \Big[\bmm\cdot \nabla E'(\rho)-s\cdot E'(\rho)\Big]  dx dt\\
=&\int_0^T\int_\Omega \Big[\frac{\| \bmm\|^2}{ 2V_1(\rho)}+ \frac{|s|^2}{2V_2(\rho)}+\frac{\beta^2}{2}\|\nabla E'(\rho)\|^2V_1(\rho)+\frac{\beta^2}{2}|E'(\rho)|^2V_2(\rho)\Big] dxdt\\
&+\beta\int_0^T\int_\Omega E'(\rho)\Big[-\nabla\cdot \bmm+s\Big]  dx dt\\
=&\int_0^T\int_\Omega \Big[\frac{\| \bmm\|^2}{ 2V_1(\rho)}+ \frac{|s|^2}{2V_2(\rho)}+\frac{\beta^2}{2}\|\nabla E'(\rho)\|^2V_1(\rho)+\frac{\beta^2}{2}|E'(\rho)|^2V_2(\rho)\Big] dxdt\\
&+\beta\int_0^T\int_\Omega \Big[E'(\rho)\partial_t \rho\Big]  dx dt\\
=&\int_0^T\int_\Omega \Big[\frac{\| \bmm\|^2}{ 2V_1(\rho)}+ \frac{|s|^2}{2V_2(\rho)}+\frac{\beta^2}{2}\|\nabla E'(\rho)\|^2V_1(\rho)+\frac{\beta^2}{2}|E'(\rho)|^2V_2(\rho)\Big] dxdt\\
&+\beta\int_\Omega \Big[E(\rho(T,\cdot))-E(\rho^0)\Big]dx, 
\end{split}
\end{equation*}
where the last equality follows the fact that $\int_0^TE'(\rho)\partial_t \rho dt=\int_0^T\partial_t E(\rho) dt=E(\rho(T,\cdot))-E(\rho^0)$. 
Thus we derive the variational problem \eqref{smfc2A} using the definition of $V_3$ in \eqref{v3-form}. 
\end{proof}

\begin{proof}[Proof of Proposition \ref{smfcKKT}]
Denote the Lagrange multiplier of problem \eqref{smfc2A} as $\phi\colon [0, T]\times\Omega\rightarrow\mathbb{R}$. Consider the following saddle point problem:
\begin{equation*}
\begin{split}
\inf_{\bmm, s, \rho, \rho_T}\sup_\phi \quad \mathcal{L}(\bmm, s, \rho, \rho_T, \phi),
\end{split}
\end{equation*}
where
\begin{equation*}
\begin{split}
\mathcal{L}(\bmm, s,\rho, \phi)=& \int_0^T \int_\Omega \Big[ \frac{\|\bmm\|^2}{2V_1(\rho)}+\frac{|s|^2}{2V_2(\rho)}+\phi\Big(\partial_t \rho+ \nabla\cdot \bmm-s\Big)\Big] dxdt\\
&+\int_{0}^T\Big[\frac{\beta^2}{2}\mathcal{I}(\rho)-\mathcal{F}(\rho)\Big] dt
+\mathcal{G}(\rho_T)+\beta\mathcal{E}(\rho_T).
\end{split}
\end{equation*}
Assume $\rho>0$. By solving the saddle point problem of $\mathcal{L}$, we obtain
\begin{equation*}
\left\{\begin{split}
&\frac{\delta}{\delta \bmm}\mathcal{L}=0,\\
& \frac{\delta}{\delta s}\mathcal{L}=0,\\
 &\frac{\delta}{\delta \rho}\mathcal{L}=0,\\
 &\frac{\delta}{\delta \phi}\mathcal{L}=0,\\
  &\frac{\delta}{\delta \rho_T}\mathcal{L}=0,\\
\end{split}\right.\quad\Rightarrow\quad\left\{\begin{split}
&\frac{\bmm}{V_1}=\nabla \phi,\\
& \frac{s}{V_2}=\phi,\\
 &-\frac{1}{2}\frac{\|\bmm\|^2}{V_1^2}V_1'-\frac{1}{2}\frac{|s|^2}{V_2^2}V_2'+\frac{\delta}{\delta \rho}\Big[\frac{\beta^2}{2}\mathcal{I}(\rho)-\mathcal{F}(\rho)\Big]-\partial_t\phi=0,\\
 &\partial_t\rho+\nabla\cdot \bmm-s=0,\\
 & \phi_T+\frac{\delta}{\delta\rho_T}\Big(\mathcal{G}(\rho_T)+\beta\mathcal{E}(\rho_T)\Big)=0. 
\end{split}\right.
\end{equation*}
This finishes the proof. 
\end{proof}

\begin{proof}[Proof of Proposition \ref{vmfc2}]
The proof is similar to the scalar case. Denote $\mathcal{E}_i(\rho_i)=\int_\Omega E_i(\rho_i(x))dx$. Thus $\frac{\delta}{\delta\rho_i}\mathcal{E}_i(\rho_i)=E_i'(\rho_i)$. 
By the change of variable formula, we obtain 
\begin{equation*}
   \bmm_i=\tilde\bmm_i-\beta V_{1,i}(\rho_i)\nabla E'_i(\rho_i),\quad {s}_p=\tilde s_p-\beta V_{2,p}(\bmr)\sum_{j=1}^M\gamma_{j,p}E'_j(\rho_j). 
\end{equation*}
The constraint set \eqref{vmfc2A2} satisfies 
\begin{equation*}
\begin{split}
  0=&\partial_t \rho_i+\nabla\cdot(\tilde\bmm_i-\beta V_{1,i}(\rho_i)\nabla E'_i(\rho_i))-(\sum_{p=1}^R\gamma_{i,p}\tilde s_p-\beta\sum_{p=1}^R\gamma_{i,p}V_{2,p}(\bmr)\sum_{j=1}^M\gamma_{j,p}E'_j(\rho_j))\\ 
=&\partial_t \rho_i+\nabla\cdot \bmm_i-\sum_{p=1}^R\gamma_{i,p}{s}_p.   
  \end{split}
\end{equation*}
Moreover, the terms in objective functional \eqref{vmfc2A1} satisfies 
\begin{equation*}
\begin{split}
&\int_0^T\int_\Omega \Big[\sum_{i=1}^M\frac{\|\tilde \bmm_i(t,x)\|^2}{ 2V_{1,i}(\rho_i(t,x))}+ \sum_{p=1}^R\frac{|\tilde s_p(t,x)|^2}{2V_{2,p}(\bmr(t,x))}\Big] dxdt\\
=&\int_0^T\int_\Omega \Big[\sum_{i=1}^M\frac{\| \bmm_i+\beta V_{1,i}(\rho_i)\nabla E'_i(\rho_i)\|}{ 2V_{1,i}(\rho_i)}+ \sum_{p=1}^R\frac{|s_p+\beta \sum_{j=1}^M \gamma_{j,p}V_{2,p}(\bmr)E'_j(\rho_j)|^2}{2V_{2,p}(\bmr)}\Big] dxdt\\
=&\int_0^T\int_\Omega \Big[\sum_{i=1}^M\frac{\| \bmm_i\|^2+\beta^2 V_{1,i}(\rho_i)^2\|\nabla E'_i(\rho_i)\|^2}{ 2V_{1,i}(\rho_i)}+ \sum_{p=1}^R\frac{|s_p|^2+\beta^2 |\sum_{j=1}^M\gamma_{j,p}V_{2,p}(\rho)E'_j(\rho_j)|^2}{2V_{2,p}(\bmr)}\Big] dxdt\\
&+\beta\int_0^T\int_\Omega \Big[\sum_{i=1}^M\bmm_i\cdot \nabla E'_i(\rho_i)-\sum_{p=1}^R s_p\cdot\sum_{j=1}^M\gamma_{j,p}V_{2,p}(\bmr)E'_j(\rho_j)\Big]  dx dt. 
\end{split}
\end{equation*}
We only need to show that 
\begin{equation*}
    \int_0^T\int_\Omega \Big[\sum_{i=1}^M\bmm_i\cdot \nabla E'_i(\rho_i)-\sum_{p=1}^R s_p\cdot\sum_{j=1}^M\gamma_{j,p}E'_j(\rho_j)\Big]  dx dt=\int_\Omega \sum_{i=1}^M\Big[E_i(\rho_i(T,\cdot))-E_i(\rho_i^0)\Big]dx. 
\end{equation*}
This is true from the following fact:
\begin{equation*}
\begin{split}
    &\int_0^T\int_\Omega \Big[\sum_{i=1}^M\bmm_i\cdot \nabla E'_i(\rho_i)-\sum_{p=1}^R s_p\cdot\sum_{j=1}^M\gamma_{j,p}E'_j(\rho_j)\Big]  dx dt\\
    =&\int_0^T \int_\Omega \sum_{i=1}^M  E'_i(\rho_i)\cdot \Big[-\nabla\cdot \bmm_i+\sum_{p=1}^R\gamma_{i,p}s_p\Big] dx dt\\ 
    =&\int_\Omega \int_0^T \sum_{i=1}^M E'_i(\rho_i)\cdot\partial_t\rho_i dt dx\\
=&
\int_\Omega \sum_{i=1}^M\Big[E_i(\rho_i(T,\cdot))-E_i(\rho_i^0)\Big]dx,
\end{split}
\end{equation*}
where the last equality uses the integration by parts in the time variable. This finishes the proof. 
\end{proof}

\begin{proof}[Proof of Proposition \ref{vmfcKKT}]
The proof is similar to the scalar case. We derive the minimizer system for variational problem \eqref{vmfc2A}. Denote $\phi_i\colon [0, T]\times \Omega\rightarrow \mathbb{R}$ as the Lagrange multiplier of constraint \eqref{vmfc2A1}, for $i=1,2,\cdots, M$. Write $\phi=(\phi_i)_{i=1}^M$. Consider the following saddle point problem 
\begin{equation*}
\inf_{\mm,s,\rho, \rho_T}\sup_{\phi}\quad \mathcal{L}_1(\mm, s, \bmr, \bmr_T, \phi),
\end{equation*}
where 
\begin{equation*}
\begin{split}
    &\mathcal{L}_{1}(\mm, s, \bmr, \bmr_T, \phi)\\
    =&\int_0^T\int_\Omega \Big[\frac{1}{2}\sum_{i=1}^M
    \frac{\|\bmm_i\|^2}{V_{1,i}(\rho_i)}
    +\frac{1}{2}\sum_{p=1}^R\frac{|\s_{p}|^2}{V_{2,p}(\bmr)}+\frac{\beta^2}{2}\bm{\mathcal{I}}(\bmr)\Big]dx-\mathcal{F}(\bmr(t,\cdot))dt\\
    &+\mathcal{G}(\bmr_T)+\beta \Big(\mathcal{E}(\bmr_T)-\mathcal{E}(\bmr_0)\Big)\\
&+\sum_{i=1}^M\int_0^T\int_\Omega \phi_i\cdot\Big\{\partial_t \rho_i + \nabla\cdot \m_i-\sum_{p=1}^R\gamma_{i,p}\s_{p}\Big\} dxdt. 
    \end{split}
\end{equation*}
Assume $\rho_i>0$, $i=1,2,\cdots, M$. By solving the saddle point problem of $\mathcal{L}_1$, we obtain
\begin{equation*}
\left\{\begin{split}
&\frac{\delta}{\delta \bmm_{i}}\mathcal{L}_1=0,\\
& \frac{\delta}{\delta \s_{p}}\mathcal{L}_1=0,\\
 &\frac{\delta}{\delta \rho_i}\mathcal{L}_1=0,\\
 &\frac{\delta}{\delta \phi_i}\mathcal{L}_1=0,\\
  &\frac{\delta}{\delta \bmr_T}\mathcal{L}_1=0,
\end{split}\right.\quad\Rightarrow\quad\left\{\begin{split}
&{\bmm_i}=V_{1,i}(\rho_i)\nabla \phi_i,\\
& \s_{p}=V_{2,p}(\bmr)\sum_{j=1}^M \gamma_{j,p}\phi_p,\\
&-\frac{1}{2}\sum_{i=1}^M\frac{\|\bmm_{i}\|^2}{V_{1,i}(\rho_i)^2}V_{1,i}(\rho_i)'\\
&-\frac{1}{2}\sum_{p=1}^R\frac{|s_{p}|^2}{V_{2,p}(\bmr)^2}\frac{\partial}{\partial \rho_i}V_{2,p}(\bmr)\\
&+\frac{\partial}{\partial \rho_i}[\frac{\beta^2}{2}\bm{\mathcal{I}}(\bmr)-\mathcal{F}(\bmr)]-\partial_t \phi_i=0,\\
& \partial_t \rho_i + \nabla\cdot \bmm_i-\sum_{p=1}^R\gamma_{i,p}s_{p}=0,\\
&\phi_{T}+\frac{\delta}{\delta \bmr_{T}}\mathcal{G}(\bmr_T)+\beta \frac{\delta}{\delta\bmr_{T}}\mathcal{E}(\bmr_T)=0. 
\end{split}\right.
\end{equation*}
This finishes the proof. 
\end{proof}

\end{document}